\documentclass[12pt,a4paper]{article}

\usepackage[utf8]{inputenc}
\usepackage{amsmath}
\usepackage{amssymb}
\usepackage{mathdots}
\usepackage{mathabx}
\usepackage{mathtools}
\usepackage{tikz}
\usepackage{tikz-cd}
\usepackage[inline]{enumitem}
\usepackage{color}
\definecolor{linkcolor}{rgb}{0,0,0.6}
\usepackage{amsthm}
\usepackage{bbm}
\usepackage{multicol}

\tikzset{
  symbol/.style={
    draw=none,
    every to/.append style={
      edge node={node [sloped, allow upside down, auto=false]{$#1$}}}
  }
}

\DeclareUnicodeCharacter{2212}{\textendash}

\usepackage[english]{babel}
\usepackage{csquotes}

\usepackage[colorlinks=true,
pdfstartview=FitV,
linkcolor= linkcolor,
citecolor= linkcolor,
urlcolor= linkcolor,
hyperindex=true,
hyperfigures=false]
{hyperref}

\usepackage{orcidlink}

\title{The Bruhat-Tits stratification for basic unramified $\mathrm{GU}(1,n-1)$ Rapoport-Zink spaces at arbitrary parahoric level}
\author{Joseph Muller\footnote{NCTS, National Taiwan University, muller@ncts.ntu.edu.tw}\,\orcidlink{0000-0002-1546-0910}}
\date{}

\usepackage[backend=bibtex, style=alphabetic]{biblatex}
\begin{document}

\newtheorem{theo}{Theorem}[section]
\newtheorem{prop}[theo]{Proposition}
\newtheorem{lem}[theo]{Lemma}
\newtheorem{corol}[theo]{Corollary}
\newtheorem{conj}[theo]{Conjecture}
\newtheorem*{theo*}{Theorem}
\newtheorem*{lem*}{Lemma}
\newtheorem*{prop*}{Proposition}
\newtheorem*{corol*}{Corollary}

\theoremstyle{remark}
\newtheorem{rk}[theo]{Remark}
\newtheorem{rks}[theo]{Remarks}
\newtheorem{ex}[theo]{Example}

\theoremstyle{definition}
\newtheorem{defi}[theo]{Definition}
\newtheorem*{notation}{Notation}
\newtheorem*{notations}{Notations}


\maketitle

\begin{center}

\parbox{15cm}{\small
\textbf{Abstract} : \it In this paper, we describe a stratification on the reduced special fiber of the basic unramified unitary Rapoport-Zink space of signature $(1,n-1)$ and at arbitrary parahoric level. We prove the smoothness, irreducibility and compute the dimensions of the closed strata, which are isomorphic to the closure of certain fine Deligne-Lusztig varieties for a product of unitary and general linear groups. We also describe the incidence relations of the stratification by using Bruhat-Tits indices, which are related to the Bruhat-Tits building of an underlying $p$-adic unitary group.}

\vspace{0.5cm}
\end{center}

\tableofcontents

\vspace{1.5cm}

\section{Introduction} 
In \cite{cox15}, \cite{fullHN} and \cite{cox24}, Görtz, He and Nie identified and fully classified all the basic Rapoport-Zink spaces whose (reduced) special fiber can be, up to perfection, naturally stratified as a union of classical Deligne-Lusztig varieties. The Rapoport-Zink space (or rather the underlying group theoretic datum) is said to be \textit{fully Hodge-Newton decomposable} if such a decomposition exists, and \textit{of Coxeter type} if additionally all the Deligne-Lusztig varieties are actually Coxeter varieties, as defined by Lusztig in \cite{CoxOrbits}. The authors used a purely group theoretic approach relying on affine Deligne-Lusztig varieties. While they prove the existence of a stratification by classical Deligne-Lusztig varieties, the indexing poset of this stratification, denoted $\mathrm{Adm}(\{\mu\})\cap {}^K\widetilde{W}$ in \cite{fullHN}, is not explicitely described. In fact, the underlying combinatorics is expected to be quite complicated, especially outside of the Coxeter case. Nonetheless, it has been determined successfully on a case-by-case principle in a handful of cases throughout the years, following the pioneering approach of Vollaard and Wedhorn in \cite{vw1} and \cite{vw2} using vertex lattices. The resulting stratification is known as the \textit{Bruhat-Tits stratification} because of the relation between vertex lattices and the Bruhat-Tits building of the underlying $p$-adic group. Together with loc. cit., a long series of contributions by many authors (in chronological order of publication \cite{howardpappasgu}, \cite{RTW}, \cite{wu}, \cite{howardpappasspinor}, \cite{tatecycle}, \cite{wangquaternion}, \cite{fox}, \cite{oki}) unveiled the combinatorics of the stratification in many cases of Coxeter type. In cases which are fully Hodge-Newton decomposable but not of Coxeter type, a similar approach was also successful in
\cite{cho}, \cite{wangparahoric}, \cite{maxramified} and \cite{gspinbasiclocus}. Eventually, we point out that the special fiber of certain Rapoport-Zink spaces which are not fully Hodge-Newton decomposable has also been studied in \cite{FoxImai}, \cite{FoxImaiHoward}, \cite{shimada} and \cite{trentin}. While one can not expect a decomposition into a union of classical Deligne-Lusztig varieties anymore, it turns out that there still exists a stratification, whose strata can be Deligne-Lusztig varieties or fiber bundles over them, giving a glimpse of what might be a bigger picture going beyond the fully Hodge-Newton decomposable case. With the exception of \cite{wangparahoric}, all the papers cited above deal with Rapoport-Zink spaces at maximal parahoric level (or, in ramified settings, vertex stabilizer level).\\
In this paper, we consider the basic unramified unitary Rapoport-Zink space of signature $(1,n-1)$. The hyperspecial level has been studied by Vollaard and Wedhorn in \cite{vw1} and \cite{vw2}, and their results have been generalized by Cho to the case of any maximal parahoric level in \cite{cho}. We complete the picture by generalizing one step further to the case of arbitrary parahoric level. We define a notion of \textit{Bruhat-Tits index}, which is a certain kind of chains of vertex lattices subject to some constraints on their types, and use them as indexing set for our Bruhat-Tits stratification. We prove that the closed Bruhat-Tits strata (ie. the Zariski closure of the strata) are isomorphic to the closure in a finite partial flag variety of a fine Deligne-Lusztig variety which we explicitely determine. Using a result of He in \cite{he}, we compute the decomposition of these closures as a union of fine Deligne-Lusztig varieties of smaller dimensions. We prove smoothness and irreducibility of the closed Bruhat-Tits strata, and study their orbits under the action of the group of auto-quasi-isogenies of the underlying $p$-divisible group. In particular, we determine the dimensions and the number of orbits of the irreducible components of the special fiber of the Rapoport-Zink space.\\
Via $p$-adic uniformization, these results can be transported to the basic locus of the corresponding unitary Shimura variety of signature $(1,n-1)$ over an inert prime. While we do not spell this out in this paper, the interested reader may follow the lines of \cite{RZ}, \cite{vw2} and \cite{cho} with no difficulty. Explicit constructions of the Bruhat-Tits stratifications in the papers cited in the first paragraph have many applications of arithmetical nature, such as the Kudla-Rapoport conjecture, the Tate-conjecture for Shimura varieties, the arithmetic Gan-Gross-Prasad conjecture, level raising and level lowering problems, etc. We hope that our results may contribute to such applications as well.\\
Let us sum up our main results in more details.\\

Let $p>2$ be an odd prime number and let $n\geq 1$. Fix a tuple of integers $\mathbbm h := (h_1 < \ldots < h_{m})$ where $m \geq 1$, $0 \leq h_i \leq n$ for all $1\leq i \leq m$, and where all the $h_i$'s have the same parity. Such tuples determine the parahoric level of the corresponding Rapoport-Zink space. For instance, the hyperspecial level studied in \cite{vw1} and \cite{vw2} corresponds to $m=1$ and $h_1 = 0$ or $h_1 = n$ (both choices giving rise to isomorphic spaces). In \cite{cho}, the case $m = 1$ with no restriction on $h_1$, ie. the maximal parahoric case, has been considered.\\

\textbf{For notational convenience, in this Introduction we assume {\boldmath $m \geq 2$}, {\boldmath $h_1 = 0$} and {\boldmath $h_m = n$}.} As it turns out, many definitions in this paper require to distinguish whether $h_1 = 0$ or $h_m = n$, or not. To avoid lengthy discussions, we leave this aside here. Let $F$ be a $p$-adic field and $E/F$ an unramified quadratic extension. Let $\pi$ denote a common uniformizer, and let $\breve E$ denote the completion of the maximal unramified extension of $E$. The symbols $\mathcal O$ and $\kappa$ will be used to denote rings of integers and the residue fields of the corresponding local fields. Let $q := \#\kappa_F$. The Rapoport-Zink space $\mathcal N_{E/F}^{\mathbbm h}$ is a formal scheme over $\mathrm{Spec}(\mathcal O_{\breve E})$ which is locally formally of finite type and regular. It classifies deformations by quasi-isogenies of a fixed framing object, that is a chain of mutually isogeneous supersingular strict formal $\mathcal O_F$-modules $(\mathbb X^{[h_i]})_{1\leq i \leq m}$ over $\mathbb F_{q^2}$, equipped with an $\mathcal O_E$-action of signature $(1,n-1)$ and a polarization of degree $q^{2h}$. By Dieudonné theory, we produce a certain hermitian space $(C,\{\cdot,\cdot\})$ of dimension $n$ over $F_{\mathbb F_{q^2}} := \mathrm{Frac}(W_{\mathcal O_F}(\mathbb F_{q^2}))$, whose discriminant is determined by the parity of the components of $\mathbbm h$. Note that $F_{\mathbb F_{q^2}} \simeq E$, but since both fields play different roles we single them with different notations using relative rings of Witt vectors. For an algebraically closed field $k$ containing $\kappa_{\breve E}$, we have (cf. Proposition \ref{PointsRZSpaceOverFields})
\begin{equation*}
\mathcal N_{E/F}^{\mathbbm h}(k) \simeq \left\{\begin{array}{c}
W_{\mathcal O_F}(k)\text{-lattices in }C_k\\
A_m \subset \ldots \subset A_1 \subset B_1 \subset \ldots \subset B_m
\end{array}
\,\middle|\, \forall 1\leq i \leq m, 
\begin{array}{c}
\pi A_i^{\vee} \overset{1}{\subset} B_i \subset A_i^{\vee},\\
\pi B_i^{\vee} \overset{1}{\subset} A_i \subset B_i^{\vee},\\
\pi B_i \subset A_i \overset{h_i}{\subset} B_i 
\end{array}
\right\}.
\end{equation*}
Here $C_k := C \otimes_{F_{\mathbb F_{q^2}}} \mathrm{Frac}(W_{\mathcal O_F}(k))$ and $\cdot^{\vee}$ denotes dual lattices with respect to the extension of $\{\cdot,\cdot\}$ to $C_k$. We note that $A_1 = B_1$ and $\pi B_m = A_m$ since we assumed $h_1 = 0$ and $h_m = n$ in this Introduction. \\
We point out that $(M^{\vee})^{\vee} = \tau(M)$ for any $W_{\mathcal O_F}(k)$-lattice $M \subset C_k$, where $\tau := \mathrm{id} \otimes \sigma^2$ and $\sigma$ is the lift to $\mathrm{Frac}(W_{\mathcal O_F}(k))$ of the $q$-power arithmetic Frobenius $x\mapsto x^q$ on $k$. Fix a point $(A_m \subset \ldots \subset B_m) \in \mathcal N_{E/F}^{\mathbbm h}(k)$. For $1 \leq i \leq m$, let $c_i,d_i \geq 1$ be the smallest integers such that the lattices 
\begin{align*}
T_{c_i}(A_i) & := A_i + \tau(A_i) + \ldots + \tau^{c_i-1}(A_i),\\
T_{d_i}(B_i) & := B_i + \tau(B_i) + \ldots + \tau^{d_i-1}(B_i),
\end{align*}
are $\tau$-stable. Let $\Lambda_{A_i},\Lambda_{B_i}$ be the $W_{\mathcal O_F}(\mathbb F_{q^2})$-lattices in $C$ such that $T_{c_i}(A_i) = (\Lambda_{A_i})_k$ and $T_{d_i}(B_i) = (\Lambda_{B_i})_k$. Eventually, for $i=0,1$ define the set of \textit{vertex lattices} 
\begin{equation*}
\mathcal L_i := \left\{W_{\mathcal O_F}(\mathbb F_{q^2})\text{-lattices } \Lambda \subset C \,\middle|\, \pi^{i+1}\Lambda^{\vee} \subset \Lambda \subset \pi^i\Lambda^{\vee}\right\}.
\end{equation*}
The \textit{type} of $\Lambda \in \mathcal L_i$ is the integer $t(\Lambda) := [\Lambda:\pi^{i+1}\Lambda^{\vee}]$. We have $0 \leq t(\Lambda) \leq n$, $t(\Lambda)$ is odd if $i=0$, and $t(\Lambda) \equiv n+1 \mod 2$ if $i=1$. We define the \textit{Bruhat-Tits type} of a point $(A_m \subset \ldots \subset B_m) \in \mathcal N_{E/F}^{\mathbbm h}(k)$ as the subset $I \subset \{1,\ldots,m-1\}$ such that 
\begin{itemize}
\item for $1 \leq i \leq m-1$, $i \in I \iff \Lambda_{B_i} \subset \pi\Lambda_{A_{i+1}}^{\vee}$.
\end{itemize}
We note that the inclusion $\Lambda_{B_i} \subset \pi\Lambda_{A_{i+1}}^{\vee}$ implies that $\Lambda_{B_i} \in \mathcal L_0$ and $\Lambda_{A_{i+1}} \in \mathcal L_1$ (cf. Proposition \ref{ConditionStar}), but the converse does not hold. If $h_1 > 0$ or $h_m < n$, the definition of Bruhat-Tits type is slightly different, cf. Definition \ref{BTType}. The starting point of our analysis is the following ``crucial Lemma'' which, in its general version, is a generalization of \cite{vw1} Lemma 2.1 and of \cite{cho} Lemma 2.7 (cf Lemma \ref{CrucialLemma}).

\begin{lem*}
The Bruhat-Tits type $I$ of any point $(A_m \subset \ldots \subset B_m) \in \mathcal N_{E/F}^{\mathbbm h}(k)$ is not empty.
\end{lem*}

A pair $(I,\mathbf{\Lambda})$ is called a \textit{Bruhat-Tits index} (cf. Definition \ref{DefinitionBTIndex}) if $I$ is a non-empty subset of $\{1,\ldots ,m-1\}$ and $\mathbf{\Lambda}$ is a collection of vertex lattices $\Lambda_0^i$ and $\Lambda_1^i$ for all $i\in I$, such that 
\begin{itemize}
\item $\forall i \in I, \Lambda_0^i \in \mathcal L_0^{\geq h_i+1}$ and $\Lambda_1^{i} \in \mathcal L_1^{\geq n-h_{i+1}+1}$,
\item if $1 \leq i_1 < \ldots < i_s \leq m-1$ denote the elements of $I$, we have
\begin{equation*}
\Lambda_0^{i_1} \subset \pi\Lambda_1^{i_1\vee} \subset \Lambda_0^{i_2} \subset \ldots \subset \pi \Lambda_1^{i_s-1\vee} \subset \Lambda_0^{i_s} \subset \pi\Lambda_1^{i_s\vee}.
\end{equation*}
\end{itemize}
Here, for any $0 \leq x \leq n$ we write $\mathcal L_i^{\geq x}$ for the subset of $\mathcal L_i$ consisting of those vertex lattices with $t(\Lambda) \geq x$. To any Bruhat-Tits index $(I,\mathbf{\Lambda})$, we attach a subset $\mathcal N_{I,\mathbf{\Lambda}}^{\mathbbm h}(k) \subset \mathcal N_{E/F}^{\mathbbm h}(k)$ (cf. Definition \ref{DefinitionBTStrata}) consisting of all the points $(A_m \subset \ldots \subset B_m)$ such that 

\begin{center}
\begin{tikzcd}[column sep=small, row sep=small]
\pi(\Lambda_0^{i_1})_k^{\vee} \arrow[r,symbol=\subset] & \pi B_{i_1}^{\vee} \arrow[d,symbol=\subset, outer sep=2pt,"1"] \arrow[r,symbol=\subset] & \ldots \arrow[r,symbol=\subset] & \pi B_1^{\vee} \arrow[d,symbol=\subset, outer sep=2pt,"1"] \arrow[r,symbol=\subset] & \pi A_1^{\vee} \arrow[d,symbol=\subset, outer sep=2pt,"1"] \arrow[r,symbol=\subset] & \ldots \arrow[r,symbol=\subset] & \pi A_{i_1}^{\vee} \arrow[d,symbol=\subset, outer sep=2pt,"1"] & {} \\
{} & A_{i_1} \arrow[r,symbol=\subset] & \ldots \arrow[r,symbol=\subset] & A_1 \arrow[r,symbol=\subset] & B_1 \arrow[r,symbol=\subset] & \ldots \arrow[r,symbol=\subset] & B_{i_1}  \arrow[r,symbol=\subset] & (\Lambda_0^{i_1})_k
\end{tikzcd}
\end{center} 

\noindent and

\begin{center}
\begin{tikzcd}[column sep=small, row sep=small]
\pi^2(\Lambda_1^{i_s})_k^{\vee} \arrow[r,symbol=\subset] & \pi^2 A_{i_s+1}^{\vee} \arrow[d,symbol=\subset, outer sep=2pt,"1"] \arrow[r,symbol=\subset] & \ldots \arrow[r,symbol=\subset] & \pi^2 A_m^{\vee} \arrow[d,symbol=\subset, outer sep=2pt,"1"] \arrow[r,symbol=\subset] & \pi B_m^{\vee} \arrow[d,symbol=\subset, outer sep=2pt,"1"] \arrow[r,symbol=\subset] & \ldots \arrow[r,symbol=\subset] & \pi B_{i_s+1}^{\vee} \arrow[d,symbol=\subset, outer sep=2pt,"1"] & {} \\
{} & \pi B_{i_s+1} \arrow[r,symbol=\subset] & \ldots \arrow[r,symbol=\subset] & \pi B_m \arrow[r,symbol=\subset] & A_m \arrow[r,symbol=\subset] & \ldots \arrow[r,symbol=\subset] & A_{i_s+1}  \arrow[r,symbol=\subset] & (\Lambda_1^{i_s})_k
\end{tikzcd}
\end{center}

\noindent and for all $1 \leq j \leq s-1$,\\

\hspace*{-3.2cm}
\begin{tikzcd}[column sep=small, row sep=small]
\pi^2(\Lambda_1^{i_j})_k^{\vee} \arrow[r,symbol=\subset] & \pi^2 A_{i_j+1}^{\vee} \arrow[d,symbol=\subset, outer sep=2pt,"1"] \arrow[r,symbol=\subset] & \ldots \arrow[r,symbol=\subset] & \pi^2 A_{i_{j+1}}^{\vee} \arrow[d,symbol=\subset, outer sep=2pt,"1"] & {} \\
{} & \pi B_{i_j+1} \arrow[r,symbol=\subset] & \ldots \arrow[r,symbol=\subset] & \pi B_{i_{j+1}} \arrow[r,symbol=\subset] & \pi(\Lambda_0^{i_{j+1}})_k
\end{tikzcd}
\begin{tikzcd}[column sep=small, row sep=small]
\pi(\Lambda_0^{i_{j+1}})_k^{\vee} \arrow[r,symbol=\subset] & \pi B_{i_{j+1}}^{\vee} \arrow[d,symbol=\subset, outer sep=2pt,"1"] \arrow[r,symbol=\subset] & \ldots \arrow[r,symbol=\subset] & \pi B_{i_j+1}^{\vee} \arrow[d,symbol=\subset, outer sep=2pt,"1"] & {} \\
{} & A_{i_{j+1}} \arrow[r,symbol=\subset] & \ldots \arrow[r,symbol=\subset] & A_{i_j+1} \arrow[r,symbol=\subset] & (\Lambda_1^{i_j})_k
\end{tikzcd}.\\

It turns out that $\mathcal N_{I,\mathbf{\Lambda}}^{\mathbbm h}(k)$ is the set of $k$-rational points of some closed subscheme $\mathcal N_{I,\mathbf{\Lambda}}^{\mathbbm h}$ of the reduced special fiber $\mathcal N_{E/F,\mathrm{red}}^{\mathbbm h}$ of the Rapoport-Zink space, cf. Propositions \ref{ClosedBTStrataAsSubschemes} and \ref{EquivalenceClosedBTStrata}. Moreover, by ``reducing the lattices mod $p$'', we prove that $\mathcal N_{I,\mathbf{\Lambda}}^{\mathbbm h}$ is isomorphic to the closure $X_{I,\mathbf{\Lambda}}^{\mathbbm h}$ in the corresponding partial flag variety, of an explicitely determined fine Deligne-Lusztig variety for the product of the finite groups $\mathrm{U}(\Lambda_0^{i_1}/\pi\Lambda_0^{i_1\vee})$, $\mathrm U(\Lambda_1^{i_s}/\pi^2\Lambda_1^{i_s\vee})$ and $\mathrm{GL}(\Lambda_0^{i_{j+1}}/\pi\Lambda_1^{i_j\vee})$ for all $1 \leq j \leq s-1$, cf. Section \ref{Section3.3} and Theorem \ref{IsomorphismWithDLVariety}. It follows that the subschemes  $\mathcal N_{I,\mathbf{\Lambda}}^{\mathbbm h}$ are smooth, projective and irreducible. \\
We define a notion of inclusion $(I,\mathbf{\Lambda}) \subset (I',\mathbf{\Lambda'})$ and intersection $(I,\mathbf{\Lambda})\cap (I',\mathbf{\Lambda'})$ on Bruhat-Tits indices, cf. Definitions \ref{InclusionBTIndices} and \ref{IntersectionBTIndices}, and we prove that they describe the incidence relations of the subschemes $\mathcal N_{I,\mathbf{\Lambda}}^{\mathbbm h}$, cf. Proposition \ref{CombinatoricsBTStratification}. We define 
\begin{equation*}
\mathcal N_{I,\mathbf{\Lambda}}^{\mathbbm h,0} := \mathcal N_{I,\mathbf{\Lambda}}^{\mathbbm h} \setminus \bigcup_{(I',\mathbf{\Lambda'})\subsetneq (I,\mathbf{\Lambda})} \mathcal N_{I',\mathbf{\Lambda'}}^{\mathbbm h}. 
\end{equation*} 
Then $\mathcal N_{I,\mathbf{\Lambda}}^{\mathbbm h,0}$ defines a locally closed subscheme of $\mathcal N_{E/F,\mathrm{red}}^{\mathbbm h}$, whose rational points are described as follows (cf. Lemma \ref{RationalPointsBTStrata}).

\begin{lem*}
Let $k$ be an algebraically closed field containing $\kappa_{\breve E}$ and let $(A_m \subset \ldots \subset B_m) \in \mathcal N_{E/F}^{\mathbbm h}(k)$. The following statements are equivalent. 
\begin{enumerate}
\item $(A_m\subset \ldots \subset B_m) \in \mathcal N_{I,\mathbf{\Lambda}}^{\mathbbm h,0}(k)$,
\item $I$ is the Bruhat-Tits type of $(A_m\subset \ldots \subset B_m)$, and for all $i\in I$, we have $\Lambda_0^{i} = \Lambda_{B_i}$ and $\Lambda_1^{i} = \Lambda_{A_{i+1}}$. 
\end{enumerate}
\end{lem*}

Using the main result of \cite{he}, we compute the decomposition of $X_{I,\mathbf{\Lambda}}^{\mathbbm h}$ as a disjoint union of fine Deligne-Lusztig varieties in Section \ref{Section3.3}. With this decomposition, we identify a certain open subvariety $X_{I,\mathbf{\Lambda}}^{\mathbbm h,0}$, and prove that the isomorphism between $\mathcal N_{I,\mathbf{\Lambda}}^{\mathbbm h}$ and $X_{I,\mathbf{\Lambda}}^{\mathbbm h}$ induces an isomorphism between $\mathcal N_{I,\mathbf{\Lambda}}^{\mathbbm h,0}$ and $X_{I,\mathbf{\Lambda}}^{\mathbbm h,0}$. In particular, we have the following (cf. Corollary \ref{BTStrataNotEmpty}).

\begin{corol*}
Let $k$ be an algebraically closed field containing $\kappa_{\breve E}$. For every Bruhat-Tits type $(I,\mathbf{\Lambda})$, there exists a point $(A_m \subset \ldots \subset B_m) \in \mathcal N_{I,\mathbf{\Lambda}}^{\mathbbm h,0}(k)$. In particular, $\mathcal N_{I,\mathbf{\Lambda}}^{\mathbbm h,0} \not = \emptyset$.
\end{corol*}

The locally closed subschemes $\mathcal N_{I,\mathbf{\Lambda}}^{\mathbbm h,0}$ are called the \textit{Bruhat-Tits strata}, and their closures $\mathcal N_{I,\mathbf{\Lambda}}^{\mathbbm h}$ are called the \textit{closed Bruhat-Tits strata}. See Remark \ref{ComparisonWithCho} for a comparison with the Bruhat-Tits strata defined in \cite{cho}. \\
As it turns out, the Bruhat-Tits strata $\mathcal N_{I,\mathbf{\Lambda}}^{\mathbbm h,0}$ are in general not isomorphic to a fine Deligne-Lusztig variety, but rather a disjoint union of them. This disjoint union consists of a single variety for all Bruhat-Tits indices if and only if the Rapoport-Zink space is of Coxeter type, ie. ($m = 1$ and $h_1 = 0$ or $n$) or ($m=2$, $n$ is even and $\mathbbm h = (0,n)$). \\
The strata cover the whole special fiber, since for any algebraically closed field $k$ containing $\kappa_{\breve E}$, we have 
\begin{equation*}
\mathcal N_{E/F}^{\mathbbm h}(k) = \bigsqcup_{(I,\mathbf{\Lambda})} \mathcal N_{I,\mathbf{\Lambda}}^{\mathbbm h,0}(k),
\end{equation*}
where $(I,\mathbf{\Lambda})$ runs over all the Bruhat-Tits indices. The Rapoport-Zink space carries a natural action of the group of auto-quasi-isogenies $J(E) := \mathrm{Aut}(\mathbb X^{[h_i]})$ of any member of the framing object (the choice of $i$ does not matter), which can be identified with the group $\mathrm{GU}^{0}(C,\{\cdot,\cdot\})$ of unitary similitudes of $C$ whose factor of similitude is a unit. The group also acts on vertex lattices preserving types and inclusions, thus preserving Bruhat-Tits indices. Any $g \in J(E)$ induces an isomorphism 
\begin{equation*}
g:\mathcal N_{I,\mathbf{\Lambda}}^{\mathbbm h} \xrightarrow{\sim} \mathcal N_{I,g(\mathbf{\Lambda})}^{\mathbbm h}.
\end{equation*}
By construction, the collection of vertex lattices $\mathbf{\Lambda}$ can be identified with a facet in the Bruhat-Tits building of the unitary group of $C$. It follows that the stabilizer of $\mathbf{\Lambda}$ in $J(E)$ is a parahoric subgroup, and its action on $\mathcal N_{I,\mathbf{\Lambda}}^{\mathbbm h}$ factors through its reductive quotient which is isomorphic, up to the similitude factor, to the product of the finite groups $\mathrm{U}(\Lambda_0^{i_1}/\pi\Lambda_0^{i_1\vee})$, $\mathrm U(\Lambda_1^{i_s}/\pi^2\Lambda_1^{i_s\vee})$ and $\mathrm{GL}(\Lambda_0^{i_{j+1}}/\pi\Lambda_1^{i_j\vee})$ for all $1 \leq j \leq s-1$. The isomorphism between $\mathcal N_{I,\mathbf{\Lambda}}^{\mathbbm h}$ and $X_{I,\mathbf{\Lambda}}^{\mathbbm h}$ is equivariant under this action.\\
Eventually, we investigate the irreducible components of $\mathcal N_{E/F,\mathrm{red}}^{\mathbbm h}$. These correspond to the maximal closed Bruhat-Tits strata. Given a Bruhat-Tits index $(I,\mathbf{\Lambda})$, for $i \in I$ define $\mathbf{\Lambda}^i := \{\Lambda_0^i,\Lambda_1^i\}$. Then we have (cf. Lemma \ref{BruhatTitsStrataAsIntersection})
\begin{equation*}
\mathcal N_{I,\mathbf{\Lambda}}^{\mathbbm h} = \bigcap_{i \in I} \mathcal N_{\{i\},\mathbf{\Lambda}^{i}}^{\mathbbm h}.
\end{equation*}
We deduce the following description of the irreducible components (cf. Corollary \ref{IrreducibleComponents}).

\begin{corol*}
The irreducible components of $\mathcal N_{E/F,\mathrm{red}}^{\mathbbm h}$ consists of all the closed Bruhat-Tits strata of the form $\mathcal N_{\{i\},\{\Lambda_0^i,\Lambda_1^i\}}^{\mathbbm h}$ where $1 \leq i \leq m-1$, $\Lambda_0^i \in \mathcal L_0^{\geq h_i+1}$, $\Lambda_1^i \in \mathcal L_1^{\geq n-h_{i+1}+1}$ and $\Lambda_0^i = \pi\Lambda_1^{i\vee}$. Moreover, we have
\begin{equation*}
\dim\left(\mathcal N_{\{i\},\{\Lambda_0^i,\Lambda_1^i\}}^{\mathbbm h}\right) = n-1-\frac{h_{i+1}-h_i}{2}.
\end{equation*}
\end{corol*}

Moreover, we compute the number of $J(E)$-orbits of irreducible components (cf. Theorem \ref{NumberOrbitsIrreducibleComponents}).

\begin{theo*}
The number of $J(E)$-orbits of irreducible components in $\mathcal N_{E/F,\mathrm{red}}^{\mathbbm h}$ is $\frac{n}{2}$. More precisely, for a fixed $1 \leq i \leq m-1$, the irreducible components of the form $\mathcal N_{\{i\},\{\Lambda_0^i,\Lambda_1^i\}}^{\mathbbm h}$ contribute to $\frac{h_{i+1}-h_i}{2}$ orbits. 
\end{theo*}

We point out that in general, the number of orbits computed in Theorem \ref{NumberOrbitsIrreducibleComponents} depends only on $h_1$ and on $h_m$; in this Introduction, we specialized the statement to the case $h_1 = 0$ and $h_m = n$. Finally, in Section \ref{Examples} we illustrate the results stated above in the case $m=2$ and in the Iwahori case.

\section{The moduli space of formal $\mathcal O_F$-modules of signature $(1,n-1)$}
\subsection{Maximal parahoric level}

Let $p>2$ be an odd prime. Let $F$ be a finite extension of $\mathbb Q_p$, with ring of integers $\mathcal O_F$ and uniformizer $\pi \in \mathcal O_F$. Let $\kappa_F := \mathcal O_F/(\pi)$ denote the residue field of $F$, and let $q := \#\kappa_F$ so that $\kappa_F \simeq \mathbb F_q$. Let $E$ be a quadratic unramified extension of $F$, with ring of integer $\mathcal O_E$ and residue field $\kappa_E = \mathcal O_E/(\pi) \simeq \mathbb F_{q^2}$. Let $x\mapsto x^*$ denote the non-trivial element of $\mathrm{Gal}(E/F)$. The following definition was introduced in \cite{rz17} (see also \cite{Mih}).

\begin{defi} 
A \textit{formal $\mathcal O_F$-module} over a scheme $S$ over which $p$ is locally nilpotent, is a formal $p$-divisible group $X$ over $S$ together with an $\mathcal O_F$-action, ie. a ring morphism $i:\mathcal O_F \to \mathrm{End}(X)$.\\
Assume that $S$ is an $\mathcal O_F$-scheme. The formal $\mathcal O_F$-module $X$ is said to be \textit{strict} if the $\mathcal O_F$-action on $\mathrm{Lie}(X)$ induced by $i$ coincides with the natural action given by the structure morphism $\mathcal O_F \to \mathcal O_S$.
\end{defi}

For $0\leq h \leq n$, we fix a datum $\mathbb X^{[h]} = (\mathbb X,i_{\mathbb X},\lambda_{\mathbb X}^{[h]})$ where
\begin{itemize}
\item $\mathbb X$ is a strict formal $\mathcal O_F$-module over $\mathbb F_{q^2}$, supersingular and of relative $F$-height $2n$,
\item $i_{\mathbb X}:\mathcal O_E\to\mathrm{End}(\mathbb X)$ is an $\mathcal O_E$-action on $\mathbb X$ which is compatible with the natural $\mathcal O_F$-action,
\item $\lambda_{\mathbb X}^{[h]}:\mathbb X \to \mathbb X^{\vee}$ is an $\mathcal O_E$-linear polarization on $\mathbb X$ such that $\mathrm{Ker}(\lambda_{\mathbb X}^{[h]}) \subset \mathbb X[\pi]$ has order $q^{2h}$.
\end{itemize}

In the last item, $\mathbb X^{\vee}$ denotes the Serre dual of $\mathbb X$, which is equipped with the $\mathcal O_E$-action $i_{\mathbb X^{\vee}}(x) := i_{\mathbb X}(x^*)^{\vee}$. Lastly, one also requires that the $(1,n-1)$ signature condition is satisfied, ie. 
\begin{equation*}
\forall x\in\mathcal O_E, \mathrm{charpol}(i_{\mathbb X}(x)\,|\, \mathrm{Lie}(\mathbb X)) = (T-x)(T-x^*)^{n-1} \in W(\mathbb F_{q^2})[T].
\end{equation*}   
The existence and unicity of such a datum $\mathbb X^{[h]}$ is well-known, see \cite{cho} Remark 3.32. Let $\textbf{Nilp}$ denote the category of $\mathcal O_E$-schemes $S$ over which $\pi$ is locally nilpotent. For $S \in \textbf{Nilp}$, let $\mathcal N_{E/F}^{h}(S)$ denote the set of tuples $(X,i_X,\lambda_X,\rho_X)$ up to isomorphism, where 
\begin{itemize}
\item $X$ is a strict formal $\mathcal O_F$-module of relative $F$-height $2n$ over $S$,
\item $i_X$ is an $\mathcal O_E$-action on $X$ compatible with the natural $\mathcal O_F$-action, satisfying the $(1,n-1)$ signature condition 
\begin{equation*}
\forall x\in\mathcal O_E, \mathrm{charpol}(i_{X}(x)\,|\, \mathrm{Lie}(X)) = (T-x)(T-x^*)^{n-1} \in \mathcal O_S[T].
\end{equation*}   
\item $\lambda_X:X\to X^{\vee}$ is an $\mathcal O_E$-linear polarization on $X$,
\item $\rho_X:X_{\overline S} \to \mathbb X_{\overline S}$ is an $\mathcal O_E$-linear quasi-isogeny of height $0$, making the following diagram commute up to a unit scalar in $\mathcal O_F^{\times}$
\begin{equation*}
\begin{tikzcd}
X_{\overline S} \arrow[r,"(\lambda_{X})_{\overline S}"] \arrow[d,swap,"\rho_X"] & X_{\overline S}^{\vee} \\
\mathbb X_{\overline S} \arrow[r,"(\lambda_{\mathbb X})_{\overline S}"] & \mathbb X_{\overline S}^{\vee} \arrow[u,swap,"\rho_X^{\vee}"]
\end{tikzcd}
\end{equation*}
\end{itemize}
In the last item, $\overline S := S \times_{\mathcal O_F} \kappa_F$ and $X_{\overline S}$ and $\mathbb X_{\overline S}$ denote the base change to $\overline S$. An isomorphism $(X,i_X,\lambda_X,\rho_X) \xrightarrow{\sim} (X',i_{X'},\lambda_{X'},\rho_{X'})$ is an $\mathcal O_E$-linear isomorphism $\gamma:X\xrightarrow{\sim} X'$ such that $\rho_{X'} \circ \gamma_{\overline S} = \rho_X$ and $\gamma^{\vee}\circ\lambda_{X'}\circ\gamma$ differs from $\lambda_X$ locally on $S$ by a scalar in $\mathcal O_F^{\times}$. \\
Let $\breve{E}$ denote the completion of the maximal unramified extension of $E$. The following statement follows from \cite{Mih} and \cite{cho}.

\begin{theo}
The set-valued functor $\mathcal N_{E/F}^{h} \otimes \mathcal O_{\breve E}$ is representable by a formal scheme over $\mathrm{Spf}(\mathcal O_{\breve E})$ which is locally formally of finite type and regular.
\end{theo}

By abuse of notation, this formal scheme is also denoted by $\mathcal N_{E/F}^{h}$. Let $\kappa_{\breve E} = \overline{\mathbb F_{q^2}}$ denote the residue field of $\breve E$. In \cite{cho}, the author describes the Bruhat-Tits stratification on the reduced special fiber $\mathcal N_{E/F,\mathrm{red}}^{h} := (\mathcal N^{h}_{E/F}\otimes \kappa_{\breve E})_{\mathrm{red}}$. The remaining of this section is dedicated to recall their constructions. For notions related to relative Dieudonné theory, see \cite{KR14} Notations.\\ 

Let $k$ be an algebraically closed field containing $\kappa_E = \mathbb F_{q^2}$. For an $\mathcal O_F$-algebra $R$, we denote by $W_{\mathcal O_F}(R)$ the ring of relative Witt vectors of $R$. Note that if $R$ is a perfect field extension of $\mathbb F_q$, we have a natural isomorphism $W_{\mathcal O_F}(R) \simeq \mathcal O_F \otimes_{\mathcal O_{F^u}} W(R)$, where $W(R)$ is the absolute ring of Witt vectors of $R$ and $\mathcal O_{F^u} = W(\mathbb F_q)$ is the ring of integers of the maximal unramified extension $F^u$ of $\mathbb Q_p$ which is contained in $F$. In this case, we write $F_R$ for the fraction field of $W_{\mathcal O_F}(R)$. Thus $F_R$ is a field extension of $F$ with residue field $R$. Let $\varphi_0:E\xrightarrow{\sim} F_{\mathbb F_{q^2}}$ be a fixed isomorphism which extends the natural embedding of $F$ into the right-hand side, and let $\varphi_1$ be given by $\varphi_1(x) := \varphi_0(x^*)$ for all $x\in E$.\\
Let $(\mathbb M,\mathcal V,\mathcal F)$ denote the relative Dieudonné module of $\mathbb X$, and let $\mathbb N := \mathbb M\otimes \mathbb Q$ denote the relative Dieudonné crystal. Then $\mathbb M$ is a free $W_{\mathcal O_F}(\mathbb F_{q^2})$-module of rank $2n$, and $\mathbb N$ is an $F_{\mathbb F_{q^2}}$-vector space of dimension $2n$. Furthermore the Verschiebung $\mathcal V$ and the Frobenius $\mathcal F$ are $\sigma$-linear operators on $\mathbb M$ satisfying $\mathcal V\mathcal F = \mathcal F\mathcal V = \pi$, where $\sigma\in \mathrm{Gal}(F_{\mathbb F_{q^2}}/F)$ is the non-trivial element. The $\mathcal O_E$-action $i_{\mathbb X}$ on $\mathbb X$ induces a structure of $\mathcal O_E \otimes_{\mathcal O_F} W_{\mathcal O_F}(\mathbb F_{q^2})$-module on $\mathbb M$. Via the decomposition 
\begin{align*}
\mathcal O_E \otimes_{\mathcal O_F} W_{\mathcal O_F}(\mathbb F_{q^2}) & \xrightarrow{\sim} W_{\mathcal O_F}(\mathbb F_{q^2}) \times W_{\mathcal O_F}(\mathbb F_{q^2}),\\
x\otimes a & \mapsto (\varphi_0(x)a,\varphi_1(x)a),
\end{align*}
we obtain a $\mathbb Z/2\mathbb Z$-grading $\mathbb M = \mathbb M_0 \oplus \mathbb M_1$ for which $\mathcal F$ and $\mathcal V$ are homogeneous operators of degree $1$. Here, each summand $\mathbb M_0$ and $\mathbb M_1$ is a free $W_{\mathcal O_F}(\mathbb F_{q^2})$-module of rank $n$. Likewise, we also have a $\mathbb Z/2\mathbb Z$-grading $\mathbb N = \mathbb N_0 \oplus \mathbb N_1$ where, for $i=0,1$, we have $\mathbb N_i = \mathbb M_i \otimes\mathbb Q$. Eventually, the polarization $\lambda_{\mathbb X}^{[h]}$ induces a nondegenerate $F_{\mathbb F_{q^2}}$-valued bilinear pairing $\langle \cdot,\cdot \rangle_{[h]}$ on $\mathbb N$ such that 
\begin{align*}
\forall v,w\in \mathbb N, \forall x\in E, & & \langle \mathcal Fv,w\rangle_{[h]} = \langle v,\mathcal Vw\rangle_{[h]}^{\sigma}, & & \langle i_{\mathbb X}(x)v,w\rangle_{[h]} = \langle v,i_{\mathbb X}(x^*)w\rangle_{[h]}.
\end{align*} 
In particular, for $i=0,1$ the subspace $\mathbb N_i$ is totally isotropic for $\langle\cdot,\cdot\rangle_{[h]}$.\\
Now, let $\mathbb M_k := \mathbb M \otimes_{W_{\mathcal O_F}(\mathbb F_{q^2})} W_{\mathcal O_F}(k)$ and $\mathbb N_k := \mathbb N \otimes_{F_{\mathbb F_{q^2}}} F_k$ denote the base changes to $k$. Note that $\mathcal F$ (resp. $\mathcal V$) is extended to a $\sigma$-linear (resp. $\sigma^{-1}$-linear) operator on $\mathbb M_k$, where $\sigma \in \mathrm{Gal}(F_k/F)$ now denotes the relative Frobenius automorphism. Note that we still have $\mathbb Z/2\mathbb Z$-gradings 
\begin{align*}
\mathbb M_k = \mathbb M_{k,0}\oplus \mathbb M_{k,1}, & & \mathbb N_k = \mathbb N_{k,0}\oplus \mathbb N_{k,1}.
\end{align*}
Given an $\mathcal O_E$-stable $W_{\mathcal O_F}(k)$-lattice $M = M_0 \oplus M_1 \subset \mathbb N_k$, we define the dual lattice $M^{\dagger}$ with respect to $\langle\cdot,\cdot\rangle_{[h]}$ by 
\begin{equation*}
M^{\dagger} := \{x \in \mathbb N_k \,|\, \langle x,M \rangle_{[h]} \subset W_{\mathcal O_F}(k)\}.
\end{equation*}
The dual lattice $M^{\dagger}$ is also stable under the action of $\mathcal O_E$, so that it decomposes as $M^{\dagger} = (M^{\dagger})_0 \oplus (M^{\dagger})_1$. In fact, we have $(M^{\dagger})_i = M_{i+1}^{\dagger}$ where $i\in \mathbb Z/2\mathbb Z$ and
\begin{equation*}
M_i^{\dagger} := \{x \in \mathbb N_{k,i+1} \,|\, \langle x,M_i\rangle_{[h]} \subset W_{\mathcal O_F}(k)\}.
\end{equation*}
Let $\tau := \pi\mathcal V^{-2}$. Then $\tau$ is a $\sigma^2$-linear operator on $\mathbb N_k$ whose slopes are all zero. Let $\mathbb N_{k,0}^{\tau}$ denote the subset of all $\tau$-fixed vectors in $\mathbb N_{k,0}$. Then $\mathbb N_{k,0}^{\tau}$ is an $F_{\mathbb F_{q^2}}$-vector space and we have $\mathbb N_{k,0} = \mathbb N_{k,0}^{\tau}\otimes_{F_{\mathbb F_{q^2}}}F_k$.

\begin{rk}
The space $\mathbb N_{k,0}^{\tau}$ does not depend on $k$, in the sense that if $k'/k$ is an extension of algebraically closed fields containing $\kappa_{\breve E}$, the isomorphism $\mathbb N_{k,0}\otimes_{F_k} F_{k'} \simeq \mathbb N_{k',0}$ identifies $\mathbb N_{k,0}^{\tau}$ with $\mathbb N_{k',0}^{\tau}$. This allows us to define $C := \mathbb N_{k,0}^{\tau}$ without ambiguity. 
\end{rk}

We define an $F_k$-valued pairing $\{\cdot,\cdot\}_{[h]}$ on $\mathbb N_{k,0}$ via the formula 
\begin{align*}
\forall v,w\in \mathbb N_{k,0},& & \{v,w\}_{[h]}:= \delta\langle v,\mathcal Fw\rangle_{[h]},
\end{align*} 
where $\delta \in W_{\mathcal O_F}(\mathbb F_{q^2})^{\times}$ is a fixed scalar such that $\delta^{\sigma} = -\delta$. Then $\{\cdot,\cdot\}_{[h]}$ is left linear, right $\sigma$-linear and non-degenerate. Furthermore it satisfies 
\begin{align*}
\forall v,w\in \mathbb N_{k,0}, & & \{v,w\}_{[h]} = \{w,\tau^{-1}(v)\}_{[h]}^{\sigma}, & & \{\tau(v),\tau(w)\}_{[h]} = \{v,w\}^{\sigma^2}_{[h]}.
\end{align*}
It follows that $\{\cdot,\cdot\}_{[h]}$ induces a non-degenerate $F_{\mathbb F_{q^2}}/F$-hermitian pairing on $C$. Given a $W_{\mathcal O_F}(k)$-lattice $A\subset \mathbb N_{k,0}$, we write 
\begin{equation*}
A^{\vee} := \{v\in \mathbb N_{k,0}\,|\,\{v,A\}_{[h]}\subset W_{\mathcal O_F}(k)\},
\end{equation*}
for the dual lattice of $A$. The following Lemma is proved in \cite{vw1}.

\begin{lem}
For every $W_{\mathcal O_F}(k)$-lattice $A\subset \mathbb N_{k,0}$, we have 
\begin{align*}
\tau(A^{\vee}) = \tau(A)^{\vee}, & & (A^{\vee})^{\vee} = \tau(A).
\end{align*}
\end{lem}

Given two lattices $A,B \subset \mathbb N_{k,0}$ and a nonnegative integer $x$, we write $A\overset{x}{\subset} B$ if $A\subset B$ and the module $B/A$ has length $x$. We also write $x = [B:A]$, and say that $x$ is the index of $A$ in $B$. The following statement is proved in \cite{cho} Propositions 2.2 and 2.4, and describes the $k$-rational points of $\mathcal N_{E/F}^h$ in terms of relative Dieudonné theory. 

\begin{theo}\label{RationalPointsMaximalParahoric}
There is a bijection 
\begin{align*}
\mathcal N_{E/F}^h(k) & \simeq \left\{ W_{\mathcal O_F}(k)\text{-lattices } M \subset \mathbb N_k \,\middle|\, \begin{array}{c}
M \text{ is } \mathcal O_E,\mathcal F,\mathcal V\text{-stable,}\\
M_0 \overset{h}{\subset} M_1^{\dagger} \overset{n-h}{\subset} \pi^{-1}M_0, \quad M_1 \overset{h}{\subset} M_0^{\dagger} \overset{n-h}{\subset} \pi^{-1}M_1, \\
\pi M_0 \overset{n-1}{\subset} \mathcal VM_1 \overset{1}{\subset} M_0, \quad \pi M_1 \overset{1}{\subset} \mathcal VM_0 \overset{n-1}{\subset} M_1.
\end{array} \right\},\\
& \simeq \left\{W_{\mathcal O_F}(k)\text{-lattices } A\overset{h}{\subset}B \subset \mathbb N_{k,0} \,\middle|\, \begin{array}{c}
\pi A^{\vee} \overset{1}{\subset} B \subset A^{\vee},\\
\pi B^{\vee} \overset{1}{\subset} A \subset B^{\vee},\\
\pi B \subset A \subset B
\end{array} \right\}.
\end{align*}
\end{theo}

One goes from one description to the other as follows. If $M = M_0 \oplus M_1 \in \mathcal N_{E/F}^{h}(k)$, one defines $A := M_0$ and $B := M_1^{\dagger}$. Conversely, if $(A \overset{h}{\subset} B) \in \mathcal N_{E/F}^{h}(k)$, one defines $M := A \oplus B^{\dagger}$. To prove that both maps are well-defined, one may rely on the following identity
\begin{equation*}
\pi L^{\vee} = \mathcal F (L^{\dagger}),
\end{equation*}
for all $W_{\mathcal O_F}(k)$-lattice $L\subset \mathbb N_{k,0}$. We will mostly use the second description throughout the paper, however the first one will be useful in Section \ref{SectionClosedSubschemes}.

\begin{rk}
If $h=0$ then $B = A$ and we require that $\pi B^{\vee} \overset{1}{\subset} B \subset B^{\vee}$. This case corresponds to the hyperspecial level as studied in \cite{vw1} and \cite{vw2}.\\
If $h=n$ then $A = \pi B$ and we require that $\pi^2 A^{\vee} \overset{1}{\subset} A \subset \pi A^{\vee}$. 
\end{rk}

In order to define the Bruhat-Tits stratification on $\mathcal N_{E/F,\mathrm{red}}^h$, Cho first introduces the sets of $k$-rational points of what will be the closed Bruhat-Tits strata, and proves that these sets cover $\mathcal N_{E/F}^h(k)$. \\
Given a point $(A\overset{h}{\subset} B) \in \mathcal N_{E/F}^h(k)$, for $i\geq 1$ we define 
\begin{align*}
T_i(A) & := A + \tau(A) + \ldots + \tau^{i-1}(A),\\
T_i(B) & := B + \tau(B) + \ldots + \tau^{i-1}(B).
\end{align*}
According to \cite{RZ} Proposition 2.17, the lattices $T_i(A)$ and $T_i(B)$ are $\tau$-invariant for $i$ large enough. Note that any $\tau$-invariant lattice $M\subset \mathbb N_{k,0}$ has the form $M = M^{\tau}\otimes_{W_{\mathcal O_F}(\mathbb F_{q^2})}W_{\mathcal O_F}(k) =: M^{\tau}_k$, where $M^{\tau}$ is seen as a lattice in $C$. We will always denote by $\Lambda_A$ and $\Lambda_B$ the lattices in $C$ such that $(\Lambda_A)_k = T_i(A)$ and $(\Lambda_B)_k = T_i(B)$ for $i$ large enough. The following statement is proved in \cite{cho} Lemma 2.7 and Lemma 2.8. 

\begin{lem}\label{CrucialLemmaCho}
Let $(A\overset{h}{\subset} B)\in \mathcal N_{E/F}^h(k)$ and let $c,d\geq 1$ be the smallest positive integers such that $T_c(A)$ and $T_d(B)$ are $\tau$-invariant.
\begin{enumerate}
\item For all $1\leq i < c$ we have 
\begin{align*}
T_i(A)\overset{1}{\subset} T_{i+1}(A), & & \text{and if } i+1<c\text{ then } \tau(T_i(A)) = T_{i+1}(A)\cap\tau(T_{i+1}(A)).
\end{align*}
\item For all $1\leq j < d$ we have 
\begin{align*}
T_j(B)\overset{1}{\subset} T_{j+1}(B), & & \text{and if } j+1<d\text{ then } \tau(T_j(B)) = T_{j+1}(B)\cap\tau(T_{j+1}(B)).
\end{align*}
\item We have $T_c(A) \subset \pi T_c(A)^{\vee}$ or $T_d(B) \subset T_d(B)^{\vee}$. Furthermore, if $c<d$ then the former is true, and if $d<c$ then the latter is true. 
\end{enumerate}
\end{lem}

\begin{rk}
If $h=0$ then we always have $T_d(B) \subset T_d(B)^{\vee}$. If $h=n$ then we always have $T_c(A) \subset \pi T_c(A)^{\vee}$. 
\end{rk}

\begin{defi}
For $i\in\mathbb Z$, we define 
\begin{equation*}
\mathcal L_i := \left\{W_{\mathcal O_F}(\mathbb F_{q^2})\text{-lattices } \Lambda \subset C \,\middle|\, \pi^{i+1}\Lambda^{\vee} \subset \Lambda \subset \pi^i\Lambda^{\vee}\right\}.
\end{equation*}
The elements of $\mathcal L_i$ are called \textit{vertex lattices} of rank $i$. The \textit{type} of a vertex lattice $\Lambda \in \mathcal L_i$ is the index $t(\Lambda) := [\Lambda:\pi^{i+1}\Lambda^{\vee}]$.
\end{defi}

\begin{rk}
The notion of vertex lattice depends on the pairing $\langle\cdot,\cdot\rangle_{[h]}$, although the notation $\mathcal L_i$ does not make it apparent. 
\end{rk}

\begin{lem}\label{ParityType}
Given a vertex lattice $\Lambda \in \mathcal L_i$, we have $0 \leq t(\Lambda) \leq n$ and 
\begin{equation*}
t(\Lambda) \equiv \begin{cases}
h+1 \mod 2 & \text{if } i \text{ is even},\\
n-h+1 \mod 2 & \text{if } i \text{ is odd}.
\end{cases}
\end{equation*}
\end{lem} 

\begin{proof}
The fact that $0 \leq t(\Lambda) \leq n$ is obvious. On the other hand, by construction, $\mathbb X^{[h]}\otimes k$ together with $\rho_{\mathbb X^{[h]}} := \mathrm{id}$ defines a point in $\mathcal N_{E/F}^h(k)$. This corresponds to lattices $\mathbb A \overset{h}{\subset} \mathbb B \subset \mathbb N_{k,0}$ by relative Dieudonné theory. We have 
\begin{align*}
t(\Lambda) = [\Lambda_k:\pi^{i+1}\Lambda_k^{\vee}] & = [\Lambda_k:\mathbb A] + [\mathbb A:\pi \mathbb B^{\vee}] + [\pi \mathbb B^{\vee}:\pi \mathbb A^{\vee}] + [\pi \mathbb A^{\vee}:\pi^{i+1}\Lambda_k^{\vee}]\\
& = [\Lambda_k:\mathbb A] + 1 - h + [\pi^{-i}\Lambda_k:\mathbb A] \\
& = 2[\Lambda_k:\mathbb A] + 1 - h + ni.
\end{align*}
The result follows.
\end{proof}

Given a point $(A\overset{h}{\subset} B)\in \mathcal N_{E/F}^h(k)$, we always have the following diagram 
\begin{center}
\begin{tikzcd}[column sep=small, row sep=small]
& \pi A^{\vee} \arrow[r,symbol=\subset, outer sep=2pt,"1"] & B \arrow[r,symbol=\subset] & (\Lambda_B)_k \\
\pi (\Lambda_B)_k^{\vee} \arrow[r,symbol=\subset] & \pi B^{\vee} \arrow[u,symbol=\subset] \arrow[r,symbol=\subset, outer sep=2pt,"1"] & A \arrow[u,symbol=\subset] \arrow[r,symbol=\subset] & (\Lambda_A)_k \arrow[u,symbol=\subset] \\
\pi^2 (\Lambda_A)_k^{\vee} \arrow[u,symbol=\subset] \arrow[r,symbol=\subset] & \pi^2A^{\vee} \arrow[u,symbol=\subset] \arrow[r,symbol=\subset, outer sep=2pt,"1"] & \pi B \arrow[u,symbol=\subset] & 
\end{tikzcd}
\end{center}

Thus, item 3. of Lemma \ref{CrucialLemmaCho} says that $\Lambda_B \in \mathcal L_0$ or $\Lambda_A \in \mathcal L_1$. In particular, the former holds if $d<c$ and the latter holds for $c<d$. Note that both cases are not exclusive. This motivates the following definitions.

\begin{defi}\label{BTStrataForMaximalParahoric}
Let $\Lambda_0 \in \mathcal L_0$ and $\Lambda_1 \in \mathcal L_1$. We define the following sets
\begin{align*}
\mathcal N^h_{\Lambda_0}(k) & := \left\{(A\overset{h}{\subset} B)\in \mathcal N_{E/F}^h(k) \,\middle|\, 
\begin{tikzcd}[ampersand replacement=\&,column sep=small, row sep=small]
\& \pi A^{\vee} \arrow[r,symbol=\subset, outer sep=2pt,"1"] \& B \arrow[r,symbol=\subset] \& (\Lambda_0)_k \\
\pi (\Lambda_0)_k^{\vee} \arrow[r,symbol=\subset] \& \pi B^{\vee} \arrow[u,symbol=\subset] \arrow[r,symbol=\subset, outer sep=2pt,"1"] \& A \arrow[u,symbol=\subset] \&
\end{tikzcd} \right\},\\
\mathcal N^h_{\Lambda_1}(k) & := \left\{(A\overset{h}{\subset} B)\in \mathcal N_{E/F}^h(k) \,\middle|\, 
\begin{tikzcd}[ampersand replacement=\&,column sep=small, row sep=small]
\& \pi B^{\vee} \arrow[r,symbol=\subset, outer sep=2pt,"1"] \& A \arrow[r,symbol=\subset] \& (\Lambda_1)_k \\
\pi^2 (\Lambda_1)_k^{\vee} \arrow[r,symbol=\subset] \& \pi^2A^{\vee} \arrow[u,symbol=\subset] \arrow[r,symbol=\subset, outer sep=2pt,"1"] \& \pi B \arrow[u,symbol=\subset] \& 
\end{tikzcd} \right\}.
\end{align*}
Moreover, if $\pi\Lambda_1^{\vee} \subset \Lambda_0$ then we define 
\begin{equation*}
\mathcal N^h_{\Lambda_0,\Lambda_1}(k) := \left\{(A\overset{h}{\subset} B)\in \mathcal N_{E/F}^h(k) \,\middle|\, 
\begin{tikzcd}[ampersand replacement=\&,column sep=small, row sep=small]
(\Lambda_0)_k \& B \arrow[l,symbol=\subset] \& \pi A^{\vee} \arrow[l,symbol=\subset, outer sep=2pt,swap,"1"] \& \pi(\Lambda_1)_k^{\vee} \arrow[l,symbol=\subset] \\
\pi (\Lambda_0)_k^{\vee} \arrow[r,symbol=\subset] \& \pi B^{\vee} \arrow[r,symbol=\subset, outer sep=2pt,"1"] \& A \arrow[r,symbol=\subset] \& (\Lambda_1)_k \arrow[u,symbol=\subset]
\end{tikzcd} \right\}.
\end{equation*}
\end{defi}

In \cite{cho}, the sets $\mathcal N^h_{\Lambda_0}(k)$ and $\mathcal N^h_{\Lambda_1}(k)$ are denoted by $S_{\Lambda_0}(k)$ and $R_{\Lambda_1}(k)$ respectively. Moreover, when $\pi\Lambda_1^{\vee} \subset \Lambda_0$ we have 
\begin{equation*}
\mathcal N^h_{\Lambda_0,\Lambda_1}(k) = \mathcal N^h_{\Lambda_0}(k) \cap \mathcal N^h_{\Lambda_1}(k).
\end{equation*}
We point out that $\mathcal N^h_{\Lambda_0}(k) \not = \emptyset$ if and only if $t(\Lambda_0) \geq h+1$, and $\mathcal N^h_{\Lambda_1}(k) \not = \emptyset$ if and only if $t(\Lambda_1) \geq n-h+1$. Given an integer $0 \leq x \leq n$, we write 
\begin{equation*}
\mathcal L_i^{\geq x} := \{\Lambda \in \mathcal L_i \,|\, t(\Lambda) \geq x \}.
\end{equation*}
The following theorem is a direct consequence of Lemma \ref{CrucialLemmaCho}. 

\begin{theo}
We have 
\begin{equation*}
\mathcal N_{E/F}^h(k) = \bigcup_{\Lambda_0 \in \mathcal L_0^{\geq h+1}} \mathcal N_{\Lambda_0}^h(k) \cup \bigcup_{\Lambda_1 \in \mathcal L_1^{\geq n-h+1}} \mathcal N_{\Lambda_1}^h(k).
\end{equation*}
\end{theo}

\begin{rk}
If $h=0$ the set $\mathcal L_1^{\geq n-h+1}$ is empty. Likewise, if $h=n$ the set $\mathcal L_0^{\geq h+1}$ is empty. 
\end{rk}

It turns out that $\mathcal N^h_{\Lambda_0}(k)$ and $\mathcal N^h_{\Lambda_1}(k)$ are the sets of $k$-points of closed subvarieties of $\mathcal N_{E/F,\mathrm{red}}^h$, which correspond to the closures of the Bruhat-Tits strata. Moreover, these closed subvarieties turn out to be isomorphic to the closure of some coarse Deligne-Lusztig varieties attached to finite unitary groups. 

\subsection{Some results on vertex lattices} \label{Section2.2}

In this section we prove several helpful results related to vertex lattices. 

\begin{prop}
Let $i<j$ be two distinct integers. If $j\not = i+1$ then $\mathcal L_i \cap \mathcal L_j = \emptyset$. Moreover, we have 
\begin{equation*}
\mathcal L_i \cap \mathcal L_{i+1} = \begin{cases}
\{\Lambda \subset C \,|\, \Lambda = \pi^{i+1}\Lambda^{\vee}\} & \text{if } h+1 \equiv ni \mod 2,\\
\emptyset & \text{else}.
\end{cases}
\end{equation*}
\end{prop}

\begin{rk}
It follows that the rank of a vertex lattice is not exactly unique. In particular, the type $t(\Lambda)$ of a vertex lattice might not be well-defined. Indeed, if $\Lambda = \pi^{i+1}\Lambda^{\vee}$, then $\Lambda$ has type $0$ as a vertex lattice of rank $i$, and has type $n$ as a vertex lattice of rank $i+1$. We hope that the context is clear enough to avoid any confusion anytime we refer to the type of a vertex lattice in the remaining of the exposition.
\end{rk}

\begin{proof}
Assume that there exists $\Lambda \in \mathcal L_i \cap \mathcal L_j$. We have $\pi^{i+1}\Lambda^{\vee} \subset \Lambda \subset \pi^{j}\Lambda^{\vee}.$ It follows that $i+1 \geq j$. Since $i<j$, we have $j=i+1$. Thus, we deduce that $\Lambda = \pi^{i+1}\Lambda^{\vee}$. It follows that as a vertex lattice of rank $i$, $\Lambda$ has type $0$. By Lemma \ref{ParityType}, if $i$ is even then $h+1 \equiv 0 \mod 2$, and if $i$ is odd then $h+1 \equiv n \mod 2$. In other words, we have $h+1 \equiv ni \mod 2$.\\
Conversely, if $\Lambda$ is a lattice such that $\Lambda = \pi^{i+1}\Lambda^{\vee}$, then by definition we have $\Lambda \in \mathcal L_i \cap \mathcal L_{i+1}$, and the condition $h+1 \equiv ni \mod 2$ is satisfied. 
\end{proof}

\begin{prop}
Let $\Lambda \in \mathcal L_i$ for some integer $i\in \mathbb Z$. We have $\Lambda^{\vee} \in \mathcal L_{-i-1}$ and $t(\Lambda^{\vee}) = n - t(\Lambda)$. 
\end{prop}

\begin{proof}
This is straightforward, since we have $\pi^{-i}\Lambda \subset \Lambda^{\vee} \subset \pi^{-i-1}\Lambda$. 
\end{proof}

Let $i\in \mathbb Z$ and let $\Lambda \in \mathcal L_i$. We define
\begin{align*}
V_{\Lambda}^0 := \Lambda / \pi^{i+1}\Lambda^{\vee}, & & V_{\Lambda}^1 := \pi^i \Lambda^{\vee} / \Lambda.
\end{align*}
Then $V_{\Lambda}^0$ and $V_{\Lambda}^1$ are $\mathbb F_{q^2}$-vector spaces of dimension respectively $t(\Lambda)$ and $n-t(\Lambda)$. The restriction to $\Lambda$ (resp. to $\pi^i \Lambda^{\vee}$) of the form $\pi^{-i}\{\cdot,\cdot\}_{[h]}$ (resp. $\pi^{-i+1}\{\cdot,\cdot\}_{[h]}$) induces a structure of non-degenerate $\mathbb F_{q^2}/\mathbb F_q$-hermitian space on $V_{\Lambda}^0$ (resp. on $V_{\Lambda}^1$). By abuse of notations, we still denote by $\{\cdot,\cdot\}_{[h]}$ the hermitian form on both spaces. If $U$ is a subspace of $V_{\Lambda}^{0}$ or of $V_{\Lambda}^1$, we denote its orthogonal by $U^{\perp}$. The following Proposition follows from \cite{vw1} and \cite{vw2}. 

\begin{prop}\label{SubOverVertexLattices}
Let $i\in \mathbb Z$ and let $\Lambda \in \mathcal L_i$. Let $t^{-}$ and $t^{+}$ be integers such that $0 \leq t^{-} \leq t(\Lambda) \leq t^{+} \leq n$ and $t^- \equiv t^+ \equiv t(\Lambda) \mod 2$. 
\begin{enumerate}
\item The mapping $\Lambda' \mapsto \Lambda'/\pi^{i+1}\Lambda^{\vee}$ defines a bijection between the set of vertex lattices $\Lambda' \in \mathcal L_i$ such that $\Lambda' \subset \Lambda$ and $t(\Lambda') = t^-$, and the set of subspaces $U \subset V_{\Lambda}^0$ such that $U^{\perp} \subset U$ and $\dim(U) = \frac{t(\Lambda) + t^-}{2}$.
\item The mapping $\Lambda' \mapsto \pi^{i}\Lambda^{\prime \vee}/\Lambda$ defines a bijection between the set of vertex lattices $\Lambda' \in \mathcal L_i$ such that $\Lambda \subset \Lambda'$ and $t(\Lambda') = t^+$, and the set of subspaces $U \subset V_{\Lambda}^1$ such that $U^{\perp} \subset U$ and $\dim(U) = n-\frac{t(\Lambda)+t^+}{2}$.
\end{enumerate}
\end{prop}

Note that if $\Lambda' \subset \Lambda$, we have 
\begin{equation*}
\pi^{i+1}\Lambda^{\vee} \overset{\frac{t(\Lambda)-t^-}{2}}{\subset} \pi^{i+1}\Lambda^{\prime\vee} \overset{t^-}{\subset} \Lambda' \overset{\frac{t(\Lambda)-t^-}{2}}{\subset} \Lambda.
\end{equation*}
Together with the identity $(\Lambda'/\pi^{i+1}\Lambda^{\vee})^{\perp} = \pi^{i+1}\Lambda^{\prime\vee} /\pi^{i+1}\Lambda^{\vee}$, it justifies that the map in 1. is well-defined. Likewise, if $\Lambda \subset \Lambda'$, we have 
\begin{equation*}
\Lambda \overset{\frac{t^+-t(\Lambda)}{2}}{\subset} \Lambda' \overset{n-t^+}{\subset} \pi^{i}\Lambda^{\prime\vee} \overset{\frac{t^+-t(\Lambda)}{2}}{\subset} \pi^{i}\Lambda^{\vee}.
\end{equation*}
Together with the identity $(\pi^i\Lambda^{\prime\vee}/\Lambda)^{\perp} = \Lambda'/\Lambda$, it justifies that the map in 2. is well-defined.

\subsection{The moduli space $\mathcal N_{E/F}^{\mathbbm h}$ at arbitrary parahoric level} \label{Section2.3}

In this section, we are going to generalize the results of \cite{cho} to the case of arbitrary parahoric level. First of all, given $0 \leq h < h' \leq n$, let us observe that there exists an isogeny $\alpha_{h',h}:\mathbb X^{[h']}\to \mathbb X^{[h]}$ compatible with the additional structures and such that $\mathrm{Ker}(\alpha_{h',h})\subset \mathbb X[\pi]$ has degree $q^{h'-h}$, if and only if $h \equiv h' \mod 2$, see \cite{LRZ} Section 3.4. For $h \leq n-2$, we fix such an isogeny $\alpha_{h+2,h}$. For $h$ and $h'$ as above with the same parity, we define $\alpha_{h',h} := \alpha_{h+2,h}\circ \ldots \circ\alpha_{h',h'-2}$. In particular, the compatibility with the polarizations means that the following diagram
\begin{center}
\begin{tikzcd}
\mathbb X^{[h']} \arrow[r,"\alpha_{h',h}"] \arrow[d,swap,"\lambda_{\mathbb X}^{[h']}"] & \mathbb X^{[h]} \arrow[d,"\lambda_{\mathbb X}^{[h]}"] \\
\mathbb X^{[h']\vee} & \mathbb X^{[h]\vee} \arrow[l,"\alpha_{h',h}^{\vee}"] 
\end{tikzcd}
\end{center}
commutes up to a scalar in $\mathbb F_{q^2}$. Thus, $\alpha_{h',h}$ induces an isometry between $(\mathbb N,\langle\cdot,\cdot\rangle_{[h']})$ and $(\mathbb N,\langle\cdot,\cdot\rangle_{[h]})$.
\\
Let $m \geq 1$ and let $\mathbbm h = (h_1,\ldots,h_m)$ be a $m$-tuple of integers such that $0 \leq h_1 < \ldots < h_m \leq n$, and such that all the $h_i$'s have the same parity. We define a functor $\mathcal N_{E/F}^{\mathbbm h}$ as follows. For $S \in \textbf{Nilp}$, let $\mathcal N_{E/F}^{\mathbbm h}(S)$ denote the set of tuples $(X^{[i]},i_{X^{[i]}},\lambda_{X^{[i]}},\rho_{X^{[i]}})_{1\leq i \leq m}$ up to isomorphism, where 
\begin{itemize}
\item for all $1\leq i \leq m$, $(X^{[i]},i_{X^{[i]}},\lambda_{X^{[i]}},\rho_{X^{[i]}})\in \mathcal N_{E/F}^{h_i}(S)$,
\item for $1\leq i < m$, there exists an isogeny $\widetilde{\alpha}_{i+1,i}:X^{[i+1]}\to X^{[i]}$ such that the following diagram commutes 
\begin{center}
\begin{tikzcd}[column sep = large]
X^{[i+1]}_{\overline S} \arrow[r,"(\widetilde{\alpha}_{i+1,i})_{\overline S}"] \arrow[d,swap,"\rho_{X^{[i+1]}}"] & X^{[i]}_{\overline S} \\
\mathbb X_{\overline S}^{[h_{i+1}]} \arrow[r,"(\alpha_{h_{i+1},h_i})_{\overline S}"] & \mathbb X_{\overline S}^{[h_i]} \arrow[u,swap,"\rho_{X^{[i]}}^{-1}"]
\end{tikzcd}
\end{center}
\end{itemize}

Note that the isogeny $\widetilde{\alpha}_{i+1,i}$, when it exists, is unique. An isomorphism 
\begin{equation*}
(X^{[i]},i_{X^{[i]}},\lambda_{X^{[i]}},\rho_{X^{[i]}})_{1\leq i \leq m} \xrightarrow{\sim} (X^{\prime[i]},i_{X^{\prime[i]}},\lambda_{X^{\prime[i]}},\rho_{X^{\prime[i]}})_{1\leq i \leq m}
\end{equation*}
is a collection of isomorphisms $\gamma_i:(X^{[i]},i_{X^{[i]}},\lambda_{X^{[i]}},\rho_{X^{[i]}}) \xrightarrow{\sim} (X^{\prime[i]},i_{X^{\prime[i]}},\lambda_{X^{\prime[i]}},\rho_{X^{\prime[i]}})$ in the sense of the maximal parahoric case, such that 
\begin{equation*}
\gamma_{i}^{-1} \circ \widetilde{\alpha'}_{i+1,i}\circ\gamma_{i+1} = \widetilde{\alpha}_{i+1,i}
\end{equation*}
for all $1 \leq i < m$.

\begin{prop}
The functor $\mathcal N_{E/F}^{\mathbbm h}\otimes \mathcal O_{\breve E}$ is represented by a formal scheme over $\mathrm{Spf}(\mathcal O_{\breve E})$ which is locally formally of finite type and regular.
\end{prop}

By abuse of notations, $\mathcal N_{E/F}^{\mathbbm h}$ will denote this formal scheme over $\mathrm{Spf}(\mathcal O_{\breve E})$. We refer to it as the \textit{(relative) basic unramified unitary Rapoport-Zink space of parahoric level $\mathbbm h$}. For $1\leq i < m$, we define 
\begin{equation*}
\Delta h_i := \frac{h_{i+1}-h_i}{2}.
\end{equation*}
Let $k$ be an algebraically closed field containing $\kappa_{\breve E}$. By relative Dieudonné theory and following \cite{cho}'s arguments, it is easy to see that the $k$-points of $\mathcal N_{E/F}^{\mathbbm h}$ are described as follows.

\begin{prop}\label{PointsRZSpaceOverFields}
There is a bijection 
\begin{equation*}
\mathcal N_{E/F}^{\mathbbm h}(k) \simeq \left\{\begin{array}{c}
W_{\mathcal O_F}(k)\text{-lattices in }\mathbb N_{k,0}\\
A_m \subset \ldots \subset A_1 \subset B_1 \subset \ldots \subset B_m
\end{array}
\,\middle|\, \forall 1\leq i \leq m, 
\begin{array}{c}
\pi A_i^{\vee} \overset{1}{\subset} B_i \subset A_i^{\vee},\\
\pi B_i^{\vee} \overset{1}{\subset} A_i \subset B_i^{\vee},\\
\pi B_i \subset A_i \overset{h_i}{\subset} B_i 
\end{array}
\right\}.
\end{equation*}
\end{prop}

\begin{rk}
In the right-hand side, the dual lattices are taken with respect to a single fixed pairing $\{\cdot,\cdot\}_{[h_i]}$ for some $1\leq i \leq m$. For another $1 \leq j \leq m$, the isogeny $\alpha_{h_j,h_i}$ (if $j>i$) or the isogeny $\alpha_{h_i,h_j}$ (if $j<i$) induces a bijection between the two right-hand side sets defined respectively with $\{\cdot,\cdot\}_{[h_i]}$ and with $\{\cdot,\cdot\}_{[h_j]}$. In order to remove ambiguity, we impose the following convention.\\
\textbf{Convention:} Unless precised otherwise, the space $\mathbb N_{0}$ and its base changes are always equipped with the pairing $\{\cdot,\cdot\}_{[h_1]}$. If $M \subset \mathbb N_{0}$ is a lattice, the dual lattice $M^{\vee}$ is taken with respect to $\{\cdot,\cdot\}_{[h_1]}$. In particular, this applies to vertex lattices $\Lambda \in \mathcal L_i$ as well.
\end{rk} 

\begin{rk}
If $h_1 = 0$ then $A_1 = B_1$, and if $h_m = n$ then $A_m = \pi B_m$. 
\end{rk}

Our first goal is to define subsets of $\mathcal N_{E/F}^{\mathbbm h}(k)$ indexed by vertex lattices, which will later become the sets of $k$-rational points of the closed Bruhat-Tits strata. Let us start off with the following lemma.

\begin{lem}
Given a point $(A_m\subset \ldots \subset B_m)\in \mathcal N_{E/F}^{\mathbbm h}(k)$, for $1\leq i < m$ we have 
\begin{align*}
A_{i+1}\ \overset{\Delta h_{i}}{\subset} A_i, & & B_i \overset{\Delta h_i}{\subset} B_{i+1}.
\end{align*}
\end{lem}

\begin{proof}
Let us write
\begin{equation*}
A_{m} \overset{x_{m-1}}{\subset} \ldots \overset{x_1}{\subset} A_1 \overset{h_1}{\subset} B_1 \overset{y_1}{\subset} \ldots \overset{y_{m-1}}{\subset} B_{m},
\end{equation*}
for some integers $x_i,y_i\geq 0$. Since $A_{i+1} \overset{h_{i+1}}{\subset} B_{i+1}$, for all $1\leq i < m$ we have 
\begin{equation*}
x_{i}+\ldots + x_1 + h_1 + y_1 + \ldots + y_i = h_{i+1}.
\end{equation*} 
Thus, by induction, it is enough to prove that $x_i = y_i$. We have
\begin{align*}
x_i = [A_i:A_{i+1}] & = [A_i:\pi B_i^{\vee}] + [\pi B_i^{\vee}:\pi B_{i+1}^{\vee}] + [\pi B_{i+1}^{\vee}:A_{i+1}]\\
& = 1 + [B_{i+1} : B_i] - 1 = y_i,
\end{align*}
thus the result follows.
\end{proof}

For $1\leq i \leq m$, let $c_i,d_i \geq 1$ denote the smallest positive integers such that $T_{c_i}(A_i)$ and $T_{d_i}(B_i)$ are $\tau$-stable. Let $\Lambda_{A_i}$ and $\Lambda_{B_i}$ be the $W_{\mathcal O_F}(\mathbb F_{q^2})$-lattices in $C$ such that $T_{c_i}(A_i) = (\Lambda_{A_i})_k$ and $T_{d_i}(B_i) = (\Lambda_{B_i})_k$. 

\begin{rk}
If $h_1 = 0$, we have $B_1 = A_1$ so that $c_1 = d_1$. Likewise, if $h_m = n$ we have $c_k = d_k$. 
\end{rk}

\begin{lem}\label{OneVertexLatticeThenManyOther}
For $1 \leq i \leq m$, we have 
\begin{enumerate}
\item if $\Lambda_{B_i} \in \mathcal L_0$ then $\Lambda_{B_1},\ldots , \Lambda_{B_i} \in \mathcal L_0$,
\item if $\Lambda_{A_i} \in \mathcal L_i$ then $\Lambda_{A_i},\ldots,\Lambda_{A_m} \in \mathcal L_1$.
\end{enumerate}
\end{lem}

\begin{proof}
As observed in the case $m=1$, we always have $\pi \Lambda_{B_i}^{\vee} \subset \Lambda_{B_i}$ and $\pi^{2}\Lambda_{A_i}^{\vee} \subset \Lambda_{A_i}$. Now let us assume that $\Lambda_{B_i}\in\mathcal L_0$. Then we have 
\begin{equation*}
\Lambda_{B_1} \subset \ldots \subset \Lambda_{B_i} \subset \Lambda_{B_i}^{\vee} \subset \ldots \subset \Lambda_{B_1}^{\vee},
\end{equation*}
from which it follows that $\Lambda_{B_1},\ldots , \Lambda_{B_i} \in \mathcal L_0$. Point 2. is proved similarly.
\end{proof}

Given a point $(A_m\subset \ldots \subset B_m) \in \mathcal N_{E/F}^{\mathbbm h}(k)$, let us define two integers 
\begin{align*}
a := \min \{1\leq i \leq m, \Lambda_{A_i} \in \mathcal L_1\}, & & b := \max \{1\leq i \leq m, \Lambda_{B_i} \in \mathcal L_0\}.
\end{align*}

If none of the $\Lambda_{A_i}$'s are in $\mathcal L_i$, we put $a := m+1$. Likewise, if none of the $\Lambda_{B_i}$'s is in $\mathcal L_0$, we put $b = 0$.

\begin{lem}
We have $a \leq b+1$.
\end{lem}

\begin{proof}
According to Lemma \ref{CrucialLemmaCho}, for all $1 \leq i \leq m$ we have $\Lambda_{A_i} \in \mathcal L_1$ or $\Lambda_{B_i} \in \mathcal L_0$. If $b = m$, the inequality is trivial. Otherwise, if $b<m$ by definition we have $\Lambda_{B_{b+1}} \not \in \mathcal L_0$. Thus, we must have $\Lambda_{A_{b+1}} \in \mathcal L_1$, which implies $a\leq b+1$.
\end{proof}

Table 1 shows a table aimed at helping visualize the repartition of vertex lattices among the $\Lambda_{A_i}$'s and the $\Lambda_{B_i}$'s. For each box, a symbol $\bigcirc$ indicates that the corresponding lattice is a vertex lattice in $\mathcal L_0$ for the first row, and in $\mathcal L_1$ for the second row. On the other hand, a symbol $\times$ indicates that the corresponding lattice is not a vertex lattice. 

\begin{table}[h]
\centering
\begin{tabular}{|c|ccccccccc|}
\hline 
$i$ & $1$ & $2$ & $\ldots$ & $a$ & $\ldots$ & $b$ & $\ldots$ & $m-1$ & $m$\\
\hline
$\Lambda_{B_i}$ & $\bigcirc$ & $\bigcirc$ & $\bigcirc$ & $\bigcirc$ & $\bigcirc$ & $\bigcirc$ & $\times$ & $\times$ & $\times$ \\
$\Lambda_{A_i}$ & $\times$ & $\times$ & $\times$ & $\bigcirc$ & $\bigcirc$ & $\bigcirc$ & $\bigcirc$ & $\bigcirc$ & $\bigcirc$ \\
\hline
\end{tabular}
\caption{The repartition of vertex lattices.}
\end{table}

Note that in the extremal case $a=b+1$, the circles in the two rows do not overlap, see Table 2. 

\begin{table}[h]
\centering
\begin{tabular}{|c|cccccccc|}
\hline 
$i$ & $1$ & $2$ & $\ldots$ & $b$ & $a$ & $\ldots$ & $m-1$ & $m$\\
\hline
$\Lambda_{B_i}$ & $\bigcirc$ & $\bigcirc$ & $\bigcirc$ & $\bigcirc$ & $\times$ & $\times$ & $\times$ & $\times$ \\
$\Lambda_{A_i}$ & $\times$ & $\times$ & $\times$ & $\times$ & $\bigcirc$ & $\bigcirc$ & $\bigcirc$ & $\bigcirc$ \\
\hline
\end{tabular}
\caption{The repartition of vertex lattices when $a=b+1$.}
\end{table}

\begin{rk}
We point out that when $h_1 = 0$, we always have $\Lambda_{B_1} \in \mathcal L_0$ and $\Lambda_{A_1} \not \in \mathcal L_1$. Thus, we have $a \geq 2$ and $b \geq 1$ in this case. Likewise, we have $a \leq m$ and $b \leq m-1$ if $h_m = n$. 
\end{rk}

\begin{prop}\label{ConditionStar}
Let $1 \leq i < j \leq m$. The following statements hold. 
\begin{enumerate}
\item We have 
\begin{align*}
\Lambda_{B_i}\in\mathcal L_0 & \iff \Lambda_{B_i}+\pi\Lambda_{A_j}^{\vee} \in \mathcal L_0, \\
\Lambda_{A_j}\in\mathcal L_1 & \iff \Lambda_{B_i}\cap \pi \Lambda_{A_j}^{\vee} \in \mathcal L_0.
\end{align*}
\item If $\Lambda_{B_i} \subset \pi \Lambda_{A_j}^{\vee}$ then $\Lambda_{B_i} \in \mathcal L_0$ and $\Lambda_{A_j} \in \mathcal L_1$.
\item If $d_i < d_j$ or if $c_j < c_i$ then $\Lambda_{B_i} \subset \pi \Lambda_{A_j}^{\vee}$.
\end{enumerate}
\end{prop}

\begin{proof}
1. Observe that 
\begin{equation*}
(\Lambda_{B_i}+\pi\Lambda_{A_j}^{\vee})^{\vee} = \Lambda_{B_i}^{\vee}\cap \pi^{-1}\Lambda_{A_j}.
\end{equation*}
On the one hand, the inclusion $\pi(\Lambda_{B_i}+\pi\Lambda_{A_j}^{\vee})^{\vee} \subset \Lambda_{B_i}+\pi\Lambda_{A_j}^{\vee}$ is obvious since $\pi \Lambda_{B_i}^{\vee} \subset \Lambda_{B_i}$. On the other hand, the inclusion $\Lambda_{B_i}+\pi\Lambda_{A_j}^{\vee} \subset (\Lambda_{B_i}+\pi\Lambda_{A_j}^{\vee})^{\vee}$ is equivalent to 
\begin{multicols}{2}
\begin{enumerate}[label=\alph*.]
\item $\Lambda_{B_i} \subset \Lambda_{B_i}^{\vee}$,
\item $\Lambda_{B_i} \subset \pi^{-1}\Lambda_{A_j}$,
\item $\pi\Lambda_{A_j}^{\vee} \subset \Lambda_{B_i}^{\vee}$,
\item $\pi\Lambda_{A_j}^{\vee} \subset \pi^{-1}\Lambda_{A_j}$.
\end{enumerate}
\end{multicols}
Now, b. is always true, c. is equivalent to b. by duality, d. is always true and a. is equivalent to $\Lambda_{B_i} \in \mathcal L_0$. The other equivalence of 1. is proved similarly. \\
2. Assume that $\Lambda_{B_i} \subset \pi\Lambda_{A_j}^{\vee}$. If $\Lambda_{B_i} \in \mathcal L_0$, then by 1. we have $\pi\Lambda_{A_j}^{\vee} \in \mathcal L_0$, thus $\Lambda_{A_j}\in\mathcal L_1$. Now, assume that $\Lambda_{B_i} \not \in \mathcal L_0$. By Lemma \ref{CrucialLemmaCho}, we must have $\Lambda_{A_i} \in \mathcal L_1$. Then by Lemma \ref{OneVertexLatticeThenManyOther}, since $i<j$ we must have $\Lambda_{A_j} \in \mathcal L_1$ as well. By 1. it follows that $\Lambda_{B_i} \in \mathcal L_0$, which is a contradiction. \\
3. Assume that $d_i < d_j$. The inclusion $B_i \subset B_j$ implies that $T_{s}(B_i) \subset T_{s}(B_j)$ for all $s\geq 1$. Consider $s=d_i$. We have $T_{d_i}(B_i) = (\Lambda_{B_i})_k \subset T_{d_i}(B_j)$. Since $(\Lambda_{B_i})_k$ is $\tau$-invariant, we deduce that 
\begin{equation*}
(\Lambda_{B_i})_k \subset \bigcap_{\ell\in\mathbb Z} \tau^{\ell}(T_{d_i}(B_j)).
\end{equation*}
Since $d_i < d_j$, we know that $T_{d_i}(B_j)\cap \tau(T_{d_i}(B_j)) = \tau(T_{d_i-1}(B_j))$ by Lemma \ref{CrucialLemmaCho}. Thus, we have 
\begin{equation*}
(\Lambda_{B_i})_k \subset \bigcap_{\ell\in\mathbb Z} \tau^{\ell}(T_{d_i-1}(B_j)).
\end{equation*}
By induction, it actually follows that $(\Lambda_{B_i})_k \subset \bigcap_{\ell\in\mathbb Z} \tau^{\ell}(B_j)$. Now, notice that we have 
\begin{align*}
\pi A_j^{\vee} \overset{1}{\subset} B_j, & & \pi A_j^{\vee} \overset{1}{\subset} \tau(B_j).
\end{align*} 
Since $d_j > 1$, we have $B_j \not = \tau(B_j)$. Thus, we must have $B_j\cap \tau(B_j) = \pi A_j^{\vee}$. It follows that 
\begin{equation*}
(\Lambda_{B_i})_k \subset \pi\bigcap_{\ell\in\mathbb Z}\tau^{\ell}(A_j)^{\vee} = \pi(\Lambda_{A_j})_k^{\vee}.
\end{equation*}
Assume now that $c_j < c_i$. Using the inclusion $A_j \subset A_i$, we prove similarly that $(\Lambda_{A_j})_k \subset \pi (\Lambda_{B_i})_k^{\vee}$. Taking duals, the result follows.
\end{proof}

For all $2 \leq i \leq a-1$, by definition we have $\Lambda_{A_i} \not \in \mathcal L_1$. By the previous lemma, it follows that $d_{i-1} \geq d_i$ and that $c_{i-1} \geq c_i$. Similarly, for $b+1\leq i \leq m-1$ we have $\Lambda_{B_i} \not \in \mathcal L_0$, thus one must have $d_i \geq d_{i+1}$ and $c_{i+1} \geq c_i$. Besides, if $\Lambda_{A_i} \not \in \mathcal L_1$ then $c_i \geq d_i$ by Lemma \ref{CrucialLemmaCho}. Likewise if $\Lambda_{B_i} \not \in \mathcal L_0$ then $d_i \geq c_i$. We sum up these inequalities in Table 3. 

\begin{table}[h]
\centering
\begin{tabular}{|c|ccccccccccccccccc|}
\hline 
$i$ & $1$ & & $2$ & & $\ldots$ & & $a-1$ & $a$ & $\ldots$ & $b$ & $b+1$ & & $\ldots$ & & $m-1$ & & $m$\\
\hline
$\Lambda_{B_i}$ & $d_1$ & $\geq$ & $d_2$ & $\geq$ & $\ldots$ & $\geq$ & $d_{a-1}$ & $d_a$ & $\ldots$ & $d_b$ & $d_{b+1}$ & $\geq$ & $\ldots$ & $\geq$ & $d_{m-1}$ & $\geq$ & $d_m$ \\
& \rotatebox{90}{$\geq\;\,$} & & & & & & & & & & & & & & & &  \rotatebox{90}{$\leq\;\,$}\\
$\Lambda_{A_i}$ & $c_1$ & $\leq$ & $c_2$ & $\leq$ & $\ldots$ & $\leq$ & $c_{a-1}$ & $c_a$ & $\ldots$ & $c_b$ & $c_{b+1}$ & $\leq$ & $\ldots$ & $\leq$ & $c_{m-1}$ & $\leq$ & $c_m$  \\
\hline
\end{tabular}
\caption{Inequalities between the $d_i$'s and the $c_i$'s.}
\end{table}

\begin{defi}\label{BTType}
Given a point $(A_m \subset \ldots \subset B_m) \in \mathcal N_{E/F}^{\mathbbm h}(k)$, we attach the subset $I \subset \{0,\ldots,m\}$ which contains 
\begin{itemize}
\item $0$ if and only if $\Lambda_{A_1} \in \mathcal L_1$,
\item $k$ if and only if $\Lambda_{B_m} \in \mathcal L_0$,
\item all the integers $1 \leq i \leq m-1$ such that $\Lambda_{B_i} \subset \pi \Lambda_{A_{i+1}}^{\vee}$.
\end{itemize}
We refer to the set $I$ as the \textit{Bruhat-Tits type} of the point $(A_m \subset \ldots \subset B_m)$. 
\end{defi}

\begin{lem}\label{CrucialLemma}
The Bruhat-Tits type of any point $(A_m \subset \ldots \subset B_m)\in \mathcal N_{E/F}^{\mathbbm h}(k)$ is not empty. 
\end{lem}

\begin{proof}
Towards a contradiction, let us consider a point $(A_m \subset \ldots \subset B_m)\in \mathcal N_{E/F}^{\mathbbm h}(k)$ whose Bruhat-Tits type is empty. In other words, we have $a>1$, $b<m$ and for all $1\leq i \leq m-1$, we have $\Lambda_{B_i} \not \subset \pi \Lambda_{A_{i+1}}^{\vee}$. In particular, it follows that all the integers $d_i$'s and $c_i$'s are equal. Let us write $t \geq 1$ for their common value.\\
Assume that for some $1\leq i \leq m-1$, we have $B_i \subset \pi A_{i+1}^{\vee}$. In particular, we have $t>1$. It follows that 
\begin{equation*}
(\Lambda_{B_i})_k = T_t(B_i) \subset \pi A_{i+1}^{\vee} + \pi \tau(A_{i+1}^{\vee}) + \ldots + \pi \tau^{t-1}(A_{i+1}^{\vee}).
\end{equation*}
Observe that $\pi A_{i+1}^{\vee} \overset{1}{\subset} \tau(B_{i+1})$ and that $\pi \tau(A_{i+1}^{\vee}) \overset{1}{\subset} \tau(B_{i+1})$. Since $A_{i+1} \not = \tau(A_{i+1})$, it follows that 
\begin{equation*}
\pi A_{i+1}^{\vee} + \pi \tau(A_{i+1}^{\vee}) = \tau(B_{i+1}).
\end{equation*}
Thus, we have $(\Lambda_{B_i})_k \subset \tau(T_{t-1}(B_{i+1}))$. Repeating the argument of the proof of point 3. of Proposition \ref{ConditionStar}, we deduce that 
\begin{equation*}
(\Lambda_{B_i})_k \subset \bigcap_{\ell\in\mathbb Z}\tau^{\ell}(T_{t-1}(B_{i+1})) = \bigcap_{\ell\in\mathbb Z} \tau^{\ell}(B_{i+1}) = \pi (\Lambda_{A_{i+1}})_k^{\vee},
\end{equation*}
which is a contradiction.\\
Thus, we have $B_i \not \subset \pi A_{i+1}^{\vee}$ for all $1\leq i \leq m-1$. Since $\pi A_{i+1}^{\vee} \overset{1}{\subset} B_{i+1}$, it follows that $B_{i+1} = B_i + \pi A_{i+1}^{\vee}$. In particular, by induction we have
\begin{equation*}
B_m = \pi A_m^{\vee} + B_1.
\end{equation*}
Likewise, one proves that $A_{i+1} \not \subset \pi B_i^{\vee}$ for all $1\leq i \leq m-1$. It follows that 
\begin{equation*}
A_1 = \pi B_1^{\vee} + A_m. 
\end{equation*} 
Consider the following chain of inclusions 
\begin{equation*}
\pi A_1^{\vee} \overset{x}{\subset} \pi A_1^{\vee} + A_m \overset{y}{\subset} B_1 \cap B_m^{\vee} \overset{z}{\subset} B_1,
\end{equation*}
where $x,y,z \geq 0$ are integers. Since the total index is $1$, we must have $x+y+z = 1$. Thus, we always have $x=0$ or $z=0$.\\
If $z=0$ then $B_1 \subset B_m^{\vee}$. Since $B_m = \pi A_m^{\vee} + B_1$, it follows that $B_m \subset B_m^{\vee}$. Recall that $b < m$, meaning that $\Lambda_{B_m} \not \in \mathcal L_0$. In particular, $t>1$. We then deduce that 
\begin{equation*}
(\Lambda_{B_m})_k = T_t(B_m) \subset B_m^{\vee} + \tau(B_m^{\vee}) + \ldots + \tau^{t-1}(B_m^{\vee}) = \pi^{-1}\tau(T_{t-1}(A_m)).
\end{equation*}
By the same argument as above, it follows that 
\begin{equation*}
(\Lambda_{B_m})_k \subset \pi^{-1} \bigcap_{\ell\in\mathbb Z} \tau^{\ell}(A_m) = \bigcap_{\ell\in\mathbb Z} \tau^{\ell}(B_m^{\vee}) = (\Lambda_{B_m})_k^{\vee},
\end{equation*}
from which we deduce that $\Lambda_{B_m} \in \mathcal L_0$, a contradiction.\\
If $x=0$, we have $A_m \subset \pi A_1^{\vee}$ from which we deduce $A_1 \subset \pi A_1^{\vee}$. Likewise, it follows that $\Lambda_{A_1} \in \mathcal L_1$, leading to a contradiction. This concludes the proof. 
\end{proof}

\subsection{Subsets $\mathcal N_{I,\mathbf{\Lambda}}^{\mathbbm h}(k)$ of $\mathcal N_{E/F}^{\mathbbm h}(k)$} \label{Section2.4}

We define subsets of $\mathcal N_{E/F}^{\mathbbm h}(k)$ as follows. 

\begin{defi}\label{DefinitionBTIndex}
A \textit{Bruhat-Tits index} is a pair $(I,\mathbf{\Lambda})$ where $I$ is a non-empty subset of $\{0,\ldots,m\}$ such that 
\begin{itemize}
\item if $h_1 = 0$ then $0 \not \in I$,
\item if $h_m = n$ then $m \not\in I$,
\end{itemize}
and $\mathbf{\Lambda}$ is a collection of vertex lattices as follows:
\begin{itemize}
\item for $i\in I\setminus \{0\}$, let $\Lambda_0^{i} \in \mathcal L_0^{\geq h_{i}+1}$,
\item for $j\in I\setminus \{m\}$, let $\Lambda_1^{j} \in \mathcal L_1^{\geq n-h_{j+1}+1}$.
\end{itemize}  
Furthermore, if we write $0 \leq i_1 < \ldots < i_s \leq m$ for the elements of $I$, we impose the following condition.
\begin{equation*}
\begin{array}{lccccccccccccc}
\underline{\text{If } i_1 \not = 0 \text{ and } i_s \not = m:} & \Lambda_0^{i_1} & \subset & \pi\Lambda_1^{i_1\vee} & \subset & \Lambda_0^{i_2} & \subset & \ldots & \subset & \pi \Lambda_1^{i_{s-1}\vee} & \subset & \Lambda_0^{i_s} & \subset & \pi\Lambda_1^{i_s\vee}.\\
\underline{\text{If } i_1 = 0 \text{ and } i_s \not = m:} & & & \pi\Lambda_1^{0\vee} & \subset & \Lambda_0^{i_2} & \subset & \ldots & \subset & \pi \Lambda_1^{i_{s-1}\vee} & \subset & \Lambda_0^{i_s} & \subset & \pi\Lambda_1^{i_s\vee}.\\
\underline{\text{If } i_1 \not = 0 \text{ and } i_s = m:}& \Lambda_0^{i_1} & \subset & \pi\Lambda_1^{i_1\vee} & \subset & \Lambda_0^{i_2} & \subset & \ldots & \subset & \pi\Lambda_1^{i_{s-1}\vee} & \subset & \Lambda_0^{m}. & &\\
\underline{\text{If } i_1 = 0 \text{ and } i_s = m:}& & & \pi\Lambda_1^{0\vee} & \subset & \Lambda_0^{i_2} & \subset & \ldots & \subset & \pi\Lambda_1^{i_{s-1}\vee} & \subset & \Lambda_0^{m}. & &
\end{array}
\end{equation*}
\end{defi}

\begin{defi}\label{DefinitionBTStrata}
Let $(I,\mathbf{\Lambda})$ be a Bruhat-Tits index. Let us write $0 \leq i_1 < \ldots < i_s \leq m$ for the elements of $I$. We define $\mathcal N_{I,\mathbf{\Lambda}}^{\mathbbm h}(k)$ as the subset of $\mathcal N_{E/F}^{\mathbbm h}(k)$ consisting of all the points $(A_m \subset \ldots \subset B_m)$ such that
\begin{itemize}
\item If $i_1 \not = 0$, we have 
\begin{center}
\begin{tikzcd}[column sep=small, row sep=small]
\pi(\Lambda_0^{i_1})_k^{\vee} \arrow[r,symbol=\subset] & \pi B_{i_1}^{\vee} \arrow[d,symbol=\subset, outer sep=2pt,"1"] \arrow[r,symbol=\subset] & \ldots \arrow[r,symbol=\subset] & \pi B_1^{\vee} \arrow[d,symbol=\subset, outer sep=2pt,"1"] \arrow[r,symbol=\subset] & \pi A_1^{\vee} \arrow[d,symbol=\subset, outer sep=2pt,"1"] \arrow[r,symbol=\subset] & \ldots \arrow[r,symbol=\subset] & \pi A_{i_1}^{\vee} \arrow[d,symbol=\subset, outer sep=2pt,"1"] & {} \\
{} & A_{i_1} \arrow[r,symbol=\subset] & \ldots \arrow[r,symbol=\subset] & A_1 \arrow[r,symbol=\subset] & B_1 \arrow[r,symbol=\subset] & \ldots \arrow[r,symbol=\subset] & B_{i_1}  \arrow[r,symbol=\subset] & (\Lambda_0^{i_1})_k
\end{tikzcd}
\end{center} 

\item If $i_s \not = m$, we have 
\begin{center}
\begin{tikzcd}[column sep=small, row sep=small]
\pi^2(\Lambda_1^{i_s})_k^{\vee} \arrow[r,symbol=\subset] & \pi^2 A_{i_s+1}^{\vee} \arrow[d,symbol=\subset, outer sep=2pt,"1"] \arrow[r,symbol=\subset] & \ldots \arrow[r,symbol=\subset] & \pi^2 A_m^{\vee} \arrow[d,symbol=\subset, outer sep=2pt,"1"] \arrow[r,symbol=\subset] & \pi B_m^{\vee} \arrow[d,symbol=\subset, outer sep=2pt,"1"] \arrow[r,symbol=\subset] & \ldots \arrow[r,symbol=\subset] & \pi B_{i_s+1}^{\vee} \arrow[d,symbol=\subset, outer sep=2pt,"1"] & {} \\
{} & \pi B_{i_s+1} \arrow[r,symbol=\subset] & \ldots \arrow[r,symbol=\subset] & \pi B_m \arrow[r,symbol=\subset] & A_m \arrow[r,symbol=\subset] & \ldots \arrow[r,symbol=\subset] & A_{i_s+1}  \arrow[r,symbol=\subset] & (\Lambda_1^{i_s})_k
\end{tikzcd}
\end{center}

\item For all $1 \leq j \leq s-1$, we have \\
\hspace*{-3.3cm}
\begin{tikzcd}[column sep=small, row sep=small]
\pi^2(\Lambda_1^{i_j})_k^{\vee} \arrow[r,symbol=\subset] & \pi^2 A_{i_j+1}^{\vee} \arrow[d,symbol=\subset, outer sep=2pt,"1"] \arrow[r,symbol=\subset] & \ldots \arrow[r,symbol=\subset] & \pi^2 A_{i_{j+1}}^{\vee} \arrow[d,symbol=\subset, outer sep=2pt,"1"] & {} \\
{} & \pi B_{i_j+1} \arrow[r,symbol=\subset] & \ldots \arrow[r,symbol=\subset] & \pi B_{i_{j+1}} \arrow[r,symbol=\subset] & \pi(\Lambda_0^{i_{j+1}})_k
\end{tikzcd}
\begin{tikzcd}[column sep=small, row sep=small]
\pi(\Lambda_0^{i_{j+1}})_k^{\vee} \arrow[r,symbol=\subset] & \pi B_{i_{j+1}}^{\vee} \arrow[d,symbol=\subset, outer sep=2pt,"1"] \arrow[r,symbol=\subset] & \ldots \arrow[r,symbol=\subset] & \pi B_{i_j+1}^{\vee} \arrow[d,symbol=\subset, outer sep=2pt,"1"] & {} \\
{} & A_{i_{j+1}} \arrow[r,symbol=\subset] & \ldots \arrow[r,symbol=\subset] & A_{i_j+1} \arrow[r,symbol=\subset] & (\Lambda_1^{i_j})_k
\end{tikzcd}
\end{itemize}
\end{defi}

\begin{prop}\label{RationalPointsCoverEverything}
We have 
\begin{equation*}
\mathcal N_{E/F}^{\mathbbm h}(k) = \bigcup_{I,\mathbf{\Lambda}}\mathcal N_{I,\mathbf{\Lambda}}^{\mathbbm h}(k),
\end{equation*}
where $(I,\mathbf{\Lambda})$ run over all the Bruhat-Tits indices.
\end{prop}

\begin{proof}
Let $(A_m \subset \ldots \subset B_m)\in \mathcal N_{E/F}^{\mathbbm h}(k)$. Let $I$ be its Bruhat-Tits type. For $i\in I\setminus \{0\}$ and $j\in I\setminus\{m\}$, we put $\Lambda_0^i := \Lambda_{B_i}$ and $\Lambda_1^{j} := \Lambda_{A_{j+1}}$. Let $\mathbf{\Lambda}$ be the collection of all the lattices defined this way. By construction, $(I,\mathbf{\Lambda})$ is a Bruhat-Tits index, and $\mathcal N_{I,\mathbf{\Lambda}}^{\mathbbm h}(k)$ contains the point $(A_m \subset \ldots \subset B_m)$.
\end{proof}

\begin{ex}\label{BTIndexMaximalParahoric}
Assume that $m=1$ so that $\mathbbm h$ consists of a single integer $0\leq h := h_1 \leq n$. A Bruhat-Tits index $(I,\mathbf{\Lambda})$ can be of three sorts:
\begin{itemize}
\item \underline{if $h \not = n$:} $I = \{1\}$ and $\mathbf{\Lambda} = \{\Lambda_0\}$ where $\Lambda_0 \in \mathcal L_0^{\geq h+1}$,
\item \underline{if $h \not = 0$:} $I = \{0\}$ and $\mathbf{\Lambda} = \{\Lambda_1\}$ where $\Lambda_1 \in \mathcal L_1^{\geq n-h+1}$,
\item \underline{if $0<h<n$:} $I = \{0,1\}$ and $\mathbf{\Lambda} = \{\Lambda_0,\Lambda_1\}$ where $\Lambda_0 \in \mathcal L_0^{\geq h+1}, \Lambda_1 \in \mathcal L_1^{\geq n-h+1}$ and $\pi\Lambda_1^{\vee} \subset \Lambda_0$.
\end{itemize}
The set $\mathcal N_{I,\mathbf{\Lambda}}^{\mathbbm h}(k)$ coincide respectively with $\mathcal N_{\Lambda_0}^{h}(k)$, $\mathcal N_{\Lambda_1}^h(k)$ and $\mathcal N_{\Lambda_0,\Lambda_1}^h(k)$ as introduced in Definition \ref{BTStrataForMaximalParahoric}, when $(I,\mathbf{\Lambda})$ runs over the three cases described above. 
\end{ex}

\begin{defi}\label{InclusionBTIndices}
Let $(I,\mathbf{\Lambda})$ and $(I',\mathbf{\Lambda'})$ be two sets Bruhat-Tits indices. We say that $(I',\mathbf{\Lambda'})$ is contained in $(I,\mathbf{\Lambda})$, and we write $(I',\mathbf{\Lambda'}) \subset (I,\mathbf{\Lambda})$ if the following conditions are satisfied:
\begin{itemize}
\item $I \subset I'$,
\item $\forall i \in I\setminus\{0\}$ and $\forall j\in I\setminus\{m\}$, we have $\Lambda_0^{\prime i} \subset \Lambda_0^{i}$ and $\Lambda_1^{\prime j} \subset \Lambda_1^{j}$.
\end{itemize}
\end{defi}

It is clear that $\subset$ defines a partial order on the set of all Bruhat-Tits indices.

\begin{lem}\label{CompletionBTType}
Let $(I,\mathbf{\Lambda})$ be a Bruhat-Tits index. Then there exists a Bruhat-Tits index $(I',\mathbf{\Lambda'}) \subset (I,\mathbf{\Lambda})$ such that 
\begin{equation*}
I' = \begin{cases}
\{0,\ldots , m\} & \text{if } h_1 \not = 0 \text{ and } h_m \not = n,\\
\{1,\ldots , m\} & \text{if } h_1 = 0 \text{ and } h_m \not = n,\\
\{0, \ldots , m-1\} & \text{if } h_1 \not = 0 \text{ and } h_m = n,\\
\{1, \ldots , m-1\} & \text{if } h_1 = 0 \text{ and } h_m = n.
\end{cases}
\end{equation*}
\end{lem}

\begin{proof}
For all $i \in I \setminus\{0\}$ and $j\in I\setminus \{m\}$, we define $\Lambda_0^{'i} := \Lambda_0^i$ and $\Lambda_1^{'j} := \Lambda_1^j$. Let us write $0 \leq i_1 < \ldots < i_s \leq m$ for the elements of $I$. First, assume that there is some $1 \leq j \leq s-1$ such that $i_{j}+1 < i_{j+1}$. For $i_j+1 \leq i \leq i_{j+1}-1$, we must find vertex lattices $\Lambda_0^{\prime i} \in \mathcal L_{0}^{\geq h_i+1}$ and $\Lambda_1^{\prime i} \in \mathcal L_1^{\geq n-h_{i+1}+1}$ which fit inside the following chain of inclusions 
\begin{equation*}
\pi\Lambda_1^{i_j\vee} \subset \Lambda_0^{\prime i_j+1} \subset \pi(\Lambda_1^{\prime i_j+1})^{\vee} \subset \ldots \subset \Lambda_0^{\prime i_{j+1}-1} \subset \pi(\Lambda_1^{\prime i_{j+1}-1})^{\vee} \subset \Lambda_0^{i_{j+1}}.
\end{equation*}
Consider the $\mathbb F_{q^2}$-hermitian space $V := V_{\Lambda_0^{i_{j+1}}}^{0} = \Lambda_0^{i_{j+1}}/\pi\Lambda_0^{i_{j+1}\vee}$ as defined in Section \ref{Section2.2}. We write $d := \dim(V) = t(\Lambda_0^{i_{j+1}}) \geq h_{i_{j+1}}+1$. By Proposition \ref{SubOverVertexLattices}, the space $V$ contains a distinguished subspace $U := \pi \Lambda_1^{i_j\vee}/\pi\Lambda_0^{i_{j+1}\vee}$ which satisfies $U^{\perp} \subset U$ and $\dim(U) = \tfrac{d+n-t(\Lambda_1^{i_j})}{2}$. Indeed, notice that $\pi\Lambda_1^{i_j\vee} \in \mathcal L_0$ and that $t(\pi\Lambda_1^{i_j\vee}) = n-t(\Lambda_1^{i_j})$. Any subspace $U \subset U' \subset V$ will satisfy $U^{\prime \perp} \subset U^{\perp} \subset U \subset U'$, and thus define a unique vertex lattice $\pi\Lambda_1^{i_j\vee} \subset \Lambda' \subset \Lambda_0^{i_{j+1}}$ of type $t(\Lambda') = 2\dim(U') - d$. Therefore, we are reduced to finding subspaces $U_0^{\prime i}, U_1^{\prime i} \subset V$ for all $i_j+1\leq i \leq i_{j+1}-1$, such that $\dim(U_0^{\prime i}) \geq \tfrac{d+h_i+1}{2}$ and $\dim(U_1^{\prime i}) \leq \tfrac{d+h_{i+1}-1}{2}$, which fit inside the following chain of inclusions 
\begin{equation*}
U \subset U_0^{\prime i_j+1} \subset U_1^{\prime i_j+1} \subset \ldots \subset U_0^{\prime i_{j+1}-1} \subset U_1^{\prime i_{j+1}-1} \subset V.
\end{equation*}
Observe that $d > \tfrac{d+h_{i_{j+1}-1}+1}{2}$ and that $\dim(U) < \tfrac{d+h_{i_j+2}-1}{2}$, so that the constraints on the dimensions of $U_0^{\prime i}$ and of $U_1^{\prime i}$ are not absurd. Moreover, since $h_i \leq h_{i+1}-2$ we have 
\begin{equation*}
\frac{d+h_i+1}{2} \leq \frac{d+h_{i+1}-1}{2}.
\end{equation*}
The existence of such subspaces is now straightforward.\\
For now, we have constructed a Bruhat-Tits index of the form $(I',\mathbf{\Lambda'}) \subset (I,\mathbf{\Lambda})$ where $I' = \{i_1,\ldots ,i_s\}$. To proceed further, we distinguish four cases. \\

\underline{If $h_1 \not = 0$:} Assume that $i_1 > 0$. One must find vertex lattices $\Lambda_0^{\prime i} \in \mathcal L_0^{\geq h_i+1}$ for $1 \leq i \leq i_1-1$, and $\Lambda_1^{\prime j} \in \mathcal L_1^{\geq n-h_{j+1}+1}$ for $0 \leq j \leq i_1-1$, which fit inside the following chain of inclusions 
\begin{equation*}
\pi(\Lambda_1^{\prime 0})^{\vee} \subset \Lambda_0^{\prime 1} \subset \pi (\Lambda_1^{\prime 1})^{\vee} \subset \ldots \subset \Lambda_0^{\prime i_1-1} \subset \pi (\Lambda_1^{\prime i_1-1})^{\vee} \subset \Lambda_0^{i_1}.
\end{equation*} 
As before, let us consider the hermitian space $V := V_{\Lambda_0^{i_1}}^0$ of dimension $d := t(\Lambda_0^{i_1}) \geq h_{i_1}+1$. Finding $\Lambda_1^{\prime 0}$ amounts to exhibit a subspace $U \subset V$ such that $U^{\perp} \subset U$ and $\dim(U) \leq \tfrac{d+h_1-1}{2}$. Observe that we have
\begin{equation*}
\frac{d+h_1-1}{2} \geq \left\lceil\frac{d}{2}\right\rceil.
\end{equation*} 
Indeed, if $d$ is even then $h_1$ is odd, so that $h_1 \geq 1$, and if $d$ is odd, $h_1$ is even and non-zero so that $h_1 \geq 2$. Thus, any totally isotropic subspace of dimension $\left\lceil \tfrac{d}{2} \right\rceil$ will do. Exhibiting the remaining $\Lambda_0^{\prime i}, \Lambda_1^{\prime j}$ for $1 \leq i,j \leq i_1-1$ can now be done the same way as above.\\

\underline{If $h_1 = 0$:} Assume that $i_1 > 1$. We must find vertex lattices $\Lambda_0^{\prime i} \in \mathcal L_0^{\geq h_i+1}$ and $\Lambda_1^{\prime j} \in \mathcal L_1^{\geq n - h_{j+1}+1}$ for $1 \leq j \leq i_1-1$, which fit inside the following chain of inclusions
\begin{equation*}
\Lambda_0^{\prime 1} \subset \pi (\Lambda_1^{\prime 1})^{\vee} \subset \ldots \subset \Lambda_0^{\prime i_1-1} \subset \pi (\Lambda_1^{\prime i_1-1})^{\vee} \subset \Lambda_0^{i_1}.
\end{equation*} 
Again, finding $\Lambda_0^{\prime 1}$ amounts to exhibiting a subspace $U \subset V := V_{\Lambda_0^{i_1}}^0$ of dimension $\dim(U) \geq \tfrac{d+1}{2}$, such that $U^{\perp} \subset U$, where $d := \dim(V) \geq h_{i_1}+1$. Since $\tfrac{d+1}{2} \geq \left\lceil\tfrac{d}{2}\right\rceil$, such a subspace exists. Finding the remaining vertex lattices can be done as above. \\

At this stage, we have built a vertex lattice of the form $(I',\mathbf{\Lambda'})$ where $I' = \{0,\ldots,i_s\}$ if $h_1 \not = 0$, and $I' = \{1,\ldots,i_s\}$ if $h_1 = 0$. To extend $I'$ to the right, we proceed as above by distinguishing whether $h_m$ is zero or not. We omit the details.
\end{proof}

\begin{lem}\label{MinimalOrbitType}
Let $(I,\mathbf{\Lambda})$ be a Bruhat-Tits index. Then there exists a Bruhat-Tits index $(I,\mathbf{\Lambda'}) \subset (I,\mathbf{\Lambda})$ such that for all $i\in I\setminus\{0\}$ and $j\in I\setminus\{m\}$, we have $t(\Lambda_0^{\prime i}) = h_i+1$ and $t(\Lambda_1^{\prime j}) = n-h_{j+1}+1$.
\end{lem}

\begin{proof}
Let us write $0 \leq i_1 < \ldots < i_s \leq m$ for the elements of $i$. If $m = 1$, one may choose any vertex lattice $\Lambda_0^{\prime i_1} \subset \Lambda_0^{i_1}$ of orbit type $h_{i_1}+1$ (if $i_1 \not = 0$) and any vertex lattice $\Lambda_1^{\prime i_1} \subset \Lambda_1^{i_1}$ of orbit type $n-h_{i_1+1}+1$ (if $i_1 \not = m$). The existence of such vertex lattices is guaranteed by Proposition \ref{SubOverVertexLattices}, and the inclusion $\Lambda_0^{\prime i_1} \subset \pi(\Lambda_1^{\prime i_1})^{\vee}$ is automatic. \\
Let us now assume that $m\geq 2$, and let $1 \leq t \leq m-1$. One must find vertex lattices $\Lambda_1^{\prime i_t} \subset \Lambda_1^{i_t}$ of orbit type $n-h_{i_t+1}+1$, and $\Lambda_0^{\prime i_{t+1}} \subset \Lambda_0^{i_{t+1}}$ of orbit type $h_{i_{t+1}}+1$ such that $\pi(\Lambda_1^{\prime i_t})^{\vee} \subset \Lambda_0^{\prime i_{t+1}}$. Let us consider the $\mathbb F_{q^2}$-vector space 
\begin{equation*}
W := \Lambda_0^{i_{t+1}}/\pi\Lambda_1^{i_t \vee},
\end{equation*}
which is isomorphic to the quotient of $V := V_{\Lambda_0^{i_{t+1}}}^{0} = \Lambda_0^{i_{t+1}}/\pi\Lambda_0^{i_{t+1}\vee}$ by the subspace $U := \pi\Lambda_1^{i_t \vee} / \pi\Lambda_0^{i_{t+1}\vee}$. Any vector subspace $W' \subset W$ of dimension $\dim(W') = d$ corresponds to a subspace $U \subset U' \subset V$ such that $U^{\prime \perp} \subset U^{\perp} \subset U \subset U'$, and we have $\dim(U') = \dim(U) + d$. By Proposition \ref{SubOverVertexLattices}, we know that $\dim(U) = \tfrac{t(\Lambda_0^{i_{t+1}})+n-t(\Lambda_1^{i_t})}{2}$. Therefore, we have 
\begin{equation*}
\dim(W) = \dim(V) - \dim(U) = \frac{t(\Lambda_0^{i_{t+1}}) + t(\Lambda_1^{i_t}) - n }{2} \geq 0.
\end{equation*}
One may choose two subspaces $W'' \subset W' \subset W$ such that $\dim(W'') = \tfrac{h_{i_t+1}-1+t(\Lambda_1^{i_t})-n}{2} \geq 0$ and $\dim(W') = \tfrac{h_{i_{t+1}}+1+t(\Lambda_1^{i_t})-n}{2} \leq \dim(W)$. From what preceeds and according to Proposition \ref{SubOverVertexLattices}, such a choice gives rise to two vertex lattices $\Lambda'',\Lambda' \in \mathcal L_0$ such that $t(\Lambda'') = h_{i_{t}+1} - 1$ and $t(\Lambda') = h_{i_{t+1}}+1$, satisfying $\pi\Lambda_1^{i_t\vee} \subset \Lambda'' \subset \Lambda' \subset \Lambda_0^{i_{t+1}}$. By taking $\Lambda_0^{\prime i_{t+1}} := \Lambda'$ and $\Lambda_1^{\prime i_t} := \pi\Lambda^{\prime\prime \vee}$, we are done.
\end{proof}

\begin{corol}\label{NotEmpty}
Let $(I,\mathbf{\Lambda})$ be a Bruhat-Tits index. Then $\mathcal N_{I,\mathbf{\Lambda}}^{\mathbbm h}(k) \not = \emptyset$.
\end{corol}

\begin{proof}
By Lemma \ref{CompletionBTType} and Lemma \ref{MinimalOrbitType}, there exists a Bruhat-Tits index $(I',\Lambda') \subset (I,\Lambda)$ such that $I'$ is maximal as a subset of $\{0,\ldots ,m\}$, subject to the constraints of Bruhat-Tits indices, and such that for all $i \in I'\setminus\{0\}$ and $j\in I' \setminus\{m\}$, we have $t(\Lambda_0^{\prime i}) = h_i+1$ and $t(\Lambda_1^{\prime j}) = n-h_{j+1}+1$. Let us first prove that $\mathcal N_{I',\mathbf{\Lambda'}}(k) \not = \emptyset$. First, for $2 \leq i \leq m-1$, if $0 \in I'$ then for $i=1$, and if $m \in I'$ then for $i=m$ as well, one must find lattices $A_i\overset{h_i}{\subset} B_i \subset \mathbb N_{0,k}$ such that 
\begin{align*}
\pi(\Lambda_1^{\prime i-1})^{\vee}_k \subset \pi A_{i}^{\vee} \overset{1}{\subset} B_{i} \subset (\Lambda_0^{\prime i})_k, & & \pi(\Lambda_0^{\prime i})_k^{\vee} \subset \pi B_{i}^{\vee} \overset{1}{\subset} A_{i} \subset (\Lambda_1^{\prime i-1})_k.
\end{align*}
Note that we have $[\Lambda_0^{\prime i}:\pi\Lambda_1^{\prime i-1\vee}] = [\Lambda_1^{\prime i-1} : \pi\Lambda_0^{\prime i \vee}] = 1$. One may take $A_i := (\Lambda_1^{\prime i-1})_k$ and $B_i := (\Lambda_0^{\prime i})_k$. It follows that 
\begin{equation*}
A_i \overset{h_i-1}{\subset} \pi A_i^{\vee} \overset{1}{\subset} B_i,
\end{equation*}
as required.\\
If $0 \not \in I'$, then $h_1 = 0$ and one must find a lattice $B_1$ such that 
\begin{equation*}
\pi(\Lambda_0^1)_k^{\vee} \subset \pi B_1^{\vee} \overset{1}{\subset} B_1 \subset (\Lambda_0^1)_k.
\end{equation*}
Since $t(\Lambda_0^{\prime 1}) = 1$, one may take $B_1 := (\Lambda_0^{\prime 1})_k$. A similar analysis holds when $m \not\in I'$, implying that $h_m = n$. \\
Thus, we have constructed a point $(A_m \subset \ldots \subset B_m) \in \mathcal N_{I',\mathbf{\Lambda'}}(k)$, proving that the right-hand side is not empty. We claim that we have a natural inclusion of sets $\mathcal N_{I',\mathbf{\Lambda'}}(k) \subset \mathcal N_{I,\mathbf{\Lambda}}(k)$. If such an inclusion holds, the proof of the Corollary would be over. We refer to Lemma \ref{InclusionBTStrata} for the proof of this claim.
\end{proof}

\begin{lem}\label{InclusionBTStrata}
Let $(I,\mathbf{\Lambda})$ and $(I',\mathbf{\Lambda'})$ be two Bruhat-Tits indices. If $(I',\mathbf{\Lambda'}) \subset (I,\mathbf{\Lambda})$ then $\mathcal N_{I',\mathbf{\Lambda'}}^{\mathbbm h}(k) \subset \mathcal N_{I,\mathbf{\Lambda}}^{\mathbbm h}(k)$.
\end{lem}

\begin{proof}
Assume that $(I',\mathbf{\Lambda'}) \subset (I,\mathbf{\Lambda})$. If $\mathcal N_{I',\mathbf{\Lambda'}}^{\mathbbm h}(k) = \emptyset$, there is nothing to prove. Else, let $(A_m \subset \ldots \subset B_m) \in \mathcal N_{I',\mathbf{\Lambda'}}^{\mathbbm h}(k)$. Let $i<i'$ be two successive elements in $I$, and write $i = i_1 < \ldots < i_{\ell} = i'$ for the elements of $I'$ lying between $i$ and $i'$. For $1 \leq j \leq \ell - 1$, the diagrams 
\begin{center}
\begin{tikzcd}[column sep=small, row sep=small]
\pi^2(\Lambda_1^{\prime i_j})_k^{\vee} \arrow[r,symbol=\subset] & \pi^2 A_{i_j+1}^{\vee} \arrow[d,symbol=\subset, outer sep=2pt,"1"] \arrow[r,symbol=\subset] & \ldots \arrow[r,symbol=\subset] & \pi^2 A_{i_{j+1}}^{\vee} \arrow[d,symbol=\subset, outer sep=2pt,"1"] & {} \\
{} & \pi B_{i_j+1} \arrow[r,symbol=\subset] & \ldots \arrow[r,symbol=\subset] & \pi B_{i_{j+1}} \arrow[r,symbol=\subset] & \pi(\Lambda_0^{\prime i_{j+1}})_k
\end{tikzcd}
\end{center}
can be put together by using the inclusions $\pi(\Lambda_0^{\prime i_{j}})_k \subset \pi^2(\Lambda_1^{\prime i_{j}})^{\vee}_k$. This gives rise to the following diagram 
\begin{center}
\begin{tikzcd}[column sep=small, row sep=small]
\pi^2(\Lambda_1^i)_k^{\vee} \arrow[r,symbol=\subset] & \pi^2(\Lambda_1^{\prime i})_k^{\vee} \arrow[r,symbol=\subset] & \pi^2 A_{i+1}^{\vee} \arrow[d,symbol=\subset, outer sep=2pt,"1"] \arrow[r,symbol=\subset] & \ldots \arrow[r,symbol=\subset] & \pi^2 A_{i'}^{\vee} \arrow[d,symbol=\subset, outer sep=2pt,"1"] & {} & {}\\
{} & {} & \pi B_{i+1} \arrow[r,symbol=\subset] & \ldots \arrow[r,symbol=\subset] & \pi B_{i'} \arrow[r,symbol=\subset] & \pi(\Lambda_0^{\prime i'})_k \arrow[r,symbol=\subset] & \pi(\Lambda_0^{i'})_k
\end{tikzcd}
\end{center}
All the other cases can be treated similarly, and show that the desired inclusion holds. We omit the details.
\end{proof}

\begin{defi}\label{IntersectionBTIndices}
Let $(I,\mathbf{\Lambda})$ and $(I',\mathbf{\Lambda'})$ be two Bruhat-Tits indices. Assume that we have
\begin{align*}
\forall i \in (I\cap I') \setminus \{0\}, & & & \Lambda_0^i \cap \Lambda^{\prime i}_0 \in \mathcal L_0^{\geq h_i+1},\\
\forall j \in (I\cap I') \setminus \{m\}, & & & \Lambda_1^j \cap \Lambda_1^{\prime j} \in \mathcal L_1^{\geq n-h_{j+1}+1},\\
\forall i_1 \in I, i_2 \in I', & & & \pi \Lambda_1^{i_1\vee} \subset \Lambda_0^{\prime i_2} \text{ if } i_1<i_2, \text{ and } \pi \Lambda_1^{\prime i_2 \vee} \subset \Lambda_0^{i_1} \text{ if } i_2 < i_1. 
\end{align*}  
We define the intersection $(I,\mathbf{\Lambda})\cap (I',\mathbf{\Lambda'}):= (I\cup I',\mathbf{\Lambda''})$ where $\mathbf{\Lambda''}$ is the collection of all the following vertex lattices, for $i \in I\cup I' \setminus\{0\}$ and $j \in I\cup I' \setminus\{m\}$,
\begin{equation*}
\begin{array}{lclc}
\underline{\text{If } i\in I \setminus I':} & \Lambda_0^i \hspace*{3cm} & \underline{\text{If } j\in I\setminus I':} & \Lambda_1^{j},\\
\underline{\text{If } i\in I' \setminus I:} & \Lambda_0^{\prime i} \hspace*{3cm} & \underline{\text{If } j\in I'\setminus I:} & \Lambda_1^{\prime j},\\
\underline{\text{If } i \in I \cap I':} & \Lambda_0^i \cap \Lambda_0^{\prime i} \hspace*{3cm} & \underline{\text{If } j\in I\cap I':} &  \Lambda_1^j\cap \Lambda_1^{\prime j}.
\end{array}
\end{equation*}
\end{defi}

It is clear that the intersection $(I,\mathbf{\Lambda})\cap (I',\mathbf{\Lambda'})$ defined above is again a Bruhat-Tits index.

\begin{prop}\label{IntersectionBTStrata}
Let $(I,\mathbf{\Lambda})$ and $(I',\mathbf{\Lambda'})$ be two Bruhat-Tits indices. Assume that the intersection $(I,\mathbf{\Lambda})\cap (I',\mathbf{\Lambda'})= (I\cup I',\mathbf{\Lambda''})$ is well-defined. We have 
\begin{equation*}
\mathcal N_{I,\mathbf{\Lambda}}^{\mathbbm h}(k) \cap \mathcal N_{I',\mathbf{\Lambda'}}^{\mathbbm h}(k) = \mathcal N_{I\cup I',\mathbf{\Lambda''}}^{\mathbbm h}(k).
\end{equation*} 
Furthermore, if the intersection is not defined, then $\mathcal N_{I,\mathbf{\Lambda}}^{\mathbbm h}(k) \cap \mathcal N_{I',\mathbf{\Lambda'}}^{\mathbbm h}(k) = \emptyset$.
\end{prop}

\begin{proof}
Let $(A_m\subset \ldots \subset B_m) \in \mathcal N_{I,\mathbf{\Lambda}}^{\mathbbm h}(k) \cap \mathcal N_{I',\mathbf{\Lambda'}}^{\mathbbm h}(k)$ (which we assume is non empty). Let us first prove that the intersection $(I,\mathbf{\Lambda})\cap (I',\mathbf{\Lambda'})$ is well-defined. 
\begin{itemize}
\item Let $i\in (I \cap I') \setminus \{0\}$. By construction, we have 
\begin{equation*}
\left.\begin{array}{r}
\pi(\Lambda_0^{i})^{\vee}_k \\
\pi(\Lambda_0^{'i})^{\vee}_k
\end{array}
\right\} \subset \pi B_{i}^{\vee} \overset{1}{\subset} A_i \overset{h_i}{\subset} B_i \subset \left\{ \begin{array}{l}
(\Lambda_0^{i})_k,\\
(\Lambda_0^{'i})_k.
\end{array}\right.
\end{equation*}
It follows that 
\begin{equation*}
\pi((\Lambda_0^i)_k\cap (\Lambda_0^{'i})_k)^{\vee} = \pi(\Lambda_0^{i})^{\vee}_k + \pi(\Lambda_0^{'i})^{\vee}_k \subset \pi B_i^{\vee} \overset{1}{\subset} A_i \overset{h_i}{\subset} B_i \subset (\Lambda_0^i)_k\cap (\Lambda_0^{'i})_k.
\end{equation*}
Since the inclusion $\Lambda_0^i\cap \Lambda_0^{'i} \subset (\Lambda_0^i\cap \Lambda_0^{'i})^{\vee}$ is obvious, we have $\Lambda_0^i\cap \Lambda_0^{'i} \in \mathcal L_0^{\geq h_i+1}$.
 
\item Let $j \in (I\cap I') \setminus \{m\}$. By construction, we have 
\begin{equation*}
\left.\begin{array}{r}
\pi^2(\Lambda_1^{j})^{\vee}_k \\
\pi^2(\Lambda_1^{'j})^{\vee}_k
\end{array}
\right\} \subset \pi^2 A_{j+1}^{\vee} \overset{1}{\subset} \pi B_{j+1} \overset{n-h_{j+1}}{\subset} A_{j+1} \subset \left\{ \begin{array}{l}
(\Lambda_1^{j})_k,\\
(\Lambda_1^{'j})_k.
\end{array}\right.
\end{equation*}
It follows that 
\begin{equation*}
\pi^2((\Lambda_1^j)_k\cap (\Lambda_1^{'j})_k)^{\vee} = \pi^2(\Lambda_1^{j})^{\vee}_k + \pi^2(\Lambda_1^{'j})^{\vee}_k \subset \pi^2 A_{j+1}^{\vee} \overset{1}{\subset} \pi B_{j+1} \overset{n-h_{j+1}}{\subset} A_{j+1}  \subset (\Lambda_1^j)_k\cap (\Lambda_1^{'j})_k.
\end{equation*}
Since the inclusion $\Lambda_1^j\cap \Lambda_1^{'j} \subset \pi(\Lambda_1^j\cap \Lambda_1^{'j})^{\vee}$ is obvious, we have $\Lambda_1^j\cap \Lambda_1^{'j} \in \mathcal L_1^{\geq n-h_{j+1}+1}$. 

\item Let $i_1,i_2 \in I\cup I'$ with $i_1 < i_2$. We have 
\begin{equation*}
\left.\begin{array}{r}
\pi (\Lambda_1^{i_1})^{\vee}_k \\
\pi (\Lambda_1^{'i_1})^{\vee}_k
\end{array} \right\}
\subset \pi A_{i_1+1}^{\vee} \subset B_{i_1+1} \subset \ldots \subset B_{i_2} \subset 
\left\{ \begin{array}{l}
(\Lambda_0^{i_2})_k,\\
(\Lambda_0^{'i_2})_k.
\end{array}\right.
\end{equation*}
\end{itemize}
This proves that the intersection $(I,\mathbf{\Lambda})\cap (I',\mathbf{\Lambda'})= (I\cup I',\mathbf{\Lambda''})$ is well-defined.\\
Next, assume that the intersection $(I,\mathbf{\Lambda})\cap (I',\mathbf{\Lambda'})= (I\cup I',\mathbf{\Lambda''})$ is well-defined. By construction, we have inclusions of Bruhat-Tits indices $(I\cup I',\mathbf{\Lambda''}) \subset (I,\mathbf{\Lambda})$ and $(I\cup I',\mathbf{\Lambda''}) \subset (I',\mathbf{\Lambda'})$. By Lemma \ref{InclusionBTStrata}, we have $\mathcal N_{I\cup I',\mathbf{\Lambda'}'}^{\mathbbm h}(k) \subset \mathcal N_{I,\mathbf{\Lambda}}^{\mathbbm h}(k) \cap \mathcal N_{I',\mathbf{\Lambda'}}^{\mathbbm h}(k)$. Moreover, by Corollary \ref{NotEmpty}, we know that the left-hand side is not empty. Thus, the subsets $\mathcal N_{I,\mathbf{\Lambda}}^{\mathbbm h}(k)$ and $\mathcal N_{I',\mathbf{\Lambda'}}^{\mathbbm h}(k)$ meet non-trivially. Eventually, let $(A_m \subset \ldots \subset B_m) \in \mathcal N_{I,\mathbf{\Lambda}}^{\mathbbm h}(k) \cap \mathcal N_{I',\mathbf{\Lambda'}}^{\mathbbm h}(k)$. If $i \in I \cap I'$, we have 
\begin{align*}
\pi^2 ((\Lambda_1^{i})_k \cap (\Lambda_1^{'i})_k)^{\vee} & \subset \pi^2 A_{i+1}^{\vee}, & \pi((\Lambda_0^i)_k \cap (\Lambda_0^{'i})_k)^{\vee} & \subset \pi B_i^{\vee}, \\
A_{i+1} & \subset (\Lambda_1^{i})_k \cap (\Lambda_1^{'i})_k, & B_i & \subset (\Lambda_0^{i})_k \cap (\Lambda_0^{'i})_k.
\end{align*}
Using these facts, it is straightforward to check that the point $(A_m \subset \ldots \subset B_m)$ is contained in $\mathcal N_{I\cup I',\mathbf{\Lambda''}}^{\mathbbm h}(k)$. This concludes the proof.
\end{proof}

\begin{corol}
Let $(I,\mathbf{\Lambda})$ and $(I',\mathbf{\Lambda'})$ be two Bruhat-Tits indices. We have 
\begin{equation*}
(I',\mathbf{\Lambda'}) \subset (I,\mathbf{\Lambda}) \iff \mathcal N_{I',\mathbf{\Lambda'}}^{\mathbbm h}(k) \subset \mathcal N_{I,\mathbf{\Lambda}}^{\mathbbm h}(k).
\end{equation*}
\end{corol}

\begin{proof}
This is a generalization of Corollary \ref{InclusionBTStrata}. We only need to prove the reverse implication. Assume that $\mathcal N_{I',\mathbf{\Lambda'}}^{\mathbbm h}(k) \subset \mathcal N_{I,\mathbf{\Lambda}}^{\mathbbm h}(k)$. It follows that the intersection $\mathcal N_{I',\mathbf{\Lambda'}}^{\mathbbm h}(k) \cap \mathcal N_{I,\mathbf{\Lambda}}^{\mathbbm h}(k) = \mathcal N_{I',\mathbf{\Lambda'}}^{\mathbbm h}(k)$ is not empty. Thus, the intersection of $(I,\mathbf{\Lambda})$ and of $(I',\mathbf{\Lambda'})$ is well-defined, and we have 
\begin{equation*}
(I,\mathbf{\Lambda})\cap (I',\mathbf{\Lambda'})= (I',\mathbf{\Lambda'}).
\end{equation*}
We deduce that $I \subset I'$. Moreover, for all $i\in I\setminus\{0\}$ and $j\in I\setminus\{m\}$, we have $\Lambda_0^i \cap \Lambda_0^{\prime i} = \Lambda_0^{\prime i}$ and $\Lambda_1^j \cap \Lambda_1^{\prime j} = \Lambda_1^{\prime j}$. It follows that $(I',\mathbf{\Lambda'}) \subset (I,\mathbf{\Lambda})$ as desired. 
\end{proof}

\section{Subschemes $\mathcal N_{I,\mathbf{\Lambda}}^{\mathbbm h}$ of $\mathcal N_{E/F,\mathrm{red}}^{\mathbbm h}$ and Deligne-Lusztig varieties}

\subsection{Subschemes $\mathcal N_{I,\mathbf{\Lambda}}^{\mathbbm h}$ attached to Bruhat-Tits indices $(I,\mathbf{\Lambda})$} \label{SectionClosedSubschemes}

We keep the notations from the previous Section. In particular, recall the isogenies $\alpha_{h',h}$ defined for $0 \leq h < h' \leq n$. It induces an isometry between  $(\mathbb N_0,\{\cdot,\cdot\}_{[h']})$ and $(\mathbb N_0,\{\cdot,\cdot\}_{[h]})$. Let $(I,\mathbf{\Lambda})$ be a Bruhat-Tits index. Following \cite{vw2} and \cite{cho}, we define lattices as follows. For $i \in I\setminus \{0\}$ and $j\in I\setminus \{m\}$,
\begin{align*}
\Lambda_0^{i+} & := \alpha_{h_i,h_1}^{-1}\left(\Lambda_0^i \oplus \mathcal V^{-1}(\Lambda_0^i)\right), & \Lambda_1^{j+} & := \alpha_{h_{j+1},h_1}^{-1}\left(\Lambda_1^j \oplus \mathcal V^{-1}(\Lambda_1^{j})\right),\\
\Lambda_0^{i-} & := \alpha_{h_i,h_1}^{-1}\left(\pi\Lambda_0^{i\vee} \oplus \mathcal V(\Lambda_0^{i\vee})\right), & \Lambda_1^{j-} & := \alpha_{h_{j+1},h_1}^{-1}\left(\pi^2\Lambda_1^{j\vee} \oplus \pi \mathcal V(\Lambda_1^{j\vee})\right).
\end{align*}

Then $\Lambda_0^{i\pm},\Lambda_1^{j\pm}$ are $W_{\mathcal O_F}(\mathbb F_{p^2})$-lattices in $\mathbb N_{\kappa_{\breve E}}^{\tau}$, which are stable by the $\mathcal O_E$-action, by $\mathcal F$ and by $\mathcal V$. Moreover, the pairings $\pi\langle\cdot,\cdot\rangle_{[h_i]}$ and $\langle\cdot,\cdot\rangle_{[h_{j+1}]}$ induce a pairing on $\Lambda_0^{i\pm}$ and on $\Lambda_1^{j\pm}$ respectively. Therefore, the lattices $\Lambda_0^{i\pm}$ and $\Lambda_1^{j\pm}$, seen as relative Dieudonné modules with extra structures, determine strict formal $\mathcal O_F$-modules $X_{\Lambda_0^{i\pm}}$ and $X_{\Lambda_1^{j\pm}}$ over $\mathbb F_{q^2}$ with $\mathcal O_E$-action and $\mathcal O_E$-linear polarizations. Moreover, the inclusions $(\Lambda_0^{i\pm})_{\kappa_{\breve E}},(\Lambda_1^{j\pm})_{\kappa_{\breve E}} \hookrightarrow \mathbb N_{\kappa_{\breve E}}$ induce quasi-isogenies
\begin{align*}
\rho_{\Lambda_0^{i\pm}}:X_{\Lambda_0^{i\pm}}\times \kappa_{\breve E} \to \mathbb X^{[h_i]}\times \kappa_{\breve E}, & & \rho_{\Lambda_1^{j\pm}}: X_{\Lambda_1^{j\pm}}\times \kappa_{\breve E} \to \mathbb X^{[h_{j+1}]} \times \kappa_{\breve E},
\end{align*}
which are compatible with the additional structures. Besides, the compositions
\begin{align*}
\rho_{\Lambda_0^{i+}}^{-1}\circ \rho_{\Lambda_0^{i-}}&: X_{\Lambda_0^{i-}} \times \kappa_{\breve E} \to X_{\Lambda_0^{i+}} \times \kappa_{\breve E},\\
\rho_{\Lambda_1^{j+}}^{-1}\circ \rho_{\Lambda_1^{j-}}&: X_{\Lambda_1^{j-}} \times \kappa_{\breve E} \to X_{\Lambda_1^{j+}} \times \kappa_{\breve E},
\end{align*}
coincide with the isogenies induced by the inclusions $\Lambda_0^{i-} \subset \Lambda_0^{i+}$ and $\Lambda_1^{j-} \subset \Lambda_1^{j+}$ respectively. Moreover, observe that we have $\Lambda_0^{i+} = (\Lambda_0^{i-})^{\dagger}$ and $\Lambda_1^{j-} = \pi (\Lambda_1^{j+})^{\dagger}$, where $\cdot^{\dagger}$ denotes respectively the dual for $\langle\cdot,\cdot\rangle_{[h_i]}$ and for $\langle\cdot,\cdot\rangle_{[h_{j+1}]}$. It follows that there exists isomorphisms $\mu_i: X_{\Lambda_0^{i+}} \xrightarrow{\sim} X_{\Lambda_0^{i-}}^{\vee}$ and $\nu_j: X_{\Lambda_1^{j-}} \xrightarrow{\sim} X_{\Lambda_1^{j+}}^{\vee}$ making the following two diagrams commute.
\begin{align*}
\begin{tikzcd}[ampersand replacement=\&,column sep=1.5cm]
X_{\Lambda_0^{i+}} \times \kappa_{\breve E} \arrow[r,"\mu_i"] \arrow[d,swap,"\rho_{\Lambda_0^{i+}}"] \& X_{\Lambda_0^{i-}}^{\vee} \times \kappa_{\breve E} \\
\mathbb X^{[h_i]} \times \kappa_{\breve E} \arrow[r,swap,"\lambda_{\mathbb X}^{[h_i]}"] \& \mathbb X^{[h_{i}]\vee}\times \kappa_{\breve E} \arrow[u,swap,"\rho_{\Lambda_0^{i-}}^{\vee}"] 
\end{tikzcd} & &
\begin{tikzcd}[ampersand replacement=\&,column sep=1.5cm]
X_{\Lambda_1^{j-}} \times \kappa_{\breve E} \arrow[r,"\pi\nu_j"] \arrow[d,swap,"\rho_{\Lambda_1^{j-}}"] \& X_{\Lambda_1^{j+}}^{\vee} \times \kappa_{\breve E} \\
\mathbb X^{[h_{j+1}]} \times \kappa_{\breve E} \arrow[r,swap,"\lambda_{\mathbb X}^{[h_{j+1}]}"] \& \mathbb X^{[h_{j+1}]\vee}\times \kappa_{\breve E} \arrow[u,swap,"\rho_{\Lambda_1^{j+}}^{\vee}"] 
\end{tikzcd}
\end{align*}

Let us now fix a $\kappa_{\breve E}$-scheme $S$ and a point $X := (X^{[i]},i_{X^{[i]}},\lambda_{X^{[i]}},\rho_{X^{[i]}})_{1\leq i \leq m} \in \mathcal N_{E/F}^{\mathbbm h}(S)$. We define certain quasi-isogenies as follows.
\begin{center}
\begin{tikzcd}[row sep = small]
\forall i\in I\setminus\{0\}, & \rho_{X,\Lambda_0^{i+}}:X^{[i]} \arrow[r,"\rho_{X^{[i]}}"] & \mathbb X_{S}^{[h_{i}]} \arrow[r,"(\rho_{\Lambda_0^{i+}})_S^{-1}"] &  (X_{\Lambda_0^{i+}})_S, \\
& \rho_{\Lambda_0^{i-},X}: (X_{\Lambda_0^{i-}})_S \arrow[r, "(\rho_{\Lambda_0^{i-}})_S"] & \mathbb X_S^{[h_{i}]}  \arrow[r,"\rho_{X^{[i]}}^{-1}"] & X^{[i]},
\end{tikzcd}
\end{center}
\begin{center}
\begin{tikzcd}[row sep = small]
\forall j\in I\setminus\{m\}, & \rho_{X,\Lambda_1^{j+}} : X^{[j+1]} \arrow[r,"\rho_{X^{[j+1]}}"] & \mathbb X_{S}^{[h_{j+1}]} \arrow[r,"(\rho_{\Lambda_1^{j+}})_S^{-1}"]  & (X_{\Lambda_1^{j+}})_S, \\
& \rho_{\Lambda_1^{j-},X}: (X_{\Lambda_1^{j-}})_S \arrow[r,"(\rho_{\Lambda_1^{j-}})_S"] & \mathbb X_S^{[h_{j+1}]} \arrow[r,"\rho_{X^{[j+1]}}^{-1}"] & X^{[j+1]}.
\end{tikzcd}
\end{center}
We define a subfunctor $\mathcal N_{I,\mathbf{\Lambda}}^{\mathbbm h}$ of $\mathcal N_{E/F}^{\mathbbm h} \otimes \kappa_{\breve E}$ by assigning, for all $\kappa_{\breve E}$-scheme $S$, the subset of points $X := (X^{[i]},i_{X^{[i]}},\lambda_{X^{[i]}},\rho_{X^{[i]}})_{1\leq i \leq m} \in \mathcal N_{E/F}^{\mathbbm h}(S)$ such that $\rho_{\Lambda_0^{i-},X}$ and $\rho_{X,\Lambda_1^{j+}}$ are isogenies for all $i\in I\setminus\{0\}$ and for all $j\in I\setminus\{m\}$. 

\begin{prop}\label{ClosedBTStrataAsSubschemes}
The functor $\mathcal N_{I,\mathbf{\Lambda}}^{\mathbbm h}$ is representable by a projective closed subscheme of $\mathcal N_{E/F,\mathrm{red}}^{\mathbbm h} := (\mathcal N_{E/F}^{\mathbbm h} \otimes \kappa_{\breve E})_{\mathrm{red}}$.
\end{prop}

\begin{proof}
The argument is classical, see \cite{vw2} Lemma 4.2 and \cite{RZ} Proposition 2.9.
\end{proof}

\begin{ex}
Assume that $m = 1$, so that $\mathbbm h$ consists of a single integer $h := h_1$. Recall from Example \ref{BTIndexMaximalParahoric} that there are three types of Bruhat-Tits indices $(I,\mathbf{\Lambda})$:
\begin{itemize}
\item if $h \not = n$, $I = \{1\}$ and $\mathbf{\Lambda} = \{\Lambda_0\}$ for some $\Lambda_0 \in \mathcal L_1^{\geq h+1}$, then $\mathcal N_{I,\mathbf{\Lambda}}^{\mathbbm h}$ consists of those points $X$ such that $\rho_{\Lambda_0^{-},X}$ is an isogeny,
\item if $h\not = 0$, $I = \{0\}$ and $\mathbf{\Lambda} = \{\Lambda_1\}$ for some $\Lambda_1 \in \mathcal L_1^{\geq n-h+1}$, then $\mathcal N_{I,\mathbf{\Lambda}}^{\mathbbm h}$ consists of those points $X$ such that $\rho_{X,\Lambda_1^{+}}$ is an isogeny,
\item if $0 < h < n$, $I=\{0,1\}$ and $\mathbf{\Lambda} = \{\Lambda_0,\Lambda_1\}$ for some $\Lambda_0 \in \mathcal L_1^{\geq h+1}$ and $\Lambda_1 \in \mathcal L_1^{\geq n-h+1}$ such that $\pi\Lambda_1^{\vee}\subset \Lambda_0$, then $\mathcal N_{I,\mathbf{\Lambda}}^{\mathbbm h}$ consists of those points $X$ such that $\rho_{\Lambda_0^{-},X}$ and $\rho_{X,\Lambda_1^{+}}$ are isogenies.
\end{itemize}
When $\#I = 1$ and $\mathbf{\Lambda} = \{\Lambda\}$ for some $\Lambda \in \mathcal L_0^{\geq h+1} \sqcup \mathcal L_1^{\geq n-h+1}$, a Bruhat-Tits stratum $\mathcal N_{\Lambda}^h$ has been defined in \cite{cho} as the closed subscheme of $\mathcal N_{E/F,\mathrm{red}}^{h}$ classifying those points $X$ such that $\rho_{X,\Lambda^+}$ \textbf{and} $\rho_{\Lambda^-,X}$ are isogenies. A priori, this differs from our definition as we only ask for one of the two quasi-isogenies to be an actual isogeny. However, both definitions coincide in virtue of Lemma \ref{OneIsogenyTheOtherToo} below.
\end{ex} 

\begin{lem}\label{OneIsogenyTheOtherToo}
Let $(I,\mathbf{\Lambda})$ be a Bruhat-Tits index and let $X := (X^{[i]},i_{X^{[i]}},\lambda_{X^{[i]}},\rho_{X^{[i]}})_{1\leq i \leq m} \in \mathcal N_{E/F}^{\mathbbm h}(S)$ where $S$ is any scheme over $\kappa_{\breve E}$. Let $i \in I\setminus \{0\}$ and let $j\in I\setminus \{m\}$. 
\begin{enumerate}
\item If $\rho_{\Lambda_0^{i-},X}$ is an isogeny, then for all $1 \leq i' \leq i$, there exists an isogeny $f:X^{[i']} \to (X_{\Lambda_0^{i+}})_S$ such that $\rho_{X,\Lambda_0^{i+}} = f\circ \widetilde{\alpha}_{i',i}$.
\item If $\rho_{X,\Lambda_1^{j+}}$ is an isogeny, then for all $j \leq j' \leq m-1$, there exists an isogeny $g:(X_{\Lambda_1^{j-}})_S \to X^{[j'+1]}$ such that $\rho_{\Lambda_1^{j-},X} = \widetilde{\alpha}_{j'+1,j+1}\circ g$.
\end{enumerate}
\end{lem}

Here, given $1 \leq a \leq b \leq m$, we write $\widetilde{\alpha}_{b,a} := \widetilde{\alpha}_{a+1,a} \circ \widetilde{\alpha}_{a+2,a+1} \circ \ldots \circ \widetilde{\alpha}_{b,b-1}: X^{[b]} \to X^{[a]}$. In particular $\widetilde{\alpha}_{a,a} = \mathrm{id}$. 

\begin{proof}
For 1., we define $f$ as the composition
\begin{equation*}
f:X^{[i']} \xrightarrow{\rho_{X^{[i']}}} \mathbb X_S^{[h_{i'}]} \xrightarrow{(\alpha_{h_i,h_{i'}})_S^{-1}} \mathbb X_S^{[h_i]} \xrightarrow{(\rho_{\Lambda_0^{i+}})^{-1}_S} (X_{\Lambda_0^{i+}})_S.
\end{equation*}
By the definition of $\mathcal N_{E/F}^{\mathbbm h}$, we have $(\alpha_{h_i,h_{i'}})^{-1}_S \circ \rho_{X^{[i']}} =  \rho_{X^{[i]}}\circ (\widetilde{\alpha}_{i',i})^{-1}$, so that we clearly have $\rho_{X,\Lambda_0^{i+}} = f\circ \widetilde{\alpha}_{i',i}$. Thus it remains to show that $f$ is an isogeny. We have 
\begin{align*}
f & = (\rho_{\Lambda_0^{i+}})^{-1}_S \circ (\alpha_{h_i,h_{i'}})_S^{-1} \circ \rho_{X^{[i']}} \\
& = (\mu_i)_S^{-1}\circ (\rho_{\Lambda_0^{i-}}^{\vee})_S \circ (\lambda^{[h_i]}_{\mathbb X})_S \circ (\alpha_{h_i,h_{i'}})_S^{-1} \circ \rho_{X^{[i']}}.
\end{align*}
By compatibility between the isogeny $\alpha_{h_i,h_{i'}}$ and the polarizations, we have $(\lambda^{[h_i]}_{\mathbb X})_S \circ (\alpha_{h_i,h_{i'}})_S^{-1} = (\alpha_{h_i,h_{i'}})_S^{\vee} \circ (\lambda_{\mathbb X}^{[h_{i'}]})_S$. Moreover, since $X^{[i']} \in \mathcal N_{E/F}^{h_{i'}}(S)$, we have $(\lambda_{\mathbb X}^{[h_{i'}]})_S \circ \rho_{X^{[i']}} = c(\rho_{X^{[i']}}^{\vee})^{-1} \circ \lambda_{X^{[i']}}$ for some unit scalar $c \in \mathcal O_F^{\times}$. Thus, we have
\begin{align*}
f & = c(\mu_i)_S^{-1}\circ (\rho_{\Lambda_0^{i-}}^{\vee})_S \circ (\alpha_{h_i,h_{i'}})_S^{\vee} \circ (\rho_{X^{[i']}}^{\vee})^{-1} \circ \lambda_{X^{[i']}}\\
& = c(\mu_i)_S^{-1}\circ (\rho_{\Lambda_0^{i-}}^{\vee})_S \circ (\rho_{X^{[i]}}^{\vee})^{-1}\circ \widetilde{\alpha}_{i',i}^{\vee} \circ \lambda_{X^{[i']}}\\
& = c(\mu_i)_S^{-1} \circ \rho_{\Lambda_0^{i-},X}^{\vee} \circ \widetilde{\alpha}_{i',i}^{\vee} \circ \lambda_{X^{[i']}}.
\end{align*}
By hypothesis, $\rho_{\Lambda_0^{i-},X}$ is an isogeny, so its dual $\rho_{\Lambda_0^{i-},X}^{\vee}$ is an isogeny as well. It follows that $f$ is also an isogeny.\\
For 2., we define $g$ by the composition
\begin{equation*}
g: (X_{\Lambda_1^{j-}})_S  \xrightarrow{(\rho_{\Lambda_1^{j-}})_S} \mathbb X_S^{[h_{j+1}]} \xrightarrow{(\alpha_{h_{j'+1},h_{j+1}})_S^{-1}} \mathbb X_S^{[h_{j'+1}]} \xrightarrow{\rho_{X^{[j'+1]}}^{-1}} X^{[j'+1]}.
\end{equation*}
As in 1., we have $\rho_{X^{[j'+1]}}^{-1} \circ (\alpha_{h_{j'+1},h_{j+1}})_S^{-1} = \widetilde{\alpha}_{j'+1,j+1}^{-1} \circ \rho_{X^{[j+1]}}^{-1}$, so that we clearly have $\rho_{\Lambda_1^{j-},X} = \widetilde{\alpha}_{j'+1,j+1}\circ g$. It remains to prove that $g$ is an isogeny. We have 
\begin{align*}
g & = \rho_{X^{[j'+1]}}^{-1} \circ (\alpha_{h_{j'+1},h_{j+1}})_S^{-1} \circ (\rho_{\Lambda_1^{j-}})_S \\
& = \rho_{X^{[j'+1]}}^{-1} \circ (\alpha_{h_{j'+1},h_{j+1}})_S^{-1} \circ (\lambda_{\mathbb X}^{[h_{j+1}]})^{-1}_S \circ (\rho_{\Lambda_1^{j+}}^{\vee})^{-1}_S \circ \pi(\nu_j)_S.
\end{align*}
Now, we have $(\alpha_{h_{j'+1},h_{j+1}})_S^{-1} \circ (\lambda_{\mathbb X}^{[h_{j+1}]})^{-1}_S = (\lambda_{\mathbb X}^{[h_{j'+1}]})_S^{-1} \circ (\alpha_{h_{j'+1},h_{j+1}})_S^{\vee}$ and $\rho_{X^{[j'+1]}}^{-1} \circ (\lambda_{\mathbb X}^{[h_{j'+1}]})_S^{-1} = c \lambda_{X^{j'+1}}^{-1}\circ \rho_{X^{[j'+1]}}^{\vee}$ for some unit scalar $c\in \mathcal O_F^{\times}$. Thus we have 
\begin{align*}
g & = c \lambda_{X^{j'+1}}^{-1}\circ \rho_{X^{[j'+1]}}^{\vee} \circ (\alpha_{h_{j'+1},h_{j+1}})_S^{\vee} \circ (\rho_{\Lambda_1^{j+}}^{\vee})^{-1}_S \circ \pi(\nu_j)_S\\
& = c \lambda_{X^{j'+1}}^{-1}\circ \widetilde{\alpha}^{\vee}_{j'+1,j+1}\circ \rho_{X^{[j+1]}}^{\vee} \circ (\rho_{\Lambda_1^{j+}}^{\vee})^{-1}_S \circ \pi(\nu_j)_S\\
& = c \lambda_{X^{j'+1}}^{-1}\circ \widetilde{\alpha}^{\vee}_{j'+1,j+1}\circ \rho_{X,\Lambda_1^{j+}}^{\vee} \circ \pi(\nu_j)_S.
\end{align*}
Since $\mathrm{Ker}(\lambda_{X^{[j+1]}}) \subset X[\pi]$, we know that $\pi\lambda_{X^{[j+1]}}^{-1}$ is an isogeny. Therefore, if $\rho_{X,\Lambda_1^{j+}}$ is an isogeny, so is its dual and so is $g$.
\end{proof}

\begin{prop}\label{EquivalenceClosedBTStrata}
Let $(I,\mathbf{\Lambda})$ be a Bruhat-Tits index, and let $k$ be an algebraically closed field containing $\kappa_{\breve E}$. The set of $k$-rational points of the closed subscheme $\mathcal N_{I,\mathbf{\Lambda}}^{\mathbbm h}$ coincide with the sets $\mathcal N_{I,\mathbf{\Lambda}}^{\mathbbm h}(k)$ defined in Definition \ref{DefinitionBTStrata}.
\end{prop}

\begin{proof}
Let $(A_m \subset \ldots \subset B_m) \in \mathcal N_{E/F}^{\mathbbm h}(k)$ be a point. For all $1\leq i \leq m$, let $M_i := \alpha_{h_i,h_1}^{-1}(A_1 \oplus B_1^{\dagger})$ be the lattice of $\mathbb N_{k}$ corresponding to the point $( \alpha_{h_i,h_1}^{-1}(A_i) \overset{h_i}{\subset} \alpha_{h_i,h_1}^{-1}(B_i)) \in \mathcal N_{E/F}^{h_i}(k)$ as in Theorem \ref{RationalPointsMaximalParahoric}, where $\cdot^{\dagger}$ denotes the dual with respect to $\langle\cdot,\cdot\rangle_{[h_1]}$. By construction, for $i\in I\setminus\{0\}$ and for $j \in I\setminus\{m\}$, we have 
\begin{align*}
\rho_{\Lambda_0^{i-},X} \text{ is an isogeny} & \iff (\Lambda_0^{i-})_k \subset M_i,\\
\rho_{X,\Lambda_1^{j+}} \text{ is an isogeny} & \iff M_{j+1} \subset (\Lambda_1^{j+})_k.
\end{align*}
The condition $(\Lambda_0^{i-})_k \subset M_i$ is equivalent to $\pi(\Lambda_0^{i})_k^{\vee} \subset A_i$ and $\mathcal V\left((\Lambda_0^i)_k^{\vee}\right) \subset B_i^{\dagger}$. Observe that $B_i^{\dagger} = \mathcal V(B_i^{\vee})$. Since $(\Lambda_0^i)_k$ is $\tau$-invariant, the second inclusion is equivalent to $B_i \subset (\Lambda_0^i)_k$. On the other hand, since $\pi B_i^{\vee} \subset A_i$, the first inclusion is a consequence of the second.\\
Likewise, the condition $M_{j+1} \subset (\Lambda_1^{j+})_k$ is equivalent to $A_{j+1} \subset (\Lambda_1^{j})_k$ and $B_{j+1}^{\dagger} \subset \mathcal V^{-1}\left((\Lambda_1^{j})_k\right)$. The second inclusion is equivalent to $B_{j+1}^{\vee} \subset \mathcal V^{-2}\left((\Lambda_1^{j})_k\right)$. Since $\mathcal V^{-2} = \pi^{-1}\tau$ and since $(\Lambda_1^{j})_k$ is $\tau$-invariant, it is equivalent to $\pi B_{j+1}^{\vee} \subset (\Lambda_1^{j})_k$. Now, $\pi B_{j+1}^{\vee} \subset A_{j+1}$, so that the second inclusion is actually a consequence of the first. To sum up, we have proved that
\begin{align*}
\rho_{\Lambda_0^{i-},X} \text{ is an isogeny} & \iff B_i \subset (\Lambda_0^i)_k,\\
\rho_{X,\Lambda_1^{j+}} \text{ is an isogeny} & \iff A_{j+1} \subset (\Lambda_1^j)_k.
\end{align*}
All together, these conditions are equivalent to requiring that $(A_m \subset \ldots \subset B_m)$ belongs to $\mathcal N_{I,\mathbf{\Lambda}}^{\mathbbm h}(k)$ as defined in Definition \ref{DefinitionBTStrata}.
\end{proof}

\subsection{Deligne-Lusztig varieties}
\subsubsection{Coarse, parabolic and fine Deligne-Lusztig varieties}  \label{GeneralitiesDLVarieties}

In this section, we recall some generalities on Deligne-Lusztig varieties. Our references are \cite{dl}, \cite{dm} and \cite{dmbook}. The notations here are independent on the rest of the paper. Let $p$ be a prime number and let $q$ be a power of $p$. Let $\overline{\mathbb F_q}$ denote an algebraic closure of $\mathbb F_q$, and let $\mathbf G$ be a connected reductive group over $\overline{\mathbb F_q}$. Let $F: \mathbf G \to \mathbf G$ be a Frobenius morphism inducing an $\mathbb F_q$-rational structure on $\mathbf G$. Given any $F$-stable subgroup $\mathbf H \subset \mathbf G$, we write $H := \mathbf H^{F}$ for the subgroup of elements fixed by $F$. Let $(\mathbf T,\mathbf B)$ be a pair consisting of an $F$-stable maximal torus $\mathbf T$, contained in an $F$-stable Borel subgroup $\mathbf B$. Such a pair is unique up to $G$-conjugation. It induces a Coxeter system $(\mathbf W,\mathbf S)$ where $\mathbf W$ is the Weyl group attached to $\mathbf T$, and where $\mathbf S$ is the set of simple reflections determined by $(\mathbf T,\mathbf B)$. The Frobenius $F$ induces an action on $\mathbf W$ which preserves $\mathbf S$. Let $\ell$ denote the length function on $\mathbf W$ with respect to $\mathbf S$. For a subset $I \subset \mathbf S$, we write $\mathbf P_I,\mathbf U_I,\mathbf L_I$ respectively for the standard parabolic subgroup of type $I$, for its unipotent radical and for its unique Levi complement containing $\mathbf T$. We also denote by $\mathbf W_I$ the parabolic subgroup of $\mathbf W$ generated by $I$. We write $\ell(\mathbf W_I)$ for the maximal length of all the elements of $\mathbf W_I$. We write ${}^I\mathbf W$ (resp. $\mathbf W^{I}$) for the set of elements $w\in \mathbf W$ which are $I$-reduced (resp. reduced-$I$), ie. satisfying the relation $\ell(vw) = \ell(v) + \ell(w)$ (resp. $\ell(wv) = \ell(w) + \ell(v)$) for all $v \in \mathbf W_I$. Given two subsets $I,I' \subset \mathbf S$, we also write ${}^I\mathbf W^{I'} := {}^I\mathbf W \cap \mathbf W^{I'}$. 

\begin{defi}\label{CoarseDLVariety}
Let $\mathbf P \subset \mathbf G$ be a parabolic subgroup. The associated \textit{coarse Deligne-Lusztig variety} is defined by 
\begin{equation*}
X_{\mathbf P} := \{g\mathbf P \in \mathbf G/\mathbf P \,|\, g^{-1}F(g) \in \mathbf PF(\mathbf P)\}. 
\end{equation*} 
\end{defi}

The variety $X_{\mathbf P}$ is defined over $\mathbf F_{q^{\delta}}$, where $\delta$ is the smallest positive integer such that $F^{\delta}(\mathbf P) = \mathbf P$. It is equipped with a left action of $G$. Given two subsets $I,I' \subset \mathbf S$, recall the generalized Bruhat decomposition 
\begin{equation*}
\mathbf P_I \backslash \mathbf G / \mathbf P_{I'} = \bigsqcup_{w \in {}^I\mathbf W^{I'}} \mathbf P_I \backslash \mathbf P_I w \mathbf P_{I'} / \mathbf P_{I'} \simeq \mathbf W_I \backslash \mathbf W / \mathbf W_{I'}. 
\end{equation*}
This can be used to give an alternative parametrization of Deligne-Lusztig varieties. Namely, given $w \in {}^I\mathbf W^{F(I)}$, we define 
\begin{equation*}
X_I(w) := \{ g\mathbf P_I \in \mathbf G / \mathbf P_I \,|\, g^{-1}F(g) \in \mathbf P_I w \mathbf P_{F(I)}\}.
\end{equation*}
To go from one description to the other, let us fix a parabolic subgroup $\mathbf P \subset \mathbf G$. Let $I \subset \mathbf G$ be the unique subset such that $\mathbf P$ is conjugate to $\mathbf P_I$. Let $h \in \mathbf G$ such that $\mathbf P = {}^{h}\mathbf P_I$. By the Bruhat decomposition, there exists a unique $w\in {}^I\mathbf W^{F(I)}$ such that $h^{-1}F(h) \in P_IwP_{F(I)}$. Then the map $g\mathbf P \mapsto gh\mathbf P_I$ defines a $G$-equivariant isomorphism $X_{\mathbf P} \xrightarrow{\sim} X_{I}(w)$.

\begin{defi}
A (parabolic) \textit{Deligne-Lusztig variety} is a coarse Deligne-Lusztig variety $X_{\mathbf P}$ such that the parabolic subgroup $\mathbf P$ contains an $\mathbb F_q$-rational Levi complement $\mathbf L \subset \mathbf P$, ie. satisfying $F(\mathbf L) = \mathbf L$.
\end{defi}

If the Deligne-Lusztig variety is written as $X_{\mathbf P} \simeq X_I(w)$, the condition that $\mathbf P$ contains a rational Levi complement is equivalent to the equation 
\begin{equation}\label{RationalLeviCondition}
I = wF(I)w^{-1}.
\end{equation}

\begin{ex}
If $I = \emptyset$, then the condition \eqref{RationalLeviCondition} is always satisfied for any $w\in \mathbf W$. In this case, we call $X(w) := X_{\emptyset}(w)$ a \textit{classical Deligne-Lusztig variety}. These are the varieties originally introduced in \cite{dl}. Explicitly, we have 
\begin{equation*}
X(w) = \{g\mathbf B \in \mathbf G/\mathbf B \,|\, g^{-1}F(g) \in \mathbf B w \mathbf B\}.
\end{equation*}
\end{ex}

According to \cite{hoeve} Lemma 2.1.3 and \cite{bonnafe}, we have the following. 

\begin{prop}\label{DimDLVar}
Let $I \subset \mathbf S$ and let $w \in {}^I\mathbf W^{F(I)}$. The coarse Deligne-Lusztig variety $X_I(w)$ is smooth and purely of dimension 
\begin{equation*}
\dim(X_I(w)) = \ell(w) + \ell(\mathbf W_{F(I)}) - \ell(\mathbf W_{I \cap {}^wF(I)}).
\end{equation*}
The variety $X_I(w)$ is reducible if and only if $W_I w$ is contained in a parabolic subgroup $W_J$ for some proper subset $J \subsetneq \mathbf S$ with $F(J) = J$.
\end{prop}

In particular, if $X_I(w)$ is a parabolic Deligne-Lusztig variety, then $\dim(X_I(w)) = \ell(w)$. We recall the following definition from \cite{he}. 

\begin{defi}
Let $I \subset \mathbf S$ and $w \in {}^I\mathbf W$. The associated \textit{fine Deligne-Lusztig variety} is defined by 
\begin{equation*}
X_{I}\{w\} := \{g\mathbf P_I \in \mathbf G/\mathbf P_I \,|\, g^{-1}F(g) \in \mathbf P_I \cdot_{F} \mathbf B w \mathbf B\},
\end{equation*}
where $\cdot_{F}$ denotes the $F$-twisted conjugation, ie. $x\cdot_F y := xyF(x)^{-1}$ for all $x,y \in \mathbf G$.
\end{defi}

In other words, $X_{I}\{w\}$ is the image of the classical Deligne-Lusztig variety $X(w)$ under the natural map $\mathbf G / \mathbf B \to \mathbf G / \mathbf P_I$. 

\begin{ex}\label{FineParabolicDLVariety}
Assume that $w \in {}^I\mathbf W$ and that $I = wF(I)w^{-1}$. Then $w$ is reduced-$F(I)$ and we have $X_{I}\{w\} = X_I(w)$. This follows from \cite{he} Section 3. 
\end{ex}

The fine Deligne-Lusztig varieties define a stratification of the partial flag variety 
\begin{equation*}
\mathbf G/\mathbf P_I = \bigsqcup_{w \in {}^I\mathbf W} X_I\{w\},
\end{equation*}
and the closure of a stratum can be described via the following partial order on ${}^I\mathbf W$. For $w,w' \in {}^I\mathbf W$, we write $w' \leq_{I,F} w$ if and only if $uw'F(u)^{-1} \leq w$ for some $u \in \mathbf W_I$, where $\leq$ denotes the usual Bruhat order on $\mathbf W$. The following statement is \cite{he} Theorem 3.1.

\begin{theo}\label{ClosureFineDLVariety}
For $I \subset \mathbf S$ and ${}^I \mathbf W$, we have 
\begin{equation*}
\overline{X_{I}\{w\}} = \bigsqcup_{\substack{w' \in {}^I\mathbf W\\ w' \leq_{I,F} w}} X_{I}\{w'\},
\end{equation*}
where the closure is taken in the partial flag variety $\mathbf G/\mathbf P_I$.
\end{theo}

Given a subset $I \subset \mathbf S$, we define $T(I)$ as the set of all sequences $(I_n,w_n)_{n\geq 0}$ satisfying $I_0 = I$ and
\begin{align*}
\forall n\geq 0, w_{n} \in {}^{I_{n}}\mathbf W^{F(I_{n})}, & & I_{n+1} := I_n \cap {}^{w_n} F(I_n), & & w_{n+1} \in \mathbf W_{I_{n+1}}w_n\mathbf W_{F(I_n)}.
\end{align*}
According to \cite{bedard} Proposition I.9, any sequence $(I_n,w_n)_{n \geq 0}$ stabilizes for $n$ large enough to some pair $(I_{\infty},w_{\infty})$. In particular, we have 
\begin{equation*}
I_{\infty} = {}^{w_{\infty}}F(I_{\infty}) \text{ and } w_{\infty} \in {}^{I_{\infty}}\mathbf W ^{F(I_{\infty})}.
\end{equation*}
Moreover, we have $w_{\infty} \in {}^I \mathbf W$ and the mapping $(I_n,w_n)_{n \geq 0} \mapsto w_{\infty}$ defines a bijection $T(I) \xrightarrow{\sim} {}^I\mathbf W$. The following statement is \cite{bedard} Proposition 12. 

\begin{prop}
Let $I \subset \mathbf S$ and $w \in {}^I\mathbf W$. Let $(I_n,w_n)_{n\geq 0} \in T(I)$ be the unique sequence such that $w_{\infty} = w$. 
\begin{enumerate}
\item For $n \geq 0$, the morphism 
\begin{equation*}
X_{I_{n+1}}\{w_{n+1}\} \to X_{I_{n}}\{w_{n+1}\},
\end{equation*}
induced by the natural projection $\mathbf G/\mathbf P_{I_{n+1}} \to \mathbf G/\mathbf P_{I_{n}}$, is an isomorphism.
\item We have 
\begin{equation*}
X_{I}\{w\} \xleftarrow{\sim} X_{I_{\infty}}\{w_{\infty}\} = X_{I_{\infty}}(w). 
\end{equation*}
\end{enumerate}
\end{prop}

Point 2. is just a repeated iteration of the isomorphism of Point 1, combined with the statement of Example \ref{FineParabolicDLVariety}. In other words, fine Deligne-Lusztig varieties are just parabolic Deligne-Lusztig varieties associated to a smaller parameter $I$. We point out, as a consequence, that a fine Deligne-Lusztig variety $X_I\{w\}$ is smooth of pure dimension $\ell(w)$. 

\subsubsection{Example: the general linear group}

Let $V$ be a finite dimensional vector space over $\mathbb F_q$. Let $d:= \dim(V)$ and let $\mathbf G = \mathrm{GL}(V_{\overline{\mathbb F_q}})$ equipped with the standard Frobenius morphism $F:f \mapsto \Phi\circ f \circ \Phi^{-1}$, where $\Phi := \mathrm{id}\otimes \sigma$ is the operator on $V_{\overline{\mathbb F_q}} = V \otimes_{\mathbb F_q} \overline{\mathbb F_q}$ acting via $\sigma:x\mapsto x^q$ on the scalars. We have $G := \mathbf G^F = \mathrm{GL}(V)$. Fix a complete flag 
\begin{equation*}
\mathcal F: \{0\} = \mathcal F_0 \subset \mathcal F_1 \subset \ldots \subset \mathcal F_{d-1} \subset \mathcal F_{d} = V,
\end{equation*} 
where $\dim(\mathcal F_i) = i$ for all $0 \leq i \leq d$. The flag $\mathcal F$ determines an $F$-stable Borel subgroup $\mathbf B := \mathrm{Stab}(\mathcal F) \subset \mathbf G$ and an $F$-stable maximal torus $\mathbf T \subset \mathbf B$. Let $(\mathbf W, \mathbf S)$ be the associated Coxeter system, with $\mathbf W \simeq S_d$ and $\mathbf S \simeq \{s_1,\ldots, s_{d-1}\}$ where $s_i$ is the transposition permuting $i$ and $i+1$. The Frobenius acts trivially on $W$. Let $\mathbf d := (d_1,\ldots ,d_k)$ be a $k$-tuple of positive integers such that $d_1 + \ldots + d_k = d$, where $k\geq 1$. A partial flag of type $\mathbf d$ is a sequence 
\begin{equation*}
\mathcal G: \{0\} = \mathcal G_0 \subset \mathcal G_1 \subset \ldots \subset \mathcal G_{m-1} \subset \mathcal G_k = V,
\end{equation*} 
where $\dim(\mathcal G_{i}/\mathcal G_{i-1}) = d_i$ for all $1 \leq i \leq m$. A basis $e := (e_1,\ldots ,e_d)$ of $V$ is said to be adapted to the partial flag $\mathcal G$ of type $\mathbf d$ if, for all $1 \leq i \leq m$, the vectors $(e_1,\ldots , e_{d_1 + \ldots + d_i})$ form a basis of $\mathcal G_i$. Given two partial flags $\mathcal G$ and $\mathcal G'$ (which may be of different types), there exists a permutation $w\in \mathbf W \simeq S_d$ and a basis $(e_1,\ldots, e_d)$ which is adapted to $\mathcal G$, and such that $(e_{w(1)},\ldots,e_{w(d)})$ is adapted to $\mathcal G'$. We say that the flags $\mathcal G$ and $\mathcal G'$ are in relative position $w$. Given $I \subset \mathbf S$, write $\mathbf S\setminus I = \{s_{i_1},\ldots , s_{i_r}\}$ with $0 \leq r \leq d-1$ and $1 \leq i_1 < \ldots < i_r \leq d-1$. Define a tuple $\mathbf d_I := (d_1,\ldots,d_{r+1})$ where $d_1 := i_1, d_{r+1} = d-i_r$ and $d_j = i_j - i_{j-1}$ for all $2\leq j \leq r$. For $w\in {}^I\mathbf W^I$, the coarse Deligne-Lusztig variety $X_I(w)$ is defined over $\mathbb F_q$, and for any field extension $k/\mathbb F_q$, its $k$-rational points are given by 
\begin{equation*}
 \hspace*{-0.3cm} X_I(w)(k) = \{\text{partial flags } \mathcal G \text{ of type } \mathbf d_I \text{ in } V_k \,|\, \mathcal G \text{ and } \Phi(\mathcal G) \text{ are in relative position } w\},
\end{equation*}
where $V_k := V \otimes_{\mathbb F_q} k$. More generally, let $R$ be any $\mathbb F_q$-algebra. By a flag of type $\mathbf d_I$ in $V_R= V\otimes_{\mathbb F_{q}} R$, we mean an increasing chain of (finite locally free) locally direct summands of the $R$-module $V_R$, whose ranks increase by increments of $d_1,\ldots , d_{r+1}$. Following \cite{vw1}, we say that two flags $\mathcal F$ and $\mathcal G$, of type respectively $I$ and $J$, are in \textbf{standard position} if all the submodules $\mathcal F_i + \mathcal G_j$ are locally direct summands of $V_R$. In such a situation, one can define the relative position of $\mathcal F$ and of $\mathcal G$ as a global section $w$ of the constant sheaf ${}^{I}\underline{\mathbf W}^{J}$ on $\mathrm{Spec}(R)$, such that $\mathcal F$ and $\mathcal G$ are locally in relative position $w$. The $R$-rational points of $X_I(w)$ are given by 
\begin{equation*}
\hspace*{-0.5cm} X_I(w)(R) = \left\{\text{partial flags } \mathcal G \text{ of type } \mathbf d_I \text{ in } V_R \,\middle|\, \begin{array}{c}
\mathcal G \text{ and } \Phi(\mathcal G) \text{ are in standard position,} \\
\text{and in relative position } w
\end{array} \right\},
\end{equation*}
for $I \subset \mathbf S$ and $w \in {}^{I}\mathbf W^{I}$.

\subsubsection{Example: the fake unitary case} \label{DL:FakeUnitary}


Given a vector space $V$ over a field extension $k/\mathbb F_{q}$, we write $V^{(q)} := V \otimes_{k,\sigma} k$ where $\sigma:k\to k$ is the arithmetic Frobenius $x\mapsto x^q$. Given an endomorphism $f:V\to V$, we write $f^{(q)}$ for the induced endomorphism of $V^{(q)}$.\\
Let $V_1,V_2$ be two finite dimensional vector spaces over $\mathbb F_{q^2}$ of the same dimension. Let $d := \dim(V_1) = \dim(V_2)$, and let $B:V_1 \times V_2^{(q)}	 \rightarrow \mathbb F_{q^2}$ be an $\mathbb F_{q^2}$-bilinear perfect pairing. Let $\mathbf G := \mathrm{GL}(V_{1,\overline{\mathbb F_{q^2}}}) \times \mathrm{GL}(V_{2,\overline{\mathbb F_{q^2}}})$, and equip $\mathbf G$ with the $\mathbb F_q$-structure given by the Frobenius morphism 
\begin{align*}
F: \mathbf G & \longrightarrow \mathbf G,\\
(u,v) & \longmapsto ((v^{-1})^{(q),*},(u^{-1})^{*,(q)}).
\end{align*} 
Here, $(\,\cdot\,)^*$ denotes the adjoint endomorphism with respect to $B$. Observe that $F^2(u,v) = (u^{(q^2)},v^{(q^2)})$, and that we have $F(u,v) = (u,v)$ if and only if $u = u^{(q^2)}$ and $v = (u^{-1})^{*,(q)}$. Thus, the mapping $(u,v) \mapsto u$ defines an isomorphism
\begin{equation*}
G := \mathbf G^F \xrightarrow{\sim} \mathrm{GL}(V_1).
\end{equation*}
Let $\mathcal F^1$ be a complete flag in $V_1$. Then $\mathcal F^2 := (\mathcal F^1)^{\perp,(q)}$ is a complete flag in $V_2$, where $(\cdot)^{\perp}$ denotes the orthogonal complement with respect to $B$. The stabilizer $\mathbf B := \mathrm{Stab}(\mathcal F^1) \times \mathrm{Stab}(\mathcal F^2)$ is an $F$-stable Borel subgroup of $\mathbf G$, containing a maximal $F$-stable torus $\mathbf T$. Let $(\mathbf W,\mathbf S)$ be the associated Coxeter system. We have $\mathbf W \simeq S_d \times S_d$ and $\mathbf S \simeq \{(s_i,\mathrm{id}),(\mathrm{id},s_i)\}_{1\leq i \leq d-1}$. Moreover, the Frobenius acts on $\mathbf W$ via $F(w_1,w_2) = (w_0w_2w_0,w_0w_1w_0)$. Here, $w_0 \in S_d$ is the longest element, defined by $w_0(i) := d+1-i$. Notice that $F^2 = \mathrm{id}$ on $\mathbf W$. Given $I \subset \mathbf S$, write 
\begin{equation*}
\mathbf S \setminus I = \{(s_{i_1},\mathrm{id}),\ldots , (s_{i_r},\mathrm{id}),(\mathrm{id},s_{j_1},),\ldots,(\mathrm{id},s_{j_{r'}})\},
\end{equation*}
where $0 \leq r,r' \leq d-1$ and $1\leq i_1 < \ldots < i_r \leq d-1,$ and $1 \leq j_1 < \ldots < j_{r'} \leq d-1$. Define also $\mathbf d_I := (\mathbf d_I^1,\mathbf d_I^2)$ where $\mathbf d_I^1 := (d_1^1,\ldots, d_{r+1}^1)$ and $\mathbf d_I^2 = (d_1^2,\ldots , d_{r'+1}^2)$, with $d_k^1 := i_{k+1}-i_k$ and $d_k^2 := j_{k+1} - j_k$. Eventually, let $w = (w_1,w_2) \in {}^I\mathbf W^{F(I)}$. The coarse Deligne-Lusztig variety $X_I(w)$ is at least defined over $\mathbb F_{q^2}$, and for any field extension $k / \mathbb F_{q^2}$, its $k$ rational points are given by 
\begin{equation*}
\hspace*{-1cm} X_I(w)(k) = \left\{(\mathcal G^1,\mathcal G^2) \,\middle|\, \begin{array}{c}
\mathcal G^1 \text{ is a flag of type } \mathbf d_I^{1} \text{ in } V_{1,k}, \text{ in relative position } w_1 \text{ with }(\mathcal G^2)^{(q),\perp},\\
\mathcal G^2 \text{ is a flag of type } \mathbf d_I^{2} \text{ in } V_{2,k}, \text{ in relative position } w_2 \text{ with } (\mathcal G^1)^{\perp,(q)}.
\end{array}
\right\}.
\end{equation*}
If $R$ is an $\mathbb F_{q^2}$-algebra, we note that the orthogonal complement of a locally direct summand of $V_{1,R}$ or of $V_{2,R}^{(q)}$ is respectively a locally direct summand of $V_{2,R}^{(q)}$ or of $V_{1,R}$. Thus, the notion of orthogonal flag makes sense in this context, and $X_I(w)(R)$ can be described just as in the case of a field.

\subsubsection{Example: the unitary group} \label{DL:Unitary}

Let $V$ be a finite dimensional vector space over $\mathbb F_{q^2}$, equipped with a perfect $\mathbb F_{q^2}/\mathbb F_q$-hermitian form $(\cdot,\cdot)$. Let $d := \dim(V)$. If $k$ is a field extension of $\mathbb F_{q^2}$, extend $(\cdot,\cdot)$ to $V_k$ via the formula 
\begin{equation*}
(v\otimes x,w\otimes y) := xy^{\sigma}(v,w) \in k,
\end{equation*}
where $\sigma:x\mapsto x^q$. One may also think of $(\cdot,\cdot)$ as a perfect bilinear pairing $V\times V^{(q)} \to \mathbb F_{q^2}$, allowing us to recover notations from the previous fake unitary case with $V_1 = V_2 = V$. Let $\mathbf G := \mathrm{GL}(V_{\overline{\mathbb F_q}})$, and equip $\mathbf G$ with the $\mathbb F_q$-structure induced by the Frobenius morphism 
\begin{align*}
F: \mathbf G & \longrightarrow \mathbf G,\\
u & \longmapsto (u^{-1})^{*,(q)}.
\end{align*} 
We have $F^2(u) = u^{(q^2)}$ for all $u$, and 
\begin{equation*}
G := \mathbf G^F \simeq \mathrm{U}(V,(\cdot,\cdot)).
\end{equation*} 
Let us fix a complete flag $\mathcal F^0$ in $V$ such that $\mathcal F^0 = (\mathcal F^0)^{\perp,(q)}$. It defines a Borel subgroup $\mathbf B := \mathrm{Stab}(\mathcal F^0) \subset \mathbf G$ which is $F$-stable, and contains an $F$-stable maximal torus $\mathbf T$. Let $(\mathbf W,\mathbf S)$ be the associated Coxeter system. Then $\mathbf W \simeq S_d$ and $\mathbf S \simeq \{s_1,\ldots ,s_{d-1}\}$. The Frobenius acts on $\mathbf W$ via $F(w) = w_0ww_0$. In particular, we have $F(s_i) = s_{d-i}$ and $F^2 = \mathrm{id}$ on $\mathbf W$. Given $I \subset \mathbf S$, write $\mathbf S\setminus I = \{s_{i_1},\ldots , s_{i_r}\}$ with $0 \leq r \leq d-1$ and $1 \leq i_1 < \ldots < i_r \leq d-1$. Define a tuple $\mathbf d_I := (d_1,\ldots,d_{r+1})$ just as in the $\mathrm{GL}$ case. For $w\in {}^I\mathbf W^{F(I)}$, the coarse Deligne-Lusztig variety $X_I(w)$ is defined at least over $\mathbb F_{q^2}$, and for any field extension $k/\mathbb F_{q^2}$, its $k$-rational points are given by 
\begin{equation*}
 \hspace*{-0.3cm} X_I(w)(k) = \{\text{partial flags } \mathcal G \text{ of type } \mathbf d_I \text{ in } V_k \,|\, \mathcal G \text{ and } \mathcal G^{\perp,(q)} \text{ are in relative position } w\}.
\end{equation*}
The same description holds more generally for $X_I(w)(R)$ where $R$ is any $\mathbb F_{q^2}$-algebra. 

\subsubsection{Some combinatorial lemmas}

In this section, we put together some combinatorial lemmas related to the symmetric group, that we will refer to in later proofs. Given $1 \leq i \leq n-1$, we will always write $s_i$ for the transposition $(i \;\; i+1) \in S_n$. We write $\mathbf S := \{s_1,\ldots ,s_{n-1}\}$. We denote by $\ell$ the length function with respect to the $s_i$'s. 

\begin{lem}\label{UniqueDecompositionPermutation}
Let $0 \leq k < n$ and let $\sigma \in S_n$ such that 
\begin{equation*}
\sigma(\{1,\ldots ,k\}) \subset \{1,\ldots ,k+1\}.
\end{equation*}
Then there exists unique permutations $\sigma_1,\sigma_2,\tau \in S_n$ such that $\sigma = \tau\sigma_1\sigma_2$, satisfying the following conditions:
\begin{enumerate}[noitemsep,nolistsep]
\item $(\sigma_1)_{|\{k+1,\ldots,n\}} = \mathrm{id}$,
\item $(\sigma_2)_{|\{1,\ldots , k\}} = \mathrm{id}$,
\item $\tau$ is $\mathrm{id}$ if $\sigma^{-1}(k+1)>k$, and a transposition otherwise.
\end{enumerate}
\end{lem}

\begin{proof}
The statement is trivially true if $k=0$, so that we may assume $k\geq 1$. By the hypothesis, there exists a unique $x\in \{1,\ldots ,k+1\}$ such that $\sigma^{-1}(x) > k$. We define $\tau = (x \;\; k+1)$. Note that $\tau = \mathrm{id}$ if $x = k+1$, and $\tau$ is a transposition otherwise. We claim that 
\begin{equation*}
\tau\sigma(\{1,\ldots ,k\}) = \{1,\ldots , k\}.
\end{equation*}
Indeed, for $1 \leq i \leq k$ such that $i \not = \sigma^{-1}(k+1)$, we have $\sigma(i) \leq k$ and $\sigma(i) \not = x$. Thus, $\tau\sigma(i) = \sigma(i) \in \{1,\ldots , k\}$. Moreover, if $\sigma^{-1}(k+1) \leq k$, then $x \leq k$ and we have $\tau\sigma(\sigma^{-1}(k+1)) = x \in \{1,\ldots , k\}$. This proves the claim. We can now define $\sigma_1$ and $\sigma_2$ as follows. 
\begin{equation*}
\sigma_1(i) = \begin{cases}
\tau\sigma(i) & \text{if } 1 \leq i \leq k,\\
i & \text{if } k+1 \leq i \leq n,
\end{cases}
\text{          and          }
\sigma_2(i) = \begin{cases}
i & \text{if } 1 \leq i \leq k,\\
\tau\sigma(i) & \text{if } k+1 \leq i \leq n.
\end{cases}
\end{equation*}
Then we have $\sigma = \tau\sigma_1\sigma_2$ as desired.\\
We now prove unicity. Assume that $\sigma = \tau\sigma_1\sigma_2$ where $\tau,\sigma_1$ and $\sigma_2$ satisfy the conditions 1, 2 and 3 above. The identity $\tau\sigma = \sigma_1\sigma_2$ implies $\sigma_1(i) = \tau\sigma(i)$ for all $1\leq i \leq k$, and that $\sigma_2(i) = \tau\sigma(i)$ for all $k+1\leq i \leq n$. It remains to determine $\tau$. Assume that $\tau$ is a transposition, so that $x \leq k$. We have $\tau\sigma(\{1,\ldots,k\}) = \{1,\ldots ,k\}$. Since $\sigma^{-1}(k+1) \leq k$, it follows that $\tau$ has the form $(a \;\; k+1)$ for some $a \leq k$. Moreover, since $\sigma^{-1}(x) > k$, we must have $\tau(x) > k$. Therefore $a=x$ and $\tau = (x \; \; k+1)$ as desired. 
\end{proof}

\begin{prop}\label{FineDLVartietiesInUnitary}
Let $r\geq 1$ and let $0 \leq k_1 < \ldots < k_r < n-1$. Let $I := \mathbf S \setminus \{s_{k_1+1},s_{k_2+1},\ldots,s_{k_r+1}\}$. Let $\sigma \in S_n$ be $I$-reduced and such that for all $1\leq j \leq r$, we have 
\begin{equation*}
\sigma(\{1,\ldots,k_{j}\}) \subset \{1,\ldots,k_j+1\}. 
\end{equation*}
Then $\sigma$ can be uniquely written as $\sigma = w_1\ldots w_r$ where, for each $1 \leq i \leq r-1$, we have $w_i = s_{k_i+1}s_{k_i+2}\ldots s_{k_i+t_i}$ for some $0 \leq t_i \leq k_{i+1}-k_i$, and $w_r = s_{k_r+1}\ldots s_{k_r+t_r}$ for some $0 \leq t_r \leq n-1-k_r$.
\end{prop}

We note that if $k_1 = 0$, the condition $\sigma(\emptyset) \subset \{1\}$ is trivially true.

\begin{proof}
By induction on $r$, first let us assume that $r=1$. According to Lemma \ref{UniqueDecompositionPermutation}, we can write $\sigma = \tau\sigma_1\sigma_2$ where $\sigma_2$ is generated by $s_{k_1+1},\ldots ,s_{n-1}$, and where $\tau\sigma_1$ is generated by $s_1,\ldots s_{k_1}$. Since $\sigma$ is $I$-reduced, one may check that $\tau\sigma_1$ is $\{s_1,\ldots , s_{k_1}\}$-reduced and that $\sigma_2$ is $\{s_{k_1+2},\ldots ,s_{n-1}\}$-reduced. It follows that $\tau\sigma_1 = \mathrm{id}$ and that $\sigma_2 = \sigma = s_{k_1+1}\ldots s_{k_1+t_1}$ for some $0\leq t_1 \leq n-1-k_1$ as required.\\
Let us now assume that the statement holds for $r-1$, where $r\geq 2$. Let $\sigma$ be as in the Proposition. By Lemma \ref{UniqueDecompositionPermutation} with respect to $k = k_r$, we can decompose $\sigma = \tau\sigma_1\sigma_2$ where $\sigma_2$ is generated by $s_{k_r+1},\ldots ,s_{n-1}$, and $\tau\sigma_1$ is generated by $s_1,\ldots,s_{k_r}$. Again, since $\sigma$ is $I$-reduced, one may check that $\sigma_2$ is $\{s_{k_r+2},\ldots,s_{n-1}\}$-reduced. It follows that $\sigma_2 = w_{r} = s_{k_{r}+1}\ldots s_{k_r+t_r}$ for some $0 \leq t_r \leq n-1-k_r$. Moreover, $\tau\sigma_1$ is $I'$-reduced, where $I' := I \setminus \{s_{k_r+2},\ldots,s_{n-1}\}$. Eventually, for all $i \leq k_r$ we have $\sigma_2(i) = i$, so that $\sigma(i) = (\tau\sigma_1)(i)$. It follows that for all $1 \leq j \leq r-1$, we have 
\begin{equation*}
\tau\sigma_1(\{1,\ldots,k_j\}) \subset \{1,\ldots ,k_j+1\}.
\end{equation*} 
Thus, the restriction of $\tau\sigma_1$ to $\{1,\ldots ,k_r+1\}$ satisfies all the hypotheses of the Proposition with respect to $k_1,\ldots ,k_{r-1}$. By induction, the proof is over.
\end{proof}

Denote by $F$ the automorphism of $S_n$ defined by $\sigma \mapsto F(\sigma) = w_0\sigma w_0$, where $w_0:i\mapsto n+1-i$ is the longest element of $S_n$ with respect to $\mathbf S$. For $I \subset \mathbf S$, recall the order $\leq_{I,F}$ on ${}^IS_{n}$ which we defined in Section \ref{GeneralitiesDLVarieties}.

\begin{prop}\label{StratificationFineDLVarietiesUnitary}
Let $r\geq 0$ and let $0 \leq k_1 < \ldots < k_r < n-1$. Assume that $k_i = n-1-k_{r+1-i}$ for all $1 \leq i \leq r$ if $k_1 \not = 0$, and that $k_i = n-1-k_{r+2-i}$ for all $2 \leq i \leq r$ if $k_1 = 0$. Let $I := \mathbf S \setminus \{s_{k_1+1},\ldots,s_{k_r+1}\}$, and let $\sigma \in S_n$ be $I$-reduced and such that 
\begin{equation*}
\sigma \leq_{I,F} s_{k_1+1} s_{k_1+2} \ldots s_{n-1}.
\end{equation*}
Then $\sigma$ can be uniquely written as $\sigma = w_1\ldots w_r$ where, for each $1 \leq i \leq r-1$, we have $w_i = s_{k_i+1}s_{k_i+2}\ldots s_{k_i+t_i}$ for some $0 \leq t_i \leq k_{i+1}-k_i$, and $w_r = s_{k_r+1}\ldots s_{k_r+t_r}$ for some $0 \leq t_r \leq n-1-k_r$.
\end{prop}

The hypothesis on the $k_i$'s is imposed so that $F(I) = \mathbf S \setminus \{s_{k_1},\ldots ,s_{k_r}\}$ when $k_1 \not = 0$, and $F(I) = \mathbf S \setminus \{s_{k_2},\ldots s_{k_r}, s_{n-1}\}$ when $k_1 = 0$. 

\begin{proof}
We only prove the case $k_1 \not = 0$ By hypothesis, there is some $u \in (S_n)_I$ such that $u\sigma F(u)^{-1} \leq s_{k_1+1}\ldots s_{n-1}$. Let $k_1+1\leq i_1 < \ldots < i_t \leq n-1$ be such that $u\sigma F(u)^{-1} = s_{i_1}\ldots s_{i_t}$. Since $u \in (S_n)_I$ and $F(u) \in (S_n)_{F(I)}$, we have $u(\{1,\ldots , k_{i}+1\}) = \{1,\ldots ,k_{i}+1\}$ and $F(u)(\{1,\ldots,k_{i}\}) = \{1,\ldots,k_{i}\}$ for all $1 \leq i \leq r$. It follows that 
\begin{align*}
\sigma(\{1,\ldots,k_i\}) & = \sigma F(u)^{-1}(\{1,\ldots,k_i\}) \\
& = u^{-1}s_{i_1}\ldots s_{i_t}(\{1,\ldots ,k_i\}) \\
& \subset u^{-1}(\{1,\ldots , k_{i}+1\}) \\ 
& = \{1,\ldots, k_i+1\}.
\end{align*}
By Proposition \ref{FineDLVartietiesInUnitary}, $\sigma$ can be decomposed as $\sigma = w_1\ldots w_r$ as desired.
\end{proof}

For the next Proposition, consider the Coxeter group $S_n \times S_n$ with simple reflections $\mathbf S \sqcup \mathbf S$ and Frobenius action $F(\sigma_1,\sigma_2) := (w_0\sigma_2 w_0,w_0\sigma_1 w_0)$. 

\begin{prop}\label{StratificationFineDLVarietiesFakeUnitary}
Let $r,r'\geq 0$, let $0 \leq k_1 < \ldots < k_r < n-1$ and let $0 \leq k'_1 < \ldots < k'_{r'} < n-1$. Assume that we are in one of the four following cases:
\begin{enumerate}
\item $r = r'$ and $k_1=k'_1 = 0$, in which case we assume that $k_i = n-1-k'_{r+2-i}$ for all $2 \leq i \leq r$,
\item $r = r'$ and $k_1,k'_1 > 0$, in which case we assume that $k_i = n-1-k'_{r+1-i}$ for all $1 \leq i \leq r$,
\item $r' = r-1$, $k_1 = 0$ and $k'_1 >0$, in which case we assume that $k_i = n-1-k'_{r+1-i}$ for all $2 \leq i \leq r$,
\item $r = r'-1$, $k_1>0$ and $k'_1 = 0$, in which case we assume that $k_i = n-1-k'_{r+2-i}$ for all $1 \leq i \leq r$.
\end{enumerate}
Let $I_1 := \mathbf S \setminus \{s_{k_1+1},\ldots ,s_{k_r+1}\}$ and let $I_2 := \mathbf S \setminus \{s_{k'_1+1},\ldots ,s_{k'_{r'}+1}\}$. Let $\sigma = (\sigma_1,\sigma_2) \in S_n\times S_n$ be $(I_1\sqcup I_2)$-reduced and such that 
\begin{equation*}
\sigma \leq_{I_1\sqcup I_2,F} (s_{k_1+1}\ldots s_{n-1},s_{k'_1+1} \ldots s_{n-1}).
\end{equation*}
Then $\sigma_1$ and $\sigma_2$ can be uniquely written as $\sigma_1 = w_1\ldots w_r$ and $\sigma_2 = w_1'\ldots w_{r'}'$ where, for each $1 \leq i \leq r-1$ and each $1 \leq j \leq r'-1$, we have $w_i = s_{k_i+1}\ldots s_{k_i+t_i}$ and $w_j' = s_{k'_j+1}\ldots s_{k'_j+t_j'}$ for some $0 \leq t_i \leq k_{i+1}-k_i$ and $0 \leq t'_j \leq k'_{j+1} - k'_{j}$, and where $w_r = s_{k_r+1} \ldots s_{k_r+t_r}$ and $w'_{r'} = s_{k'_{r'}}\ldots s_{k'_{r'}+t'_{r'}}$ for some $0 \leq t_r \leq n-1-k_r$ and $0 \leq t_{r'}' \leq n-1-k'_{r'}$.
\end{prop}

The assumptions on the $k_i$ and the $k'_i$ are imposed so that we have 
\begin{equation*}
F(I_2) = \begin{cases}
\mathbf S \setminus \{s_{k_2},\ldots, s_{k_r}, s_{n-1}\} & \text{if } k_1=k'_1=0\\
\mathbf S \setminus \{s_{k_1},\ldots, s_{k_r}\} & \text{if } k_1,k'_1 > 0,\\
\mathbf S \setminus \{s_{k_2},\ldots,s_{k_r}\} & \text{if } k_1=0,k'_1>0,\\
\mathbf S \setminus \{s_{k_1},\ldots , s_{k_r},s_{n-1}\} & \text{if } k_1>0, k'_1=0.
\end{cases}
\end{equation*}
We can write down $F(I_1)$ similarly as well. 

\begin{proof}
In each case, one may treat $\sigma_1$ and $\sigma_2$ in the same way as in Proposition \ref{StratificationFineDLVarietiesUnitary}. We omit the details.
\end{proof}

\subsection{The varieties $X_{I,\mathbf{\Lambda}}^{\mathbbm h}$}\label{Section3.3}

We recover the notations of Sections \ref{Section2.3} and \ref{Section2.4}. Let $(I,\mathbf{\Lambda})$ be a Bruhat-Tits index. Write $0 \leq i_1 < \ldots < i_s \leq m$ for the elements of $I$, where $s\geq 1$. We partition $\mathbf{\Lambda}$ and $\mathbbm h$ as follows:
\begin{enumerate}
\item for $1 \leq j \leq s-1$, define $\mathbf{\Lambda}_j := \{\Lambda_1^{i_j},\Lambda_0^{i_{j+1}}\}$ and $\mathbbm h_j := (h_{i_j+1},\ldots,h_{i_{j+1}})$,
\item if $i_1 \not = 0$, define $\mathbf{\Lambda}_0 := \{\Lambda_0^{i_1}\}$ and $\mathbbm h_0 := (h_1,\ldots,h_{i_1})$,
\item if $i_s \not = m$, define $\mathbf{\Lambda}_s := \{\Lambda_1^{i_s}\}$ and $\mathbbm h_s := (h_{i_s+1},\ldots ,h_m)$.
\end{enumerate}
By doing so, it is clear that $\mathbf{\Lambda}$ is the disjoint union of all the $\mathbf{\Lambda}_j$'s and that $\mathbbm h$ is the concatenation of the $\mathbbm h_j$'s.\\

\underline{If $i_1 \not = 0$:} consider the hermitian space $V_{\Lambda_0^{i_1}}^0 := \Lambda_0^{i_1}/\pi(\Lambda_0^{i_1})^{\vee}$ over $\mathbb F_{q^2}$, as defined in Section \ref{Section2.2}. We consider Deligne-Lusztig varieties with respect to the unitary group $\mathrm U(V_{\Lambda_0^{i_1}}^{0},\{\cdot,\cdot\})$ as in Section \ref{DL:Unitary}. Note that $\dim(V_{\Lambda_0^{i_1}}^{0}) = t(\Lambda_0^{i_1}) = 2(l-1) + h_{i_1} + 1$ for some $l \geq 1$. Consider the Weyl group $\mathbf W \simeq S_{2(l-1)+h_{i_1}+1}$ and the set of simple reflections $\mathbf S = \{s_1,\ldots ,s_{2(l-1)+h_{i_1}}\}$. We consider $J \subset \mathbf S$ such that $\mathbf d_J$ corresponds to flags of the following type.
\begin{equation*}
\{0\} \overset{l}{\subset} U_{i_1} \overset{\Delta h_{i_1-1}}{\subset} U_{i_1-1} \subset \ldots \overset{\Delta h_1}{\subset} U_1 \overset{h_1}{\subset} W_1 \overset{\Delta h_1}{\subset} \ldots \subset W_{i_1-1} \overset{\Delta h_{i_1-1}}{\subset} W_{i_1} \overset{l-1}{\subset} V_{\Lambda_0^{i_1}}^0
\end{equation*}
In other words, we have
\begin{equation*}
J = \mathbf S \setminus \{s_{l},s_{l+\frac{h_{i_1}-h_{i_1-1}}{2}},\ldots , s_{l+\frac{h_{i_1}-h_1}{2}},s_{l + \frac{h_{i_1}+h_1}{2}}, \ldots , s_{l+\frac{h_{i_1}+h_{i_1-1}}{2}},s_{l+h_{i_1}}\}.
\end{equation*}
Note that if $h_1 = 0$ then $s_{l+\frac{h_{i_1}-h_1}{2}} = s_{l + \frac{h_{i_1}+h_1}{2}}$, and if $l=1$ then the last term ``$s_{1+h_{i_1}}$'' does not exist. In particular, we have 
\begin{equation*}
r := \#(\mathbf S \setminus J) = \begin{cases}
2i_1 & \text{if } h_1 > 0 \text{ and } l>1,\\
2 i_1 - 1 & \text{if } (h_1 = 0 \text{ and } l>1) \text{ or } (h_1 > 0 \text{ and } l=1),\\
2 i_1 - 2 & \text{if } h_1 = 0 \text{ and } l=1.
\end{cases}
\end{equation*}
To simplify the notations, we write $0 \leq k_1 < \ldots < k_r < 2(l-1)+h_1$ for the indices such that
\begin{equation*}
J = \mathbf S \setminus \{s_{k_1+1},s_{k_2+1},\ldots , s_{k_r+1}\}.
\end{equation*}
For instance, $k_1 = l-1$, $k_2 = l-1+\Delta h_{i_1-1}$, and so on. Note that $k_r = l+h_{i_1}-1$ if $l > 1$, and $k_r = h_{i_1}-\Delta h_{i_1-1}$ if $l=1$. We observe that $F(s_{k_i+1}) = s_{k_{r+1-i}}$ for all $1 \leq i \leq r$ when $k_1 \not = 0$, and that $F(s_{k_i+1}) = s_{k_{r+2-i}}$ for all $2 \leq i \leq r$ when $k_1 = 0$. We define
\begin{equation*}
Y_{\mathbf{\Lambda}_0}^{\mathbbm h_0} := X_J\{s_{l}s_{l+1} \ldots s_{2(l-1) + h_{i_1}}\},
\end{equation*}
and $X_{\mathbf{\Lambda}_0}^{\mathbbm h_0} := \overline{Y_{\mathbf{\Lambda}_0}^{\mathbbm h_0}}$, where the closure is taken inside the flag variety of type $J$.

\begin{prop}\label{SmoothnessDLVar}
The variety $X_{\mathbf{\Lambda}_0}^{\mathbbm h_0}$ is projective, smooth and geometrically irreducible of dimension $l+h_{i_1}-1$.
\end{prop}

\begin{proof}
We have $\dim(X_{\mathbf{\Lambda}_0}^{\mathbbm h_0}) = \ell(s_{l}s_{l+1} \ldots s_{2(l-1) + h_{i_1}}) = l+h_{i_1}-1$. In order to prove the smoothness of $X_{\mathbf{\Lambda}_0}^{\mathbbm h_0}$ and, in doing so, irreducibility as well, we follow the approach of \cite{cox24} Sections 7.1 and 7.2. There, a priori only cases of Coxeter type are considered, while our setting is in the more general case of fully Hodge-Newton decomposable type. For this reason, we choose to repeat the arguments here.\\
Let us write $w := s_{l}s_{l+1}\ldots s_{2(l-1) + h_{i_1}}$ and let $w' := s_{l}s_{l+1} \ldots s_{l-1 + h_{i_1}}$, so that $w'$ is the shortest element of $\mathbf W_J w \mathbf W_{F(J)}$. In particular $w' \in {}^J\mathbf W^{F(J)}$. We show that $X_{\mathbf{\Lambda}_0}^{\mathbbm h_0} = \overline{X_J(w')}$. Clearly we have $Y_{\mathbf{\Lambda}_0}^{\mathbbm h_0} \subset X_J(w')$, so that we also have an inclusion of closures. Moreover, the coarse Deligne-Lusztig variety $X_J(w')$ is irreducible by Proposition \ref{DimDLVar} since no proper $F$-stable parabolic subgroup of $\mathbf W$ contains $\mathbf W_Jw'$. Thus $\overline{X_J(w')}$ is irreducible as well, and to prove the equality it is enough to show that $\dim(X_J(w')) = \dim(Y_{\mathbf{\Lambda}_0}^{\mathbbm h_0}) = \ell(w)$. We know that 
\begin{equation*}
\dim(X_J(w')) = \ell(w') + \ell(\mathbf W_{F(J)}) - \ell(\mathbf W_{J\cap {}^{w'}F(J)}),
\end{equation*}
and one may check that $\ell(\mathbf W_{F(J)}) - \ell(\mathbf W_{J\cap {}^{w'}F(J)}) = l-1$. Since $\ell(w) = \ell(w') + l-1$, the result follows.\\
So far we have proved that $X_{\mathbf{\Lambda}_0}^{\mathbbm h_0} = \overline{X_J(w')}$. In particular, $X_{\mathbf{\Lambda}_0}^{\mathbbm h_0}$ is geometrically irreducible. According to \cite{cox24}, the variety $\overline{X_J(w')}$ is smoothly equivalent to the Schubert variety in the complete flag variety for the longest element of $\mathbf W_J w' \mathbf W_{F(J)}$. Since $w' \in {}^J\mathbf W^{F(J)}$, this longest element can be written as $xw'y$ where $y$ is the longest element of $\mathbf W_{F(J)}$ and $x$ is the longest element of $\mathbf W_J \cap \mathbf W^{F(J)}w^{\prime -1} = \mathbf W_J \cap \mathbf W^{J\cap {}^{w'}F(J)}$. One may check that $x = s_1\ldots s_{l-1}$, and the element $xw'y$ can be written as 
\begin{equation*}
\hspace{-1.7cm} \left(
\begin{array}{cccccccccccccc}
1 & 2 & \ldots & k_1 & k_1 + 1 & \ldots & k_2 & \ldots & k_r & k_r+1 & k_r+2 & \ldots & t(\Lambda_0^{i_1})-1 & t(\Lambda_0^{i_1}) \\
k_1+1 & k_1 & \ldots & 2 & k_2 + 1 & \ldots & k_1+2 & \ldots & k_{r-1}+2 & t(\Lambda_0^{i_1}) & t(\Lambda_0^{i_1}) - 1 & \ldots & k_r + 2 & 1
\end{array} 
\right)
\end{equation*}
This permutation avoids the patterns $(3412)$ and $(4231)$, thus the associated Schubert variety is smooth according to \cite{Schubert} Theorem 8.1.1.
\end{proof}

\begin{prop}\label{DecompositionDLVarU0}
We have 
\begin{equation*}
X_{\mathbf{\Lambda}_0}^{\mathbbm h_0} = \bigsqcup_{w_1,\ldots,w_r} X_J\{w_1w_2\ldots w_r\} 
\end{equation*}
where $w_1,\ldots , w_r$ run over all the permutations of the form $w_i = s_{k_i+1}s_{k_i+2}\ldots s_{k_i+t_i}$ for some $0 \leq t_i \leq k_{i+1}-k_i$ when $1\leq i \leq r-1$, and where $w_r = s_{k_r+1}\ldots s_{k_r+t_r}$ for some $0 \leq t_r \leq 2(l-1) + h_{i_1}-k_r$.
\end{prop}

\begin{proof}
This is a direct application of Theorem \ref{ClosureFineDLVariety} and Proposition \ref{StratificationFineDLVarietiesUnitary}. 
\end{proof}

For future reference in Section \ref{Section4.1}, we introduce the following subvariety. 

\begin{defi}\label{BTStratumDLVarU0}
Let $r_0 := r-1$ if $l>1$ and $r_0 := r$ if $l=1$, so that we have $k_{r_0} = l + h_{i_1} - \Delta h_{i_1-1} - 1$. We define a subvariety
\begin{equation*}
X_{\mathbf{\Lambda}_0}^{\mathbbm h_0,0} := \bigsqcup_{w_1,\ldots ,w_{r_0}} X_{J}\{w_1\ldots w_{r_0}s_{l+h_{i_1}}\ldots s_{2(l-1) + h_{i_1}}\} \hookrightarrow X_{\mathbf{\Lambda}_0}^{\mathbbm h_0},
\end{equation*}
where $w_1,\ldots,w_{r_0}$ are as in Proposition \ref{DecompositionDLVarU0} and such that 
\begin{itemize}
\item $t_i+t_{r_0+1-i} \geq \Delta h_{i_1-i}$ for all $1 \leq i \leq i_1-1$,
\item $2t_{i_1} \geq h_1$ if $h_1 \not = 0$.
\end{itemize}
\end{defi}

\begin{prop}
Let $k$ be a field extension of $\mathbb F_{q^2}$. We have 
\begin{equation*}
\hspace{-2cm} X_{\mathbf{\Lambda}_0}^{\mathbbm h_0}(k) = \left\{\begin{array}{c}
\{0\} \subset U_{i_1} \subset \ldots \subset W_{i_1} \subset (V_{\Lambda_0^{i_1}}^{0})_k \\
\text{partial flags of type } \mathbf d_J
\end{array} \middle| 
\begin{tikzcd}[column sep=small, row sep=small]
 W_{i_1}^{\perp} \arrow[d,symbol=\subset, outer sep=2pt,"1"] \arrow[r,symbol=\subset] & \ldots \arrow[r,symbol=\subset] & W_1^{\perp} \arrow[d,symbol=\subset, outer sep=2pt,"1"] \arrow[r,symbol=\subset] & U_1^{\perp} \arrow[d,symbol=\subset, outer sep=2pt,"1"] \arrow[r,symbol=\subset] & \ldots \arrow[r,symbol=\subset] & U_{i_1}^{\perp} \arrow[d,symbol=\subset, outer sep=2pt,"1"] \\
U_{i_1} \arrow[r,symbol=\subset] & \ldots \arrow[r,symbol=\subset] & U_1 \arrow[r,symbol=\subset] & W_1 \arrow[r,symbol=\subset] & \ldots \arrow[r,symbol=\subset] & W_{i_1} 
\end{tikzcd} 
\right\}.
\end{equation*}
\end{prop}

\begin{proof}
By definition of fine Deligne-Lusztig varieties, a partial flag $\mathcal G$ of type $\mathbf d_J$ lies in $X_J\{w\}(k)$ for some $w\in {}^J\mathbf W$, if and only if there exists a complete flag 
\begin{equation*}
\mathcal F: \{0\} = \mathcal F_0 \subset \mathcal F_1 \subset \ldots \subset \mathcal F_{t_0^{i_1}-1} \subset \mathcal F_{t_0^{i_1}} = (V_{\Lambda_0^{i_1}}^{0})_k,
\end{equation*}
which is of relative position $w$ with respect to $\mathcal F^{\perp}$, and such that removing the terms $\mathcal F_i$ for all $i$ such that $s_i \in J$ results in the original partial flag $\mathcal G$. Let us write 
\begin{equation*}
\mathcal G: \{0\} \subset U_{i_1} \subset \ldots \subset W_{i_1} \subset (V_{\Lambda_0^{i_1}}^{0})_k,
\end{equation*}
and let $\mathcal F$ be a complete flag which lifts $\mathcal G$ as above. Denote by $w \in {}^J\mathbf W$ the relative position of $\mathcal F$ and $\mathcal F^{\perp}$. The partial flag $\mathcal G$ belongs to the set on the RHS of the Proposition if and only if we have the inclusions $W_i^{\perp} \subset U_i$ and $U_i^{\perp} \subset W_i$. These, in turn, are equivalent to the conditions 
\begin{equation*}
w(\{1,\ldots ,k_i\}) \subset \{1,\ldots ,k_{i}+1\},
\end{equation*}
for all $1 \leq i \leq r$. By Proposition \ref{FineDLVartietiesInUnitary}, this is equivalent to $w$ being of the form $w=w_1w_2\ldots w_r$ as specified in Proposition \ref{DecompositionDLVarU0}. In other words, we have proved that the set on the RHS coincides with $\bigsqcup_{w_1,\ldots,w_r} X_J\{w_1w_2\ldots w_r\}(k) = X_{\mathbf{\Lambda}_0}^{\mathbbm h_0}(k)$, which concludes the proof.
\end{proof}

\underline{If $i_s \not = m$:} consider the hermitian space $V_{\Lambda_1^{i_s}}^{0} := \Lambda_1^{i_s}/\pi^2(\Lambda_1^{i_s})^{\vee}$. Similarly to the previous paragraph, we consider Deligne-Lusztig varieties with respect to the unitary group $\mathrm U(V_{\Lambda_1^{i_s}}^0,\{\cdot,\cdot\})$. We have $\dim(V_{\Lambda_1^{i_s}}^0) = t(\Lambda_1^{i_s}) = 2(l-1)+(n-h_{i_s+1})+1$ for some $l\geq 1$. Consider the Weyl group $\mathbf W \simeq S_{2(l-1)+(n-h_{i_s+1})+1}$ and the set of simple reflections $\mathbf S = \{s_1,\ldots ,s_{2(l-1)+(n-h_{i_s+1})}\}$. We consider $J \subset \mathbf S$ such that $\mathbf d_J$ corresponds to flags of the following type.
\begin{equation*}
\hspace{-1cm} \{0\} \overset{l}{\subset} W_{1} \overset{\Delta h_{i_s+1}}{\subset} \ldots \overset{\Delta h_{m-1}}{\subset} W_{k-i_s} \overset{n-h_m}{\subset} U_{k-i_s} \overset{\Delta h_{m-1}}{\subset} \ldots \overset{\Delta h_{i_s+1}}{\subset} U_{1} \overset{l-1}{\subset} V_{\Lambda_1^{i_s}}^0
\end{equation*}
Let $r := \#(\mathbf S \setminus J)$ and for simplicity, let us write 
\begin{equation*}
J = \mathbf S \setminus \{s_{k_1+1},\ldots , s_{k_r+1}\},
\end{equation*}
for some $0 \leq k_1 < \ldots < k_r < 2(l-1)+(n-h_{i_s+1})$. We define
\begin{equation*}
Y_{\mathbf{\Lambda}_s}^{\mathbbm h_s} := X_J\{s_{l}s_{l+1} \ldots s_{2(l-1)+(n-h_{i_s+1})}\},
\end{equation*}
and $X_{\mathbf{\Lambda}_s}^{\mathbbm h_s} := \overline{Y_{\mathbf{\Lambda}_s}^{\mathbbm h_s}}$, where the closure is taken inside the flag variety of type $J$. Just as for the variety $X_{\mathbf{\Lambda}_0}^{\mathbbm h_0}$, the following statements hold. 

\begin{prop}\label{SmoothnessDLVar2}
The variety $X_{\mathbf{\Lambda}_s}^{\mathbbm h_s}$ is projective, smooth and geometrically irreducible of dimension $l+(n-h_{i_s+1})-1$.
\end{prop}

\begin{prop}\label{DecompositionDLVarUs}
We have 
\begin{equation*}
X_{\mathbf{\Lambda}_s}^{\mathbbm h_s} = \bigsqcup_{w_1,\ldots,w_r} X_J\{w_1w_2\ldots w_r\},
\end{equation*}
where $w_1,\ldots , w_r$ run over all the permutations of the form $w_i = s_{k_i+1}s_{k_i+2}\ldots s_{k_i+t_i}$ for some $0 \leq t_i \leq k_{i+1}-k_i$ when $1\leq i \leq r-1$, and $w_r = s_{k_r+1}\ldots s_{k_r+t_r}$ for some $0 \leq t_r \leq 2(l-1)+(n-h_{i_s+1})-k_r$.
\end{prop}

For future reference, we introduce the following subvariety.

\begin{defi}\label{BTStratumDLVarUs}
Let $r_0 := r-1$ if $l>1$ and $r_0 := r$ if $l=1$, so that we have $k_{r_0} = l + (n-h_{i_s+1}) - \Delta h_{i_s+1} - 1$. We define a subvariety
\begin{equation*}
X_{\mathbf{\Lambda}_s}^{\mathbbm h_s,0} := \bigsqcup_{w_1,\ldots,w_{r_0}} X_J\{w_1\ldots w_{r_0}s_{l+(n-h_{i_s+1})}\ldots s_{2(l-1)+(n-h_{i_s+1})}\} \hookrightarrow X_{\mathbf{\Lambda}_s}^{\mathbbm h_s},
\end{equation*}
where $w_1,\ldots,w_{r_0}$ are as in Proposition \ref{DecompositionDLVarUs} and such that
\begin{itemize}
\item $t_i+t_{r_0+1-i} \geq \Delta h_{i_s+i}$ for all $1\leq i \leq m-1-i_s$,
\item $2t_{k-i_s} \geq n-h_m$ if $h_m \not = n$. 
\end{itemize}
\end{defi}

\begin{prop}
Let $k$ be a field extension of $\mathbb F_{q^2}$. We have 
\begin{equation*}
\hspace{-2cm} X_{\mathbf{\Lambda}_s}^{\mathbbm h_s}(k) = \left\{\begin{array}{c}
\{0\} \subset W_{1} \subset \ldots \subset U_{1} \subset (V_{\Lambda_1^{i_s}}^{0})_k \\
\text{partial flags of type } \mathbf d_J
\end{array} \middle| 
\begin{tikzcd}[column sep=small, row sep=small]
 U_{1}^{\perp} \arrow[d,symbol=\subset, outer sep=2pt,"1"] \arrow[r,symbol=\subset] & \ldots \arrow[r,symbol=\subset] & U_{k-i_s}^{\perp} \arrow[d,symbol=\subset, outer sep=2pt,"1"] \arrow[r,symbol=\subset] & W_{k-i_s}^{\perp} \arrow[d,symbol=\subset, outer sep=2pt,"1"] \arrow[r,symbol=\subset] & \ldots \arrow[r,symbol=\subset] & W_{1}^{\perp} \arrow[d,symbol=\subset, outer sep=2pt,"1"] \\
W_{1} \arrow[r,symbol=\subset] & \ldots \arrow[r,symbol=\subset] & W_{k-i_s} \arrow[r,symbol=\subset] & U_{k-i_s} \arrow[r,symbol=\subset] & \ldots \arrow[r,symbol=\subset] & U_{1} 
\end{tikzcd} 
\right\}.
\end{equation*}
\end{prop}

\underline{For $1\leq j \leq s-1$:} consider the hermitian space $V_{\Lambda_1^{i_j}}^{0} := \Lambda_1^{i_j}/\pi^2(\Lambda_1^{i_j})^{\vee}$ over $\mathbb F_{q^2}$, as defined in Section \ref{Section2.2}. Let $V_1 := \pi\Lambda_0^{i_{j+1}}/\pi^2(\Lambda_1^{i_j})^{\vee}$ and $V_2 := \Lambda_1^{i_j}/\pi(\Lambda_0^{i_{j+1}})^{\vee}$, so that $V_1$ is a subspace of $V_{\Lambda_1^{i_j}}^{0}$ and $V_2$ is a quotient of it. The spaces $V_1$ and $V_2$ are of the same dimension $d$, where
\begin{equation*}
d := \dim(V_1) = \dim(V_2) = \frac{t(\Lambda_1^{i_j})-(n-t(\Lambda_0^{i_{j+1}}))}{2} > 0.
\end{equation*}
Write $t(\Lambda_1^{i_j}) = 2(l_1-1)+(n-h_{i_{j}+1})+1$ and $t(\Lambda_0^{i_{j+1}}) = 2(l_0-1)+h_{i_{j+1}} + 1$ for some $l_0,l_1 \geq 1$. Then, we can rewrite $d$ as 
\begin{equation*}
d = (l_0+l_1-1) + \frac{h_{i_{j+1}}-h_{i_j+1}}{2}.
\end{equation*}
Moreover, the hermitian pairing on $V_{\Lambda_1^{i_j}}^0$ induces a perfect bilinear pairing $B:V_1\times V_2^{(q)} \to \mathbb F_{q^2}$. We will consider Deligne-Lusztig for $\mathrm{GL}(V_1)$ in the context of the fake unitary case, as in Section \ref{DL:FakeUnitary}. Consider the Weyl group $\mathbf W := S_d \times S_d$ and the set of simple reflections $\mathbf S = \{(s_i,\mathrm{id}),(\mathrm{id},s_i), 1\leq i \leq d-1\}$. We consider $J = J_1 \sqcup J_2 \subset \mathbf S$ such that $\mathbf d_J$ corresponds to flags of the following type. 
\begin{align*}
& \{0\} \overset{l_1}{\subset} W_{1} \overset{\Delta h_{i_j+1}}{\subset} W_2 \subset \ldots \overset{\Delta h_{i_{j+1}-1}}{\subset} W_{i_{j+1}-i_j} \overset{l_0-1}{\subset} V_1, \\
& \{0\} \overset{l_0}{\subset} U_{i_{j+1}-i_j} \overset{\Delta h_{i_{j+1}-1}}{\subset} U_{i_{j+1}-i_j-1} \subset \ldots \overset{\Delta h_{i_j+1}}{\subset} U_{1} \overset{l_1-1}{\subset} V_2.
\end{align*}

To simplify the notations, let us write 
\begin{align*}
J_1 = \mathbf S \setminus \{s_{k_1+1},\ldots ,s_{k_r+1}\}, & & J_2 = \mathbf S \setminus \{s_{k_1'+1},\ldots ,s_{k_{r'}'+1}\},
\end{align*}
for some $0 \leq k_1 < \ldots < k_r < d-1$ and $0 \leq k_1' < \ldots < k_{r'}' < d-1$. We note that $r=r'$ if either $l_0=l_1=1$ either $l_0,l_1>1$, $r=r'-1$ if $l_0=1$ and $l_1>0$, and $r=r'+1$ if $l_0>1$ and $l_1=1$. We define 
\begin{equation*}
Y_{\mathbf{\Lambda}_j}^{\mathbbm h_j} := X_{J}\{(s_{l_1}s_{l_1+1}\ldots s_{d-1},s_{l_0}s_{l_0+1}\ldots s_{d-1})\},
\end{equation*}
and $X_{\mathbf{\Lambda}_j}^{\mathbbm h_j} := \overline{Y_{\mathbf{\Lambda}_j}^{\mathbbm h_j}}$ where the closure is taken inside the flag variety of type $J = J_1 \sqcup J_2$. 

\begin{prop}\label{SmoothnessDLVar3}
The variety $X_{\mathbf{\Lambda}_j}^{\mathbbm h_j}$ is projective, smooth and geometrically irreducible of dimension $l_0+l_1+(h_{i_{j+1}}-h_{i_j+1})-2$.
\end{prop}

\begin{proof}
Let us write $w = (w_1,w_2) := (s_{l_1}s_{l_1+1}\ldots s_{d-1},s_{l_0}s_{l_0+1}\ldots s_{d-1})$. The dimension of $X_{\mathbf{\Lambda}_j}^{\mathbbm h_j}$ is just $\ell(w) = \ell(w_1) + \ell(w_2)$. We prove smoothness and irreducibility by the same method as Proposition \ref{SmoothnessDLVar}. Let $w_1' := s_{l_1}s_{l_1+1}\ldots s_{d-l_0}$ and $w_2' := s_{l_0}s_{l_0+1}\ldots s_{d-l_1}$, so that $w' = (w_1',w_2')$ is the shortest element of $\mathbf W_J w \mathbf W_{F(J)}$. Since the coarse Deligne-Lusztig variety $X_J(w')$ is irreducible and of dimension $\ell(w)$ (as can be checked by distinguishing cases), the natural inclusion $Y_{\mathbf{\Lambda}_j}^{\mathbbm h_j} \subset X_J(w')$ induces an equality $X_{\mathbf{\Lambda}_j}^{\mathbbm h_j} = \overline{X_{J}(w')}$. Then, the closure $\overline{X_{J}(w')}$ is smoothly equivalent to the Schubert variety for the product $\mathrm{GL}_d \times \mathrm{GL}_d$ in the full flag variety associated to the longest element of $\mathbf W_J w' \mathbf W_{F(J)}$. This element can be written as $xw'y$ where $y$ is the longest element of $\mathbf W_{F(J)}$ and $x$ is the longest element of $\mathbf W_J \cap \mathbf W^{J\cap {}^{w'}F(J)}$. One may check that $x = (x_1,x_2)$ where $x_1 = s_1\ldots s_{l_1-1}$ and $x_2 = s_1\ldots s_{l_0-1}$. If $l_1 > 1$, the permutation $x_1w'_1y_1$ can be written as 
\begin{equation*}
\hspace{-1cm} \left(
\begin{array}{cccccccccccccc}
1 & 2 & \ldots & k_1 & k_1 + 1 & \ldots & k_2 & \ldots & k_r & k_r+1 & k_r+2 & \ldots & d-1 & d \\
k_1+1 & k_1 & \ldots & 2 & k_2 + 1 & \ldots & k_1+2 & \ldots & k_{r-1}+2 & d & d - 1 & \ldots & k_r + 2 & 1
\end{array} 
\right)
\end{equation*}
and if $l_1 = 1$, the same formula holds with $k_1$ replaced by $k_2$, $k_2$ replaced by $k_3$ and so on. The same goes for $x_2w_2'y_2$ with $k_i$ replaced by $k_i'$ for all $i$. These permutations avoid the patterns $(3412)$ and $(4231)$, from which smoothness follows.
\end{proof}

\begin{prop}\label{DecompositionDLVarGLj}
We have 
\begin{equation*}
X_{\mathbf{\Lambda}_j}^{\mathbbm h_j} = \bigsqcup_{\substack{w_1,\ldots ,w_r \\ w_1',\ldots,w_{r'}'}} X_{J}\{(w_1w_2\ldots w_r,w_1'w_2'\ldots w_{r'}')\},
\end{equation*}
where $w_1,\ldots,w_r$ and $w_1',\ldots,w_{r'}'$ run over all the permutations of the form $w_i = s_{k_i+1}\ldots s_{k_i+t_i}$ and $w_j' = s_{k_j'+1}\ldots s_{k_j'+t_j'}$ when $1 \leq i \leq r-1$ and $1\leq j \leq r'-1$ for some $0\leq t_i \leq k_{i+1}-k_i$ and $0 \leq t_j' \leq k_{j+1}' - k_j'$, and where $w_r = s_{k_r+1} \ldots s_{k_r+t_r}$ and $w_{r'} = s_{k'_{r'}+1}\ldots s_{k'_{r'}+t_{r'}'}$ for some $0 \leq t_r \leq n-1-k_r$ and $0 \leq t_{r'}' \leq n-1-k_{r'}'$.
\end{prop}

\begin{proof}
This is a direct application of Theorem \ref{ClosureFineDLVariety} and Proposition \ref{StratificationFineDLVarietiesFakeUnitary}.
\end{proof}

For future reference, we introduce the following subvariety.

\begin{defi}\label{BTStratumDLVarGLj}
Let $r_0 := r-1$ if $l_0>1$ and $r_0 := r$ if $l_0=1$, so that we have $k_{r_0} = l_1 + \frac{h_{i_{j+1}-1}-h_{i_j+1}}{2}-1$. Likewise, let $r_0' := r'-1$ if $l_1>1$ and $r_0' := r'$ if $l_1=1$, so that we have $k'_{r'_0} = l_0 + \frac{h_{i_{j+1}}-h_{i_j+2}}{2}-1$. We define a subvariety 
\begin{equation*}
X_{\mathbf{\Lambda}_j}^{\mathbbm h_j,0} := \bigsqcup_{\substack{w_1,\ldots,w_{r_0}\\w'_1,\ldots,w'_{r_0'}}} X_J\{(w_1\ldots w_{r_0}s_{d-l_0+1}\ldots s_{d-1},w'_1\ldots w'_{r_0'}s_{d-l_1+1}\ldots s_{d-1})\} \hookrightarrow X_{\mathbf{\Lambda}_j}^{\mathbbm h_j},
\end{equation*}
where $w_1,\ldots ,w_{r_0}$ and $w'_1,\ldots,w'_{r_0'}$ are as in Proposition \ref{DecompositionDLVarGLj} and such that $t_i + t'_{r_0'+1-i} \geq \Delta h_{i_j+i}$ for all $1\leq i \leq i_{j+1}-i_{j}-1$.
\end{defi}

\begin{prop}
Let $k$ be a field extension of $\mathbb F_{q^2}$. We have 
\begin{equation*}
\hspace{-2.5cm} X_{\mathbf{\Lambda}_j}^{\mathbbm h_j}(k) = \left\{ \begin{array}{c}
\{0\} \subset W_{1} \subset \ldots \subset W_{i_{j+1}-i_j} \subset (V_1)_k,\\
\{0\} \subset U_{i_{j+1}-i_j} \subset \ldots \subset U_1 \subset (V_2)_k,\\
\text{partial flags of type } \mathbf d_J
\end{array} \middle| 
\begin{tikzcd}[column sep=small, row sep=small]
 U_{1}^{\perp} \arrow[d,symbol=\subset, outer sep=2pt,"1"] \arrow[r,symbol=\subset] & \ldots \arrow[r,symbol=\subset] & U_{i_{j+1}-i_j}^{\perp} \arrow[d,symbol=\subset, outer sep=2pt,"1"] \\
W_{1} \arrow[r,symbol=\subset] & \ldots \arrow[r,symbol=\subset] & W_{i_{j+1}-i_j} 
\end{tikzcd} 
\text{ and }
\begin{tikzcd}[column sep=small, row sep=small]
 W_{i_{j+1}-i_j}^{\perp} \arrow[d,symbol=\subset, outer sep=2pt,"1"] \arrow[r,symbol=\subset] & \ldots \arrow[r,symbol=\subset] & W_{1}^{\perp} \arrow[d,symbol=\subset, outer sep=2pt,"1"] \\
U_{i_{j+1}-i_j} \arrow[r,symbol=\subset] & \ldots \arrow[r,symbol=\subset] & U_1 
\end{tikzcd} 
\right\}.
\end{equation*}
\end{prop}

Finally, we are ready to define the variety $X_{I,\mathbf{\Lambda}}^{\mathbbm h}$. 

\begin{defi}
Let $(I,\mathbf{\Lambda})$ be a Bruhat-Tits index. Write $0\leq i_1 < \ldots < i_s \leq m$ for the elements of $I$ where $s\geq 1$. We define 
\begin{equation*}
X_{I,\mathbf{\Lambda}}^{\mathbbm h} := \begin{cases}
X_{\mathbf{\Lambda}_0}^{\mathbbm h_0} \times X_{\mathbf{\Lambda}_1}^{\mathbbm h_1} \times \ldots \times X_{\mathbf{\Lambda}_{s-1}}^{\mathbbm h_{s-1}} \times X_{\mathbf{\Lambda}_s}^{\mathbbm h_s} & \text{if } i_1 > 0 \text{ and } i_s < m,\\
X_{\mathbf{\Lambda}_0}^{\mathbbm h_0} \times X_{\mathbf{\Lambda}_1}^{\mathbbm h_1} \times \ldots \times X_{\mathbf{\Lambda}_{s-1}}^{\mathbbm h_{s-1}} \phantom{\times X_{\mathbf{\Lambda}_s}^{\mathbbm h_s}} & \text{if } i_1 > 0 \text{ and } i_s = m,\\
\phantom{X_{\mathbf{\Lambda}_0}^{\mathbbm h_0} \times {}} X_{\mathbf{\Lambda}_1}^{\mathbbm h_1} \times \ldots \times X_{\mathbf{\Lambda}_{s-1}}^{\mathbbm h_{s-1}} \times  X_{\mathbf{\Lambda}_s}^{\mathbbm h_s} & \text{if } i_1 = 0 \text{ and } i_s < m,\\
\phantom{X_{\mathbf{\Lambda}_0}^{\mathbbm h_0} \times {}}X_{\mathbf{\Lambda}_1}^{\mathbbm h_1} \times \ldots \times X_{\mathbf{\Lambda}_{s-1}}^{\mathbbm h_{s-1}} \phantom{\times X_{\mathbf{\Lambda}_s}^{\mathbbm h_s}} & \text{if } i_1 = 0 \text{ and } i_s = m.
\end{cases}
\end{equation*}
We also define $Y_{I,\mathbf{\Lambda}}^{\mathbbm h}$ and $X_{I,\mathbf{\Lambda}}^{\mathbbm h,0}$ similarly, by replacing the $X_{\mathbf{\Lambda}_j}^{\mathbbm h_j}$ with $Y_{\mathbf{\Lambda}_j}^{\mathbbm h_j}$ and with $X_{\mathbf{\Lambda}_j}^{\mathbbm h_j,0}$ respectively.
\end{defi}

\begin{rk}
The variety $Y_{I,\mathbf{\Lambda}}^{\mathbbm h}$ itself is a fine Deligne-Lusztig variety for a certain Levi complement in the group $\mathrm U(\Lambda_1^{i_1}/\pi^2(\Lambda_1^{i_1})^{\vee}) \times \mathrm U(\pi(\Lambda_1^{i_1})^{\vee}/\Lambda_1^{i_1})$. We have open immersions 
\begin{equation*}
Y_{I,\mathbf{\Lambda}}^{\mathbbm h} \hookrightarrow X_{I,\mathbf{\Lambda}}^{\mathbbm h,0} \hookrightarrow X_{I,\mathbf{\Lambda}}^{\mathbbm h},
\end{equation*}
and $Y_{I,\mathbf{\Lambda}}^{\mathbbm h}$ is dense in $X_{I,\mathbf{\Lambda}}^{\mathbbm h}$.
\end{rk}

\begin{prop}\label{BTStrataAreSmooth}
The variety $X_{I,\mathbf{\Lambda}}^{\mathbbm h}$ is projective, smooth and geometrically irreducible.
\end{prop}

The dimension of $X_{I,\mathbf{\Lambda}}^{\mathbbm h}$ can be computed by summing the formulas of Propositions \ref{SmoothnessDLVar}, \ref{SmoothnessDLVar2} and \ref{SmoothnessDLVar3}. 

\begin{ex}
Assume that $m = 1$ and write $h := h_1$. The dimension of $X_{I,\mathbf{\Lambda}}^{\mathbbm h}$ can be computed as follows:
\begin{itemize}
\item if $h \not = n$, $I = \{1\}$ and $\mathbf{\Lambda} = \{\Lambda_0\}$ for some $\Lambda_0 \in \mathcal L_0^{\geq h+1}$, then $\dim(X_{I,\mathbf{\Lambda}}^{\mathbbm h}) = \tfrac{t(\Lambda_0)+h-1}{2}$,
\item if $h \not = 0$, $I = \{0\}$ and $\mathbf{\Lambda} = \{\Lambda_1\}$ for some $\Lambda_1 \in \mathcal L_1^{\geq n-h+1}$, then $\dim(X_{I,\mathbf{\Lambda}}^{\mathbbm h}) = \tfrac{t(\Lambda_1)+(n-h)-1}{2}$,
\item if $0 < h < n$, $I = \{0,1\}$ and $\mathbf{\Lambda} = \{\Lambda_0,\Lambda_1\}$ for some $\Lambda_0 \in \mathcal L_0^{\geq h+1}$ and $\Lambda_1 \in \mathcal L_1^{\geq n-h+1}$ such that $\pi\Lambda_1^{\vee} \subset \Lambda_0$, then $\dim(X_{I,\mathbf{\Lambda}}^{\mathbbm h}) = \tfrac{t(\Lambda_0) + t(\Lambda_1) - n}{2} - 1$. 
\end{itemize}
The formula in the first two cases coincide with those given in \cite{cho} Proposition 3.11. 
\end{ex}

\begin{prop}\label{BijectionOnPoints}
Let $(I,\mathbf{\Lambda})$ be a Bruhat-Tits index and let $k$ be an algebraically closed field containing $\kappa_{\breve E}$. There is a bijection
\begin{equation*}
\mathcal N_{I,\mathbf{\Lambda}}^{\mathbbm h}(k) \simeq X_{I,\mathbf{\Lambda}}^{\mathbbm h}(k).
\end{equation*}
\end{prop}

\begin{proof}
Let $(A_m \subset \ldots \subset B_m) \in \mathcal N_{I,\mathbf{\Lambda}}^{\mathbbm h}(k)$. If $i_1 \not = 0$, for $1 \leq i \leq i_1$ we map $A_i$ to $U_i := A_i / \pi(\Lambda_0^{i_1})^{\vee}_k$ and $B_i$ to $W_i := B_i /  \pi(\Lambda_0^{i_1})^{\vee}_k$, thus defining a point $(U_{i_1}\subset \ldots \subset W_{i_1}) \in X_{\mathbf{\Lambda}_0}^{\mathbbm h_0}(k)$. If $i_s \not = m$, for $1 \leq i \leq m-i_s$ we map $\pi B_{i_s+i}$ to $W_i := \pi B_{i_s+i} / \pi^{2}(\Lambda_1^{i_s})^{\vee}_k$ and $A_{i_s+1}$ to $U_i := A_{i_s+i} / \pi^{2}(\Lambda_1^{i_s})^{\vee}_k$, thus defining a point $(W_{1}\subset\ldots\subset U_{1})\in X_{\mathbf{\Lambda}_s}^{\mathbbm h_s}(k)$. Eventually, for $1 \leq j \leq s-1$ and for $1 \leq i \leq i_{j+1}-i_j$, we map $\pi B_{i_j+i}$ to $W_i := \pi B_{i_j+i}/\pi^2(\Lambda_1^{i_j})^{\vee}_k$ and $A_{i_j+i}$ to $U_i := A_{i_j+i} / \pi(\Lambda_0^{i_{j+1}})^{\vee}_k$, thus defining a point $(W_1 \subset \ldots \subset W_{i_{j+1}-i_j}, U_{i_{j+1}-i_j} \subset \ldots \subset U_1) \in X_{\mathbf{\Lambda}_j}^{\mathbbm h_j}(k)$. This mapping is bijective by construction.
\end{proof}

\subsection{The isomorphism $\mathcal N_{I,\mathbf{\Lambda}}^{\mathbbm h} \simeq X_{I,\mathbf{\Lambda}}^{\mathbbm h}\times \kappa_{\breve E}$}

In this section, given a Bruhat-Tits index $(I,\mathbf{\Lambda})$, we build a morphism $f: \mathcal N_{I,\mathbf{\Lambda}}^{\mathbbm h} \to X_{I,\mathbf{\Lambda}}^{\mathbbm h} \times \kappa_{\breve E}$ and prove that it is an isomorphism. The construction of $f$ is now somewhat standard, so that we closely follow \cite{vw2} Section 4.7 and \cite{cho} Section 3.6. For a $\kappa_{E}$-scheme $S$ and a strict formal $\mathcal O_F$-module $X$ over $S$, the Lie algebra of the universal extension of $X$ will be denoted by $D(X)$, as defined in \cite{ACZ}. In general $D(X)$ is a locally free $\mathcal O_S$-module, and if $S = \mathrm{Spec}(k)$ for some perfect field over $\kappa_{E}$, then $D(X) = M(X)/\pi M(X)$ where $M(X)$ is the relative Dieudonné module of $X$. Recall the following statement from \cite{vw2} Corollary 4.7 (which is stated for $p$-divisible groups but works similarly for strict formal $\mathcal O_F$-modules).

\begin{prop}\label{LocallyDirectSummand}
Let $S$ be a $\kappa_E$-scheme and let $\rho_1:X \to Y_1$ and $\rho_2:X\to Y_2$ be two isogenies of strict formal $\mathcal O_F$-modules over $S$, such that $\mathrm{Ker}(\rho_1) \subset \mathrm{Ker}(\rho_2) \subset X[\pi]$. Then $\mathrm{Ker}(D(\rho_1))$ is locally a direct summand of the locally free $\mathcal O_S$-module $\mathrm{Ker}(D(\rho_2))$, and the formation of $\mathrm{Ker}(D(\rho_i))$ for $i=1,2$ is compatible with base change $S'\to S$.
\end{prop}

Let $X=(X^{[i]},i_{X^{[i]}},\lambda_{X^{[i]}},\rho_{X^{[i]}})_{1\leq i \leq m} \in \mathcal N_{I,\mathbf{\Lambda}}^{\mathbbm h}(R)$ for some Bruhat-Tits index $(I,\mathbf{\Lambda})$ and some $\kappa_{\breve E}$-algebra $R$. For each $i \in I\setminus \{0\}$ and for all $j \in I \setminus \{m\}$, the compositions
\begin{center}
\begin{tikzcd}[row sep = small]
(X_{\Lambda_0^{i-}})_R \arrow[r,"\rho_{\Lambda_0^{i-},X}"] & X^{[i]} \arrow[r,"\rho_{X,\Lambda_0^{i+}}"] & (X_{\Lambda_0^{i+}})_R, \\
(X_{\Lambda_1^{j-}})_R \arrow[r,"\rho_{\Lambda_1^{j-},X}"] & X^{[j+1]} \arrow[r,"\rho_{X,\Lambda_1^{j+}}"] & (X_{\Lambda_1^{j+}})_R,
\end{tikzcd}
\end{center}
coincide with the base change to $R$ of the isogenies $\rho_{\Lambda_0^{i-},\Lambda_0^{i+}}$ and $\rho_{\Lambda_1^{j-},\Lambda_1^{j+}}$ induced from Dieudonné theory by the inclusions $\Lambda_0^{i-} \subset \Lambda_0^{i+}$ and $\Lambda_1^{j-} \subset \Lambda_1^{j+}$. Let us define 
\begin{align*}
\mathbb B_{\Lambda_0^{i}} := \mathrm{Ker}(D(\rho_{\Lambda_0^{i-},\Lambda_0^{i+}})) = \Lambda_0^{i+} / \Lambda_0^{i-}, & & \mathbb B_{\Lambda_1^{j}} := \mathrm{Ker}(D(\rho_{\Lambda_1^{j-},\Lambda_1^{j+}})) = \Lambda_1^{j+}/\Lambda_1^{j-}.
\end{align*}
The quotients $\mathbb B_{\Lambda_0^{i}}$ and $\mathbb B_{\Lambda_1^{j}}$ are $\mathbb F_{q^2}$-vector spaces of dimension respectively $2t(\Lambda_0^i)$ and $2t(\Lambda_1^j)$, equipped with perfect $\mathbb F_{q^2}$-valued alternating forms induced by $\pi\langle\cdot,\cdot\rangle_{[h_i]}$ and by $\langle\cdot,\cdot\rangle_{[h_{j+1}]}$ respectively. We write $\cdot^{\dagger'}$ for the orthogonal complement with respect to these forms. For $1 \leq i' \leq i$ and for $j \leq j' \leq m-1$, recall from Proposition \ref{OneIsogenyTheOtherToo} that $\rho_{X,\Lambda_0^{i+}}$ and $\rho_{\Lambda_1^{j-},X}$ factor respectively through $X^{[i']}$ and through $X^{[j'+1]}$ via isogenies, which we denote by $f_{i',i}$ and $g_{j,j'}$. Thus, we have
\begin{center}
\begin{tikzcd}[row sep = small]
(X_{\Lambda_0^{i-}})_R \arrow[r,"\rho_{\Lambda_0^{i-},X}"] & X^{[i]} \arrow[r,"\widetilde{\alpha}_{i,i'}"] & X^{[i']} \arrow[r,"f_{i',i}"] & (X_{\Lambda_0^{i+}})_R, \\
(X_{\Lambda_1^{j-}})_R \arrow[r,"g_{j,j'}"] & X^{[j'+1]} \arrow[r,"\widetilde{\alpha}_{j'+1,j+1}"] & X^{[j+1]} \arrow[r,"\rho_{X,\Lambda_1^{j+}}"] & (X_{\Lambda_1^{j+}})_R.
\end{tikzcd}
\end{center}
We define
\begin{align*}
E^{i',i} := \mathrm{Ker}(D(\widetilde{\alpha}_{i,i'}\circ \rho_{\Lambda_0^{i-},X})), & & F^{j,j'} := \mathrm{Ker}(D(g_{j,j'})).
\end{align*}
According to Proposition \ref{LocallyDirectSummand}, the $R$-modules $E^{i',i}$ and $F^{j,j'}$ are locally free direct summands respectively of $\mathbb B_{\Lambda_0^{i}} \otimes_{\mathbb F_{q^2}} R$ and of $\mathbb B_{\Lambda_1^{j}} \otimes_{\mathbb F_{q^2}} R$. The $\mathcal O_E$-action on all these modules induces compatible decompositions 
\begin{align*}
\mathbb B_{\Lambda_0^{i}} & = (\mathbb B_{\Lambda_0^{i}})_0 \oplus (\mathbb B_{\Lambda_0^{i}})_1, & \mathbb B_{\Lambda_1^{j}} & = (\mathbb B_{\Lambda_1^{j}})_0 \oplus (\mathbb B_{\Lambda_1^{j}})_1,\\
 (\mathbb B_{\Lambda_0^{i}})_0 & = \alpha_{h_i,h_1}^{-1}(\Lambda_0^{i}/\pi\Lambda_0^{i\vee}), & (\mathbb B_{\Lambda_1^{j}})_0 & = \alpha_{h_{j+1},h_1}^{-1}(\Lambda_1^{j}/\pi^2\Lambda_1^{j\vee}),\\
E^{i',i} & = E^{i',i}_0 \oplus E^{i',i}_1, & F^{j,j'} & = F^{j,j'}_0 \oplus F^{j,j'}_1.
\end{align*}
Here, to make notations more readable, we wrote $\alpha_{h_i,h_1}^{-1}(\Lambda_0^{i}/\pi\Lambda_0^{i\vee})$ instead of the quotient $\alpha_{h_i,h_1}^{-1}(\Lambda_0^{i})/\alpha_{h_i,h_1}^{-1}(\pi\Lambda_0^{i\vee})$, etc. The spaces $(\mathbb B_{\Lambda_0^i})_0$ and $(\mathbb B_{\Lambda_1^j})_0$ are equipped with $\mathbb F_{q^2}/\mathbb F_q$-hermitian forms induced respectively by $\pi\{\cdot,\cdot\}_{[h_i]}$ and by $\{\cdot,\cdot\}_{[h_{j+1}]}$. We write $\cdot^{\perp}$ for the orthogonal complement with respect to respect to these forms. 

\begin{ex}\label{CaseOfGeometricFiber}
Assume that $R = m$ is an algebraically closed field over $\kappa_{\breve E}$. Let $(A_m \subset \ldots B_m)$ be the point corresponding to $X$ via the bijection of Proposition \ref{PointsRZSpaceOverFields}. For $1 \leq i' \leq i$ and for $j \leq j' \leq m-1$ we have 
\begin{align*}
E^{i',i}_0 & = \alpha_{h_i,h_1}^{-1}(A_{i'}/\pi\Lambda_0^{i\vee}), & (E^{i',i}_1)^{\dagger'} & = \alpha_{h_i,h_1}^{-1}(B_{i'}/\pi\Lambda_0^{i\vee}),\\
F^{j,j'}_0 & = \alpha_{h_{j+1},h_1}^{-1}(A_{j'+1}/\pi^2\Lambda_1^{j\vee}), & (F^{j,j'}_1)^{\dagger'} & = \alpha_{h_{j+1},h_1}^{-1}(\pi B_{j'+1}/\pi^2\Lambda_1^{j\vee}).
\end{align*}
\end{ex}
Let us go back to the case of a general $\mathbb F_{q^2}$-algebra $R$. In order to define a map $\mathcal N_{I,\mathbf{\Lambda}}^{\mathbbm h}(R) \to X_{I,\mathbf{\Lambda}}^{\mathbbm h}(R)$ we proceed as follows. Write $0 \leq i_1 < \ldots < i_s \leq m$ for the elements of $I$.\\

\underline{If $i_1 \not = 0$:} we have a diagram as follows:
\begin{center}
\begin{tikzcd}[column sep=small, row sep=small]
\{0\} \arrow[r,symbol=\subset] &  W_{i_1}^{\perp} \arrow[d,symbol=\subset] \arrow[r,symbol=\subset] & \ldots \arrow[r,symbol=\subset] & W_1^{\perp} \arrow[d,symbol=\subset] \arrow[r,symbol=\subset] & U_1^{\perp} \arrow[d,symbol=\subset] \arrow[r,symbol=\subset] & \ldots \arrow[r,symbol=\subset] & U_{i_1}^{\perp} \arrow[d,symbol=\subset] & \\
& U_{i_1} \arrow[r,symbol=\subset] & \ldots \arrow[r,symbol=\subset] & U_1 \arrow[r,symbol=\subset] & W_1 \arrow[r,symbol=\subset] & \ldots \arrow[r,symbol=\subset] & W_{i_1} \arrow[r,symbol=\subset] & (\mathbb B_{\Lambda_0^{i_1}})_{0,R}
\end{tikzcd} 
\end{center}
Here, we put $U_{i'} := E_0^{i',i_1}$ and $W_{i'} := (E_1^{i',i_1})^{\dagger'}$ for all $1\leq i' \leq i_1$. All the $U_{i'}$'s and the $W_{i'}$'s are locally direct summands of $(\mathbb B_{\Lambda_0^i})_{0,R}$, thus so are their orthogonal complements. Each inclusion in the diagram can be checked locally via Nakayama's lemma and Example \ref{CaseOfGeometricFiber}. To check the standard position condition, one must show that any sum of the form $W_{j}^{\perp} + U_{j'}$, $W_{j}^{\perp}+W_{j'}$, $U_j^{\perp} + U_{j'}$ and $U_j^{\perp}+W_{j'}$ is a locally direct summand of $(\mathbb B_{\Lambda_0^i})_{0,R}$. This is either trivial by the inclusion relations, or can be checked locally. For instance, if $j' \leq j$ then $U_{j}^{\perp} + W_{j'}$ is locally either equal to $U_{j}^{\perp}$ or to $W_j$, since $[W_j:U_j^{\perp}] = 1$ on geometric points. By taking the image of all these modules via $\alpha_{h_{i_1},h_1}$, we have built a point in $X_{\mathbf{\Lambda}_0}^{\mathbbm h_0}(R)$.\\

\underline{If $i_s \not = m$:} we have a diagram as follows:
\begin{center}
\begin{tikzcd}[column sep=small, row sep=small]
\{0\} \arrow[r,symbol=\subset] & U_{1}^{\perp} \arrow[d,symbol=\subset] \arrow[r,symbol=\subset] & \ldots \arrow[r,symbol=\subset] & U_{m-i_s}^{\perp} \arrow[d,symbol=\subset] \arrow[r,symbol=\subset] & W_{m-i_s}^{\perp} \arrow[d,symbol=\subset] \arrow[r,symbol=\subset] & \ldots \arrow[r,symbol=\subset] & W_{1}^{\perp} \arrow[d,symbol=\subset] & \\
& W_{1} \arrow[r,symbol=\subset] & \ldots \arrow[r,symbol=\subset] & W_{m-i_s} \arrow[r,symbol=\subset] & U_{m-i_s} \arrow[r,symbol=\subset] & \ldots \arrow[r,symbol=\subset] & U_{1} \arrow[r,symbol=\subset] & (\mathbb B_{\Lambda_1^{i_{s}}})_{0,R}
\end{tikzcd} 
\end{center}
Here, we put $W_{j} := (F_1^{i_s,i_s+j-1})^{\dagger'}$ and $U_{j} := (F_0^{i_s,i_s+j-1})$ for all $1 \leq j \leq m-i_s$. Taking image via $\alpha_{h_{i_s+1},h_1}$, one obtains a point in $X_{\mathbf{\Lambda}_s}^{\mathbbm h_s}(R)$.\\

\underline{For $1 \leq j \leq s-1$:} we have 
\begin{equation*}
\Lambda_1^{i_j-} \subset \alpha_{h_{i_{j+1}},h_{i_j+1}}(\pi\Lambda_0^{i_{j+1}+}) \subset \alpha_{h_{i_{j+1}},h_{i_j+1}}(\Lambda_0^{i_{j+1}-}) \subset \Lambda_1^{i_j+}.
\end{equation*}
This allows us to define $V_1^{j} := \alpha_{h_{i_{j+1}},h_{i_j+1}}(\pi\Lambda_0^{i_{j+1}+}) / \Lambda_1^{i_j-}$ and $V_2^{j} := \Lambda_1^{i_j+}/\alpha_{h_{i_{j+1}},h_{i_j+1}}(\Lambda_0^{i_{j+1}-})$. Thus $V_1^j$ is a subspace of $\mathbb B_{\Lambda_1^{i_j}}$ and $V_2^j$ is a quotient of it. Moreover, they both decompose as $V_1^j = (V_1^j)_0 \oplus (V_1^j)_1$ and $V_2^j = (V_2^j)_0 \oplus (V_2^j)_1$ via the $\mathcal O_E$-action in the usual manner. We have a diagram as follows:
\begin{center}
\hspace{-3cm}
\begin{tikzcd}[column sep=small, row sep=small]
\{0\} \arrow[r,symbol=\subset] & U_{1}^{\perp} \arrow[d,symbol=\subset] \arrow[r,symbol=\subset] & \ldots \arrow[r,symbol=\subset] & U_{i_{j+1}-i_j}^{\perp} \arrow[d,symbol=\subset] & \\
& W_{1} \arrow[r,symbol=\subset] & \ldots \arrow[r,symbol=\subset] & W_{i_{j+1}-i_j} \arrow[r,symbol=\subset] & (V_1^j)_{0,R}
\end{tikzcd} 
\begin{tikzcd}[column sep=small, row sep=small]
\{0\} \arrow[r,symbol=\subset] & W_{i_{j+1}-i_j}^{\perp} \arrow[d,symbol=\subset] \arrow[r,symbol=\subset] & \ldots \arrow[r,symbol=\subset] & W_{1}^{\perp} \arrow[d,symbol=\subset] & \\
& U_{i_{j+1}-i_j} \arrow[r,symbol=\subset] & \ldots \arrow[r,symbol=\subset] & U_1 \arrow[r,symbol=\subset] & (V_2^j)_{0,R}
\end{tikzcd} 
\end{center}
Here, we put $W_{i} := (F_1^{i_j,i_j+i-1})^{\dagger'}$ and $U_{i} := F_0^{i_j,i_j+i-1}$ for all $1 \leq i \leq i_{j+1}-i_j$. Taking image via $\alpha_{h_{i_j+1},h_1}$, one obtains a point of $X_{\mathbf{\Lambda}_j}^{\mathbbm h_j}(R)$.\\ 

We have thus successfully defined a morphism $f:\mathcal N_{I,\mathbf{\Lambda}}^{\mathbbm h} \to X_{I,\mathbf{\Lambda}}\times \kappa_{\breve E}$. We shall now prove that it is an isomorphism. 

\begin{theo}\label{IsomorphismWithDLVariety}
The morphism $f:\mathcal N_{I,\mathbf{\Lambda}}^{\mathbbm h} \to X_{I,\mathbf{\Lambda}}^{\mathbbm h}\times \kappa_{\breve E}$ is an isomorphism.
\end{theo}

\begin{proof}
The proof is classical, we refer to \cite{vw2} Theorem 4.8 and to \cite{cho} Theorem 3.14. According to Proposition \ref{BijectionOnPoints}, the morphism $f$ induces a bijection on $k$-rational points, where $k$ is an algebraically closed field containing $\kappa_{\breve E}$. Thus, $f$ is universally bijective. Now $\mathcal N_{I,\mathbf{\Lambda}}^{\mathbbm h}$ is proper and $X_{I,\mathbf{\Lambda}}^{\mathbbm h}$ is separated, so that $f$ is proper, hence it is a universal homeomorphism. By using the theory of $\mathcal O_F$-windows as in \cite{ACZ}, one may prove that $f$ actually defines on bijection on $k$-rational points for every arbitrary field $k$ containing $\kappa_{\breve E}$ (we omit the details but we refer to \cite{cho} Section 3.5 for an account on $\mathcal O_F$-windows, and how they apply here). Thus $f$ is birational. Thus $f$ finite and birational while $X_{I,\mathbf{\Lambda}}^{\mathbbm h}$ is normal, so that $f$ is an isomorphism by Zariski's main theorem. 
\end{proof}

\begin{corol}
The variety $\mathcal N_{I,\mathbf{\Lambda}}^{\mathbbm h}$ is smooth and geometrically irreducible.
\end{corol}

The dimension of $\mathcal N_{I,\mathbf{\Lambda}}^{\mathbbm h}$ can be computed as explained in the comment following Proposition \ref{BTStrataAreSmooth}.

\section{The Bruhat-Tits stratification}
\subsection{Bruhat-Tits stratification and combinatorial properties} \label{Section4.1}

Recall the notions of inclusion and intersection for Bruhat-Tits indices, see Definition \ref{InclusionBTIndices} and Definition \ref{IntersectionBTIndices}.

\begin{prop}\label{CombinatoricsBTStratification}
Let $(I,\mathbf{\Lambda})$ and $(I',\mathbf{\Lambda'})$ be two Bruhat-Tits indices. 
\begin{enumerate}
\item We have $(I',\mathbf{\Lambda'}) \subset (I,\mathbf{\Lambda}) \iff \mathcal N_{I',\mathbf{\Lambda'}}^{\mathbbm h} \subset \mathcal N_{I,\mathbf{\Lambda}}^{\mathbbm h}$. 
\item We have $\mathcal N_{I',\mathbf{\Lambda'}}^{\mathbbm h} \cap \mathcal N_{I,\mathbf{\Lambda}}^{\mathbbm h} \not = \emptyset$ if and only if the intersection $(I',\mathbf{\Lambda'}) \cap (I,\mathbf{\Lambda}) = (I\cup I',\mathbf{\Lambda''})$ is well-defined. In this case, we have $\mathcal N_{I',\mathbf{\Lambda'}}^{\mathbbm h} \cap \mathcal N_{I,\mathbf{\Lambda}}^{\mathbbm h} = N_{I\cup I',\mathbf{\Lambda''}}^{\mathbbm h}$.
\item For every algebraically closed field $k$ containing $\kappa_{\breve E}$, we have 
\begin{equation*}
\mathcal N_{E/F}^{\mathbbm h}(k) = \bigcup_{I,\mathbf{\Lambda}} \mathcal N_{I,\mathbf{\Lambda}}^{\mathbbm h}(k),
\end{equation*}
where $(I,\mathbf{\Lambda})$ run over all the Bruhat-Tits indices.
\end{enumerate}
\end{prop}

Points 1. and 2. are just scheme-theoretic upgrades of Proposition \ref{InclusionBTStrata} and Proposition \ref{IntersectionBTStrata}. Point 3. is the same as Proposition \ref{RationalPointsCoverEverything}.\\
Given a Bruhat-Tits index $(I,\mathbf{\Lambda})$, we define
\begin{equation*}
\mathcal N_{I,\mathbf{\Lambda}}^{\mathbbm h,0} := \mathcal N_{I,\mathbf{\Lambda}}^{\mathbbm h} \setminus \bigcup_{(I',\mathbf{\Lambda'}) \subsetneq (I,\mathbf{\Lambda})} \mathcal N_{I',\mathbf{\Lambda'}}^{\mathbbm h}.
\end{equation*} 

Recall the notion of Bruhat-Tits type of a point $(A_m \subset \ldots \subset B_m)\in \mathcal N_{E/F}^{\mathbbm h}(k)$, cf. Definition \ref{BTType}, as well as the notation $\Lambda_{A_i}$ and $\Lambda_{B_i}$ from Section 2.3.

\begin{lem}\label{RationalPointsBTStrata}
Let $k$ be an algebraically closed field containing $\kappa_{\breve E}$ and let $(I,\mathbf{\Lambda})$ be a Bruhat-Tits index. Let $(A_m \subset \ldots \subset B_m) \in \mathcal N_{E/F}^{\mathbbm h}(k)$. The following statements are equivalent. 
\begin{enumerate}
\item $(A_m\subset \ldots \subset B_m) \in \mathcal N_{I,\mathbf{\Lambda}}^{\mathbbm h,0}(k)$,
\item $I$ is the Bruhat-Tits type of $(A_m\subset \ldots \subset B_m)$, and for all $i\in I\setminus \{0\}$ and $j\in I\setminus \{m\}$, we have $\Lambda_0^{i} = \Lambda_{B_i}$ and $\Lambda_1^{j} = \Lambda_{A_{j+1}}$. 
\end{enumerate}
\end{lem}

\begin{proof}
Assume 2. By construction, it is clear that $(I,\mathbf{\Lambda})$, as specified in the statement, is a Bruhat-Tits index and that $(A_m\subset \ldots \subset B_m) \in \mathcal N_{I,\mathbf{\Lambda}}^{\mathbbm h}(k)$. Now assume that $(A_m\subset \ldots \subset B_m) \in \mathcal N_{I',\mathbf{\Lambda'}}^{\mathbbm h}(k)$ for some Bruhat-Tits index $(I',\mathbf{\Lambda'}) \subsetneq (I,\mathbf{\Lambda})$. Thus, we have $I \subset I'$ and for all $i \in I\setminus \{0\}$ and $j \in I \setminus \{m\}$, we have $\Lambda_0^{'i} \subset \Lambda_{B_i}$ and $\Lambda_1^{'j} \subset \Lambda_{A_{j+1}}$. However, we also have $B_i \subset (\Lambda_0^{'i})_k$ and $A_{j+1} \subset (\Lambda_1^{'j})_k$ by definition of $\mathcal N_{I',\mathbf{\Lambda'}}^{\mathbbm h}$. This implies that $\Lambda_0^{'i} = \Lambda_{B_i}$ and $\Lambda_1^{'j} = \Lambda_{A_{j+1}}$ for all $i$ and $j$ as above. In particular, we must have $I \subsetneq I'$. Let us fix $i \in I' \setminus I$. \\
If $i = 0$, let $i'$ be the minimum of $I' \setminus \{0\}$. Note that $i'$ exists since $\# I' \geq 2$. By construction, we have 
\begin{equation*}
\pi^2(\Lambda_1^{\prime 0})^{\vee}_k \subset \pi^2 A_1^{\vee} \subset \pi(\Lambda_0^{\prime i'})_k \subset \pi(\Lambda_0^{\prime i'})_k^{\vee} \subset A_1 \subset (\Lambda_1^{\prime 0})_k.
\end{equation*}
It follows that $\pi^2 \Lambda_{A_1}^{\vee} \subset \Lambda_{A_1}$ so that $\Lambda_{A_1} \in \mathcal L_1$, which contradicts $0 \not \in I$. \\
If $i = m$, let $i'$ be the maximum of $I' \setminus \{m\}$. By construction we have 
\begin{equation*}
\pi(\Lambda_0^{\prime m})^{\vee}_k \subset \pi B_m^{\vee} \subset (\Lambda_1^{\prime i'})_k \subset \pi(\Lambda_1^{\prime i'})_k^{\vee} \subset B_m \subset (\Lambda_0^{\prime m})_k.
\end{equation*}
It follows that $\pi \Lambda_{B_m}^{\vee} \subset \Lambda_{B_m}$ so that $\Lambda_{B_m} \in \mathcal L_0$, which contradicts $m \not\in I$.\\
If $0<i<m$, we have 
\begin{equation*}
B_i \subset (\Lambda_0^{\prime i})_k \subset \pi(\Lambda_1^{\prime i})_k \subset \pi A_{i+1}^{\vee},
\end{equation*} 
from which it follows that $\Lambda_{B_i} \subset \pi\Lambda_{A_{i+1}}^{\vee}$, which contradicts $i \not \in I$. Therefore, we have proved $(A_m\subset \ldots \subset B_m) \in \mathcal N_{I,\mathbf{\Lambda}}^{\mathbbm h,0}(k)$.\\
The implication 1. $\implies$ 2. follows from the reverse implication, given that the sets $\mathcal N_{I,\mathbf{\Lambda}}^{\mathbbm h}(k)$ for varying $(I,\mathbf{\Lambda})$ are mutually disjoint.
\end{proof}

\begin{theo}
The isomorphism $f:\mathcal N_{I,\mathbf{\Lambda}}^{\mathbbm h} \xrightarrow{\sim} X_{I,\mathbf{\Lambda}}^{\mathbbm h} \times \kappa_{\breve E}$ induces an isomorphism $\mathcal N_{I,\mathbf{\Lambda}}^{\mathbbm h,0} \xrightarrow{\sim} X_{I,\mathbf{\Lambda}}^{\mathbbm h,0} \times \kappa_{\breve E}$.
\end{theo}

\begin{proof}
We show that $f$ induces a bijection between $\mathcal N_{I,\mathbf{\Lambda}}^{\mathbbm h,0}(k)$ and $X_{I,\mathbf{\Lambda}}^{\mathbbm h,0}(k)$ for all agebraically closed fields $k$ containing $\kappa_{\breve E}$.\\

First, let $(A_m\subset \ldots \subset B_m) \in \mathcal N_{I,\mathbf{\Lambda}}^{\mathbbm h,0}(k)$. Then $I$ is the Bruhat-Tits type of $(A_m\subset \ldots \subset B_m)$, and we have $\Lambda_0^i = \Lambda_{B_i}$ and $\Lambda_1^j = \Lambda_{A_{j+1}}$ for all $i\in I \setminus \{0\}$ and $j\in I \setminus \{m\}$. Let $c_i$ and $d_i$ be defined as in Section \ref{Section2.3} for all $1 \leq i \leq m$.\\

\underline{If $i_1 \not = 0$:} we use the notations of Section \ref{Section3.3}. Let $\mathcal G = (U_{i_1} \subset \ldots \subset W_{i_1}) \in X_{\mathbf{\Lambda}_0}^{\mathbbm h_0}(k)$ be the partial flag in $(V_{\Lambda_{B_{i_1}}}^0)_k$ corresponding to $(A_{i_1} \subset \ldots \subset B_{i_1})$. Let $w = w_1\ldots w_r \in {}^J\mathbf W$ be such that $\mathcal G \in X_{J}\{w\}(k)$, where the $w_i$'s are as in Proposition \ref{DecompositionDLVarU0}. We have $t(\Lambda_{B_{i_1}}) = 2(l-1) + h_{i_1} + 1$ where $l = d_{i_1}$, and we consider a full flag $\mathcal F$ which lifts $\mathcal G$, so that $\mathcal F \in X(w)(k)$ where $X(w)$ is the classical Deligne-Lusztig variety associated to $w$. If $l>1$ so that $k_r = l+h_{i_1}-1$, assume that $t_r < l-1$ towards a contradiction. In this case we have
\begin{center}\hspace*{-2.2cm}
\begin{tikzcd}[sep = small]
\mathcal F_{2(l-1)+h_{i_1}}^{\perp} \arrow[d,equal] \arrow[r,symbol=\subset] & \ldots \arrow[r,symbol=\subset] & \mathcal F_{l+h_{i_1}}^{\perp} \arrow[d,equal] \arrow[r,symbol=\subset] & \ldots \arrow[r,symbol=\subset] & \mathcal F_{l-1}^{\perp}\arrow[r,symbol=\subset] \arrow[d,equal,"/" marking] \arrow[dr,hook, shorten >= 10pt] & \ldots \arrow[r,symbol=\subset] & \mathcal F_{l-t_r}^{\perp} \arrow[d,equal,"/" marking] \arrow[r,symbol=\subset] \arrow[dr,hook, shorten >= 10pt] & \mathcal F_{l-t_r-1}^{\perp} \arrow[d,equal] \arrow[r,symbol=\subset] & \ldots \arrow[r,symbol=\subset] & \mathcal F_1^{\perp} \arrow[d,equal] \\
\mathcal F_1 \arrow[r,symbol=\subset] & \ldots \arrow[r,symbol=\subset] & \mathcal F_{l-1} \arrow[r,symbol=\subset] & \ldots \arrow[r,symbol=\subset] & \mathcal F_{l+h_{i_1}} \arrow[r,symbol=\subset] & \ldots \arrow[r,symbol=\subset] & \mathcal F_{l+h_{i_1}+t_r-1} \arrow[r,symbol=\subset] & \mathcal F_{l+h_{i_1}+t_r} \arrow[r,symbol=\subset] & \ldots \arrow[r,symbol=\subset] & \mathcal F_{2(l-1)+h_{i_1}} 
\end{tikzcd}
\end{center}
Since $\mathcal F_i = \mathcal F^{\perp}_{2(l-1)+h_{i_1}-i+1}$ for all $1 \leq i \leq l-1$, we have $\mathcal F_i^{\perp} = \tau(\mathcal F_{2(l-1)+h_{i_1}-i+1})$ by taking orthogonal complements. Here $\tau = \mathrm{id}\otimes \sigma^{2}$ on $(V_{\Lambda_{B_{i_1}}}^0)_k = V_{\Lambda_{B_{i_1}}}^0 \otimes_{\mathbb F_{q^2}} k$, so that we have $(U^{\perp})^{\perp} = \tau(U)$ for all subspaces $U \subset (V_{\Lambda_{B_{i_1}}}^0)_k$. From the diagram, we deduce that 
\begin{equation*}
\mathcal F_{l+h_{i_1}+t_r} = \mathcal F_{l+h_{i_1}} + \tau(\mathcal F_{l+h_{i_1}}) + \ldots + \tau^{t_r}(\mathcal F_{l+h_{i_1}}),
\end{equation*}
and that $\mathcal F_{l+h_{i_1}+t_r}$ is $\tau$-stable. Since $\mathcal F_{l+h_{i_1}} = B_{i_1}/\pi(\Lambda_{B_{i_1}})_k$, this implies that $T_{t_r+1}(B_{i_1})$ is $\tau$-stable. Thus, we must have $t_r+1 \geq l$, which is a contradiction. Thus, if $l > 1$ then $t_r = l-1$.\\
Let us return to the case $l\geq 1$, and let $r_0$ be defined as in Definition \ref{BTStratumDLVarU0}. Let $1 \leq i \leq i_1 - 1$ so that $k_i = l + \frac{h_{i_1}-h_{i_1+1-i}}{2}-1$ and $k_{r_0+1-i} = l + \frac{h_{i_1}+h_{i_1-i}}{2}-1$. Assume that $t_i + t_{r_0+1-i} < \Delta h_{i_1-i}$ towards a contradiction. In this case we have
\begin{center}\hspace*{-2.5cm}
\begin{tikzcd}[cramped,sep = small]
\mathcal F_{k_{r_0+1-i}+\Delta h_{i_1-i}}^{\perp} \arrow[r,symbol=\subset] \arrow[d,equal,"/" marking] \arrow[dr,hook, shorten >= 10pt] & \mathcal F_{k_{r_0+1-i}+\Delta h_{i_1-i}-1}^{\perp} \arrow[r,symbol=\subset] \arrow[d,equal,"/" marking] \arrow[dr,hook, shorten >= 10pt] & \ldots \arrow[r,symbol=\subset] & \mathcal F_{k_{r_0+1-i}+\Delta h_{i_1-i}-t_i+1}^{\perp} \arrow[r,symbol=\subset] \arrow[d,equal,"/" marking] \arrow[dr,hook, shorten >= 10pt] & \mathcal F_{k_{r_0+1-i}+\Delta h_{i_1-i}-t_i}^{\perp} \arrow[r,symbol=\subset] \arrow[d,equal] & \ldots \arrow[r,symbol=\subset] & \mathcal F_{k_{r_0+1-i}+1}^{\perp} \arrow[d,equal]\\
\mathcal F_{k_i+1} \arrow[r,symbol=\subset] & \mathcal F_{k_i+2} \arrow[r,symbol=\subset] & \ldots \arrow[r,symbol=\subset] & \mathcal F_{k_i+t_i} \arrow[r,symbol=\subset] & \mathcal F_{k_i+t_i+1} \arrow[r,symbol=\subset] & \ldots \arrow[r,symbol=\subset] & \mathcal F_{k_{i}+\Delta h_{i_1-1}}
\end{tikzcd}
\end{center}
\hspace*{\fill} and \hspace*{\fill}
\begin{center}\hspace*{-1cm}
\begin{tikzcd}[cramped,sep = small]
\mathcal F_{k_{i}+\Delta h_{i_1-i}}^{\perp} \arrow[r,symbol=\subset] \arrow[d,equal,"/" marking] \arrow[dr,hook, shorten >= 10pt] & \mathcal F_{k_{i}+\Delta h_{i_1-i}-1}^{\perp} \arrow[r,symbol=\subset] \arrow[d,equal,"/" marking] \arrow[dr,hook, shorten >= 10pt] & \ldots \arrow[r,symbol=\subset] & \mathcal F_{k_{i}+\Delta h_{i_1-i}-t_{r_0+1-i}+1}^{\perp} \arrow[r,symbol=\subset] \arrow[d,equal,"/" marking] \arrow[dr,hook, shorten >= 10pt] & \mathcal F_{k_{i}+\Delta h_{i_1-i}-t_{r_0+1-i}}^{\perp} \arrow[r,symbol=\subset] \arrow[d,equal] & \ldots \arrow[r,symbol=\subset] & \mathcal F_{k_i+1}^{\perp} \arrow[d,equal]\\
\mathcal F_{k_{r_0+1-i}+1} \arrow[r,symbol=\subset] & \mathcal F_{k_{r_0+1-i}+2} \arrow[r,symbol=\subset] & \ldots \arrow[r,symbol=\subset] & \mathcal F_{k_{r_0+1-i}+t_{r_0+1-i}} \arrow[r,symbol=\subset] & \mathcal F_{k_{r_0+1-i}+t_{r_0+1-i}+1} \arrow[r,symbol=\subset] & \ldots \arrow[r,symbol=\subset] & \mathcal F_{k_{r_0+1-i}+\Delta h_{i_1-1}}
\end{tikzcd}
\end{center}
By hypothesis, we have $k_{i}+\Delta h_{i_1-i}-t_{r_0+1-i} \geq  k_i + t_i + 1$, thus according to the first diagram we have $\mathcal F_{k_{i}+\Delta h_{i_1-i}-j} = \mathcal F_{k_{r_0+1-i}+j+1}^{\perp}$ for all $0 \leq j \leq t_{r_0+1-i}-1$. By the second diagram then, we deduce that 
\begin{equation*}
\mathcal F_{k_{r_0+1-i}+t_{r_0+1-i}+1} = \mathcal F_{k_{r_0+1-i}+1} + \tau(\mathcal F_{k_{r_0+1-i}+1}) + \ldots + \tau^{t_{r_0+1-i}}(\mathcal F_{k_{r_0+1-i}+1}),
\end{equation*}
and that $\mathcal F_{k_{r_0+1-i}+t_{r_0+1-i}+1}$ is $\tau$-stable. Since $\mathcal F_{k_{r_0+1-i}+1} = B_{i_1-i}/\pi(\Lambda_{B_{i_1}})_k^{\vee}$, we deduce that $\mathcal F_{k_{r_0+1-i}+t_{r_0+1-i}+1} = \Lambda_{B_{i_1-i}}/\pi(\Lambda_{B_{i_1}})_k^{\vee}$. On the other hand, We have $k_{r_0+1-i}+t_{r_0+1-i}+1 \leq k_{r_0+1-i} + \Delta h_{i_1-i} - t_{i}$. By the first diagram, we have 
\begin{equation*}
\mathcal F_{k_i + \Delta h_{i_1-i}-t_{r_0+1-i}} = \mathcal F_{k_{r_0+1-i}+t_{r_0+1-i}+1}^{\perp} = \pi\Lambda_{B_{i_1-i}}^{\vee}/\pi(\Lambda_{B_{i_1}})_k^{\vee}.
\end{equation*}
Since $\mathcal F_{k_i+1} \subset \mathcal F_{k_i + \Delta h_{i_1-i}-t_{r_0+1-i}}$ and since we have $\mathcal F_{k_i+1} = A_{i_1-i+1}/\pi(\Lambda_{B_{i_1}})_k^{\vee}$, we deduce that
\begin{equation*}
A_{i_1-i+1} \subset \pi\Lambda_{B_{i_1-i}}^{\vee}.
\end{equation*}
It finally follows that $\Lambda_{A_{i_1-i+1}} \subset \pi\Lambda_{B_{i_1-i}}^{\vee}$ which contradicts the fact that $i_1-i \not \in I$ in regards to the definition of the Bruhat-Tits type.\\
Assume now that $h_1 \not = 0$, and that $2t_{i_1} < h_1$ towards a contradiction. We have $k_{i_1} = l + \frac{h_{i_1}-h_1}{2}$. In this case we have 
\begin{center}
\begin{tikzcd}[cramped,sep = small]
\mathcal F_{k_{i_1}+h_1}^{\perp} \arrow[r,symbol=\subset] \arrow[d,equal,"/" marking] \arrow[dr,hook, shorten >= 10pt] & \mathcal F_{k_{i_1}+h_1-1}^{\perp} \arrow[r,symbol=\subset] \arrow[d,equal,"/" marking] \arrow[dr,hook, shorten >= 10pt] & \ldots \arrow[r,symbol=\subset] & \mathcal F_{k_{i_1}+h_1-t_{i_1}+1}^{\perp} \arrow[r,symbol=\subset] \arrow[d,equal,"/" marking] \arrow[dr,hook, shorten >= 10pt] & \mathcal F_{k_{i_1}+h_1-t_{i_1}}^{\perp} \arrow[r,symbol=\subset] \arrow[d,equal] & \ldots \arrow[r,symbol=\subset] & \mathcal F_{k_{i_1}+1}^{\perp} \arrow[d,equal]\\
\mathcal F_{k_{i_1}+1} \arrow[r,symbol=\subset] & \mathcal F_{k_{i_1}+2} \arrow[r,symbol=\subset] & \ldots \arrow[r,symbol=\subset] & \mathcal F_{k_{i_1}+t_{i_1}} \arrow[r,symbol=\subset] & \mathcal F_{k_{i_1}+t_{i_1}+1} \arrow[r,symbol=\subset] & \ldots \arrow[r,symbol=\subset] & \mathcal F_{k_{i_1}+h_1}
\end{tikzcd}
\end{center}
By hypothesis, we have $k_{i_1}+h_1-t_{i_1} \geq k_{i_1}+t_{i_1}+1$. It follows that 
\begin{equation*}
\mathcal F_{k_{i_1}+t_{i_1}+1} = \mathcal F_{k_{i_1}+1} + \tau(\mathcal F_{k_{i_1}+1}) + \ldots + \tau^{t_{i_1}}(\mathcal F_{k_{i_1}+1}),
\end{equation*}
and that $\mathcal F_{k_{i_1}+t_{i_1}+1}$ is $\tau$-stable. It follows that $\mathcal F_{k_{i_1}+t_{i_1}+1} = \Lambda_{A_1}/\pi(\Lambda_{B_{i_1}})_k^{\vee}$. Besides, by the diagram we have $\mathcal F_{k_{i_1}+t_{i_1}+1} \subset \mathcal F_{k_{i_1}+h_1-t_{i_1}} = \mathcal F_{k_{i_1}+t_{i_1}+1}^{\perp}$, which translates into
\begin{equation*}
\Lambda_{A_1} \subset \pi \Lambda_{A_1}^{\vee}.
\end{equation*}
Therefore $\Lambda_{A_1} \in \mathcal L_1$, which contradicts the fact that $i_1 \not = 0$ in regards to the definition of the Bruhat-Tits type.\\
All in all, we proved that $\mathcal G \in X_{\mathbf{\Lambda}_0}^{\mathbbm h_0,0}(k)$.\\

\underline{If $i_s \not = m$:} this case is identical to the previous case, and one shows that $(\pi B_{i_s+1} \subset \ldots \subset A_{i_s+1})$ defines a point of $X_{\mathbf{\Lambda}_s}^{\mathbbm h_s,0}(k)$. We omit the details.\\

\underline{For $1 \leq j \leq s-1:$} let $\mathcal G = (\mathcal G^1,\mathcal G^2) \in X_{\mathbf{\Lambda}_j}^{\mathbbm h_j}(k)$ be the partial flag in $(V_1)_k = \pi(\Lambda_{B_{i_{j+1}}})_k / \pi^2 (\Lambda_{A_{i_j+1}})^{\vee}_k$ and in $(V_2)_k = (\Lambda_{A_{i_j+1}})_k/\pi(\Lambda_{B_{i_{j+1}}})^{\vee}_k$ corresponding to $(\pi B_{i_j+1} \subset \ldots \subset \pi B_{i_{j+1}})$ and to $(A_{i_{j+1}} \subset \ldots \subset A_{i_j+1})$. Let $(w,w') = (w_1\ldots w_r,w'_1\ldots w'_{r'}) \in {}^J\mathbf W$ be such that $\mathcal G \in X_{J}\{(w,w')\}(k)$. Let $d = (l_0+l_1-1) + \frac{h_{i_{j+1}}-h_{i_j+1}}{2}$ denote the common dimension of $V_1$ and $V_2$, where $l_0 = d_{i_{j+1}}$ and $l_1 = c_{i_j+1}$. Let $\mathcal F = (\mathcal F^1,\mathcal F^2)$ be a full flag which lifts $\mathcal G$, so that $\mathcal F \in X((w,w'))(k)$. Let $r_0$ and $r_0'$ be as in Definition \ref{BTStratumDLVarGLj}.\\
First, if $l_0 > 1$ (resp. $l_1 > 1$), we prove that $t_r = n-1-k_r$ (resp. $t'_{r'} = n-1-k'_{r'}$) exactly as in the case $i_1 \not = 0$, so that we do not repeat the same arguments.\\
Let us go back to the general case $l_0,l_1 \geq 1$, and let $1 \leq i \leq i_{j+1}-i_j-1$. We have $k_i = l_1 + \frac{h_{i_j+i}-h_{i_j+1}}{2}$ and $k'_{r_0'+1-i} = l_0 + \frac{h_{i_{j+1}}-h_{i_j+i+1}}{2}$. Towards a contradiction, let us assume that $t_i+t'_{r_0'+1-i} < \Delta h_{i_j+i}$. In this case we have
\begin{center}\hspace*{-2.5cm}
\begin{tikzcd}[cramped,sep = small]
(\mathcal F_{k'_{r_0'+1-i}+\Delta h_{i_j+i}}^2)^{\perp} \arrow[r,symbol=\subset] \arrow[d,equal,"/" marking] \arrow[dr,hook, shorten >= 10pt] & (\mathcal F_{k'_{r_0'+1-i}+\Delta h_{i_j+i}-1}^2)^{\perp} \arrow[r,symbol=\subset] \arrow[d,equal,"/" marking] \arrow[dr,hook, shorten >= 10pt] & \ldots \arrow[r,symbol=\subset] & (\mathcal F_{k'_{r_0'+1-i}+\Delta h_{i_j+i}-t_i+1}^2)^{\perp} \arrow[r,symbol=\subset] \arrow[d,equal,"/" marking] \arrow[dr,hook, shorten >= 10pt] & (\mathcal F_{k'_{r_0'+1-i}+\Delta h_{i_j+i}-t_i}^2)^{\perp} \arrow[d,equal] \\
\mathcal F_{k_i+1}^1 \arrow[r,symbol=\subset] & \mathcal F_{k_{i}+2}^1 \arrow[r,symbol=\subset] & \ldots \arrow[r,symbol=\subset] & \mathcal F_{k_{i}+t_{i}}^1 \arrow[r,symbol=\subset] & \mathcal F_{k_{i}+t_{i}+1}^1 
\end{tikzcd}
\end{center}
\begin{flushright}\hspace*{12.5cm}
\begin{tikzcd}[cramped,sep = small]
{} \arrow[r,symbol=\subset] & \ldots \arrow[r,symbol=\subset] & (\mathcal F_{k'_{r_0'+1-i}+1}^2)^{\perp} \arrow[d,equal]\\
{} \arrow[r,symbol=\subset] & \ldots \arrow[r,symbol=\subset] & \mathcal F_{k_{i}+\Delta h_{i_j+i}}^1
\end{tikzcd}
\end{flushright}
\hspace*{\fill} and \hspace*{\fill}
\begin{center}\hspace*{-3cm}
\begin{tikzcd}[cramped,sep = small]
(\mathcal F_{k_{i}+\Delta h_{i_j+i}}^1)^{\perp} \arrow[r,symbol=\subset] \arrow[d,equal,"/" marking] \arrow[dr,hook, shorten >= 10pt] & (\mathcal F_{k_{i}+\Delta h_{i_j+i}-1}^1)^{\perp} \arrow[r,symbol=\subset] \arrow[d,equal,"/" marking] \arrow[dr,hook, shorten >= 10pt] & \ldots \arrow[r,symbol=\subset] & (\mathcal F_{k_{i}+\Delta h_{i_j+i}-t'_{r_0'+1-i}+1}^1)^{\perp} \arrow[r,symbol=\subset] \arrow[d,equal,"/" marking] \arrow[dr,hook, shorten >= 10pt] & (\mathcal F_{k_{i}+\Delta h_{i_j+i}-t'_{r_0'+1-i}}^1)^{\perp} \arrow[r,symbol=\subset] \arrow[d,equal] & \ldots \arrow[r,symbol=\subset] & (\mathcal F_{k_i+1}^1)^{\perp} \arrow[d,equal]\\
\mathcal F_{k'_{r_0'+1-i}+1}^2 \arrow[r,symbol=\subset] & \mathcal F_{k'_{r_0'+1-i}+2}^2 \arrow[r,symbol=\subset] & \ldots \arrow[r,symbol=\subset] & \mathcal F_{k'_{r_0'+1-i}+t'_{r_0'+1-i}}^2 \arrow[r,symbol=\subset] & \mathcal F_{k'_{r_0'+1-i}+t'_{r_0'+1-i}+1}^2 \arrow[r,symbol=\subset] & \ldots \arrow[r,symbol=\subset] & \mathcal F_{k'_{r_0'+1-i}+\Delta h_{i_j+i}}^2
\end{tikzcd}
\end{center}

By hypothesis, we have $k_{i}+\Delta h_{i_j+i}-t'_{r_0'+1-i} \geq k_i + t_i +1$. It follows that 
\begin{equation*}
\mathcal F_{k'_{r_0'+1-i}+t'_{r_0'+1-i}+1}^2 = \mathcal F_{k'_{r_0'+1-i}+1}^2 + \tau(\mathcal F_{k'_{r_0'+1-i}+1}^2) + \ldots + \tau^{t'_{r_0'+1-i}}(\mathcal F_{k'_{r_0'+1-i}+1}^2),
\end{equation*}
and that $\mathcal F_{k'_{r_0'+1-i}+t'_{r_0'+1-i}+1}^2$ is $\tau$-stable. Since $\mathcal F_{k'_{r_0'+1-i}+1}^2 = A_{i_j+i+1}/\pi(\Lambda_{B_{i_{j+1}}})_k^{\vee}$, we deduce that $\mathcal F_{k'_{r_0'+1-i}+t'_{r_0'+1-i}+1}^2 = (\Lambda_{A_{i_j+i+1}})_k/\pi(\Lambda_{B_{i_{j+1}}})_k^{\vee}$. Then, since $k'_{r_0'+1-i}+t'_{r_0'+1-i}+1 \leq k'_{r_0'+1-i}+\Delta h_{i_j+i}-t_i$, by the first diagram we have 
\begin{equation*}
\mathcal F_{k_{i}+\Delta h_{i_j+i}-t'_{r_0'+1-i}}^1 = (\mathcal F_{k'_{r_0'+1-i}+t'_{r_0'+1-i}+1}^2)^{\perp} = \pi^2(\Lambda_{A_{i_j+i+1}})_k^{\vee}/\pi^2(\Lambda_{A_{i_j}})_k^{\vee}. 
\end{equation*}
Since $\mathcal F_{k_i+1}^1 = \pi B_{i_j+i}/\pi^2(\Lambda_{A_{i_j}})_k^{\vee} \subset \mathcal F_{k_{i}+\Delta h_{i_j+i}-t'_{r_0'+1-i}}^1$, we deduce that 
\begin{equation*}
B_{i_j+i} \subset \pi(\Lambda_{A_{i_j+i+1}})_k^{\vee},
\end{equation*}
and thus $\Lambda_{B_{i_j+i}} \subset \pi\Lambda_{A_{i_j+i+1}}^{\vee}$. This is a contradiction with the fact that $i_j+i \not \in I$, in regards to the definition of the Bruhat-Tits type.\\
To sum up, we have showed that the image of any point $(A_m\subset \ldots \subset B_m) \in \mathcal N_{I,\mathbf{\Lambda}}^{\mathbbm h,0}(k)$ in $X_{I,\mathbf{\Lambda}}^{\mathbbm h}(k)$ lies in $X_{I,\mathbf{\Lambda}}^{\mathbbm h,0}(k)$. We shall now prove the converse. \\

Let $(A_m\subset \ldots \subset B_m) \in \mathcal N_{I,\mathbf{\Lambda}}^{\mathbbm h}(k)$ be a point such that its image in $X_{I,\mathbf{\Lambda}}^{\mathbbm h}(k)$ lies in $X_{I,\mathbf{\Lambda}}^{\mathbbm h,0}(k)$. Let $I'$ denote its Bruhat-Tits type. We shall prove that $I = I'$ and that for all $i\in I\setminus\{0\}$ and $j \in I\setminus\{m\}$, we have $\Lambda_0^{i} = \Lambda_{B_i}$ and $\Lambda_1^{j} = \Lambda_{A_{j+1}}$. \\

\underline{If $i_1 \not = 0$:} we recover the same notations as in the beginning of the proof, in particular we refer to the same diagrams. Write $t(\Lambda_0^{i_1}) = 2(l-1)+h_{i_1}+1$ for some $l\geq 1$. Any complete flag $\mathcal F$ which lifts $\mathcal G$ in $(V_{\Lambda_{0}^{i_1}}^0)_k$ has relative position $w_1\ldots w_{r_0}s_{l+h_{i_1}} \ldots s_{2(l-1)+h_{i_1}}$ with respect to $\mathcal F^{\perp}$. If $l=1$, we have $\Lambda_{B_{i_1}} = \Lambda_{0}^{i_1}$ by arguing on the dimension. If $l>1$ we have the following diagram.
\begin{center}\hspace*{-1.5cm}
\begin{tikzcd}[sep = small]
\mathcal F_{2(l-1)+h_{i_1}}^{\perp} \arrow[d,equal] \arrow[r,symbol=\subset] & \ldots \arrow[r,symbol=\subset] & \mathcal F_{l+h_{i_1}}^{\perp} \arrow[d,equal] \arrow[r,symbol=\subset] & \ldots \arrow[r,symbol=\subset] & \mathcal F_{l-1}^{\perp}\arrow[r,symbol=\subset] \arrow[d,equal,"/" marking] \arrow[dr,hook, shorten >= 10pt] & \ldots \arrow[r,symbol=\subset] & \mathcal F_{2}^{\perp} \arrow[r,symbol=\subset] \arrow[d,equal,"/" marking] \arrow[dr,hook, shorten >= 10pt] & \mathcal F_1^{\perp} \arrow[d,equal,"/" marking] \arrow[r,symbol=\subset] \arrow[dr,hook, shorten >= 10pt] & (V_{\Lambda_{0}^{i_1}}^0)_k \\
\mathcal F_1 \arrow[r,symbol=\subset] & \ldots \arrow[r,symbol=\subset] & \mathcal F_{l-1} \arrow[r,symbol=\subset] & \ldots \arrow[r,symbol=\subset] & \mathcal F_{l+h_{i_1}} \arrow[r,symbol=\subset] & \ldots \arrow[r,symbol=\subset] & \mathcal F_{2(l-1)+h_{i_1}-1} \arrow[r,symbol=\subset] & \mathcal F_{2(l-1)+h_{i_1}} \arrow[r,symbol=\subset] & (V_{\Lambda_{0}^{i_1}}^0)_k
\end{tikzcd}
\end{center}
By the diagram, it is clear that 
\begin{equation*}
(V_{\Lambda_{0}^{i_1}}^0)_k = \mathcal F_{l+h_{i_1}} + \tau(\mathcal F_{l+h_{i_1}}) + \ldots + \tau^{l-1}(\mathcal F_{l+h_{i_1}}).
\end{equation*}
Since $\mathcal F_{l+h_{i_1}} = B_{i_1}/\pi(\Lambda_{0}^{i_1})^{\vee}_k$, it follows that $T_{l}(B_{i_1}) = (\Lambda_{B_{i_1}})_k$ is $\tau$-stable, and that $\Lambda_{B_{i_1}} = \Lambda_{0}^{i_1}$. \\
Let us go back to the general case, so that $l \geq 1$. Let $1 \leq i \leq i_{1}-1$ and assume, towards a contradiction, that we have $\Lambda_{A_{i_1-i+1}} \subset \pi \Lambda_{B_{i_1-i}}^{\vee}$. Since $\Delta h_{i_1-i}-t_i+1 \leq t_{r_0+1-i}$, we know that 
\begin{equation*}
\mathcal F_{k_{r_0+1-i}+\Delta h_{i_1-i}-t_i+1} = \mathcal F_{k_{r_0+1-i}+1} + \tau(\mathcal F_{k_{r_0+1-i}+1}) + \ldots + \tau^{\Delta h_{i_1-i}-t_i}(\mathcal F_{k_{r_0+1-i}+1}).
\end{equation*}
Besides, from the diagram we have 
\begin{equation*}
\mathcal F_{k_i+t_i+1} = \mathcal F_{k_{r_0+1-i}+\Delta h_{i_1-i}-t_i+1}^{\perp} + \mathcal F_{k_i+1}. 
\end{equation*}
Now, $\mathcal F_{k_i+1} = A_{i_1-i+1}/\pi(\Lambda_0^{i_1})^{\vee}_k$, and since $\mathcal F_{k_{r_0+1-i}+1} = B_{i_1-i}/\pi(\Lambda_0^{i_1})^{\vee}_k$, we have $\mathcal F_{k_{r_0+1-i}+\Delta h_{i_1-i}-t_i+1}^{\perp} = \pi T_{\Delta h_{i_1-i}-t_i+1}(B_{i_1-i})^{\vee}/\pi(\Lambda_0^{i_1})^{\vee}_k$. Since we have 
\begin{equation*}
A_{i_1-i+1} \subset (\Lambda_{A_{i_1-i+1}})_k \subset \pi(\Lambda_{B_{i_1-i}})_k^{\vee} \subset \pi T_{\Delta h_{i_1-i}-t_i+1}(B_{i_1-i})^{\vee},
\end{equation*}
it follows that $\mathcal F_{k_i+1} \subset \mathcal F_{k_{r_0+1-i}+\Delta h_{i_1-i}-t_i+1}^{\perp}$, which is a contradiction. Therefore, we have proved that $i_1-i \not \in I'$ for all $1 \leq i \leq i_{1}-1$. \\
Eventually, assume that $h_1 \not = 0$. Towards a contradiction, assume that $\Lambda_{A_1} \in \mathcal L_1$. Since $h_1 - t_{i_1} \leq t_{i_1}$, we have 
\begin{equation*}
\mathcal F_{k_{i_1}+h_1-t_{i_1}+1} = \mathcal F_{k_{i_1}+1} + \tau(\mathcal F_{k_{i_1}+1}) + \ldots + \tau^{h_1-t_{i_1}}(\mathcal F_{k_{i_1}+1}).
\end{equation*}
Besides, from the diagram we have 
\begin{equation*}
\mathcal F_{k_{i_1}+t_{i_1}+1} = \mathcal F_{k_{i_1}+h_1-t_{i_1}+1}^{\perp} + \mathcal F_{k_{i_1}+1}.
\end{equation*}
But $\mathcal F_{k_{i_1}+1} = A_1/\pi(\Lambda_0^{i_1})^{\vee}_k$ and $\mathcal F_{k_{i_1}+h_1-t_{i_1}+1}^{\perp} = \pi T_{h_1-t_{i_1}+1}(A_1)^{\vee}/\pi(\Lambda_0^{i_1})^{\vee}_k$. Since we have 
\begin{equation*}
A_1 \subset (\Lambda_{A_1})_k \subset \pi (\Lambda_{A_1})_k^{\vee} \subset \pi T_{h_1-t_{i_1}+1}(A_1)^{\vee},
\end{equation*}
it follows that $\mathcal F_{k_{i_1}+1} \subset \mathcal F_{k_{i_1}+h_1-t_{i_1}+1}^{\perp}$ which is absurd. Therefore, we have proved that $0 \not \in I'$.\\

\underline{If $i_s \not = m$:} this case is identical to the previous case. One shows that $\Lambda_{A_{i_s+1}} = \Lambda_1^{i_s}$, that $i_s+i \not \in I'$ for all $1 \leq i \leq m-i_s-1$, and that $m \not\in I'$ when $h_m \not = n$. We omit the details.\\

\underline{For $1 \leq j \leq s-1:$} we recover the notations and diagrams as above. Write $t(\Lambda_0^{i_{j+1}}) = 2(l_0-1) + h_{i_{j+1}} +1$ and $t(\Lambda_1^{i_j}) = 2(l_1-1)+(n-h_{i_j+1})+1$ for some $l_0,l_1 \geq 1$, and let $V_1 = \pi\Lambda_0^{i_{j+1}}/\pi^2(\Lambda_1^{i_j})^{\vee}$ and $V_2 = \Lambda_1^{i_j}/\pi(\Lambda_0^{i_{j+1}})^{\vee}$. Let $d = \dim(V_1) = \dim(V_2)$. Any complete flag $\mathcal F = (\mathcal F^1,\mathcal F^2)$ which lifts $\mathcal G = (\mathcal G^1,\mathcal G^2)$ has relative position $(w_1\ldots w_{r_0}s_{d-l_0+1}\ldots s_{d-1},w'_1\ldots w'_{r_0'}s_{d-l_1+1}\ldots s_{d-1})$ with respect to $\mathcal F^{\perp}$. Just as in the case $i_1 \not = 0$, one proves that $\Lambda_1^{i_j} = \Lambda_{A_{i_j+1}}$ and that $\Lambda_0^{i_{j+1}} = \Lambda_{B_{i_{j+1}}}$.\\
Now, let $1 \leq i \leq i_{j+1}-i_j -1$ and assume, towards a contradiction, that $\Lambda_{B_{i_j+i}} \subset \pi \Lambda_{A_{i_j+i+1}}^{\vee}$. Since $\Delta h_{i_j+i} - t_i \leq t'_{r_0'+1-i}$, we know that 
\begin{equation*}
\mathcal F_{k'_{r_0'+1-i}+\Delta h_{i_j+1}-t_{i}+1}^2 = \mathcal F_{k'_{r_0'+1-i}+1}^2 + \tau(\mathcal F_{k'_{r_0'+1-i}+1}^2)  + \ldots + \tau^{\Delta h_{i_j+1}-t_{i}}(\mathcal F_{k'_{r_0'+1-i}+1}^2).
\end{equation*}
Besides, from the diagram we have
\begin{equation*}
\mathcal F_{k_i+t_i+1}^1 = (\mathcal F_{k'_{r_0'+1-i}+\Delta h_{i_j+1}-t_{i}+1}^2)^{\perp} + \mathcal F_{k_i+1}^1.
\end{equation*}
Now, $\mathcal F_{k_i+1}^1 = \pi B_{i_{j}+i}/\pi^2(\Lambda_1^{i_j})_k^{\vee}$, and since $\mathcal F_{k'_{r_0'+1-i}+1}^2 = A_{i_j+i+1}/\pi(\Lambda_0^{i_{j+1}})_k^{\vee}$, we have $(\mathcal F_{k'_{r_0'+1-i}+\Delta h_{i_j+1}-t_{i}+1}^2)^{\perp} = \pi^2 T_{\Delta h_{i_j+1}-t_{i}+1}(A_{i_j+i+1})^{\vee}/\pi^2(\Lambda_1^{i_j})_k^{\vee}$. Since we have 
\begin{equation*}
\pi B_{i_{j}+i} \subset \pi (\Lambda_{B_{i_{j}+i}})_k \subset \pi^2 (\Lambda_{A_{i_j+i+1}})^{\vee}_k \subset \pi^2 T_{\Delta h_{i_j+1}-t_{i}+1}(A_{i_j+i+1})^{\vee},
\end{equation*}
it follows that $\mathcal F_{k_i+1}^1 \subset (\mathcal F_{k'_{r_0'+1-i}+\Delta h_{i_j+1}-t_{i}+1}^2)^{\perp}$ which is a contradiction. Therefore, we have proved that $i_j+i \not \in I'$ for all $1 \leq i \leq i_{j+1}-i_j -1$.\\

Putting things together, we have proved that for all $i\in I\setminus\{0\}$ and $j \in I\setminus\{m\}$, we have $\Lambda_0^{i} = \Lambda_{B_i}$ and $\Lambda_1^{j} = \Lambda_{A_{j+1}}$. Moreover, we have showed that the complement of $I$ in $\{0,\ldots ,m\}$ is included in the complement of $I'$. In other words, we proved that $I' \subset I$. The reverse inclusion is now obvious from the definition of $I$. 	Therefore, the point $(A_m \subset \ldots \subset B_m)$ belongs to $\mathcal N_{I,\mathbf{\Lambda}}^{\mathbbm h,0}(k)$ and this concludes the proof.
\end{proof}

\begin{corol}\label{BTStrataNotEmpty}
Let $k$ be an algebraically closed field containing $\kappa_{\breve E}$. For every Bruhat-Tits type $(I,\mathbf{\Lambda})$, there exists a point $(A_m \subset \ldots \subset B_m) \in \mathcal N_{I,\mathbf{\Lambda}}^{\mathbbm h,0}(k)$. In particular, $\mathcal N_{I,\mathbf{\Lambda}}^{\mathbbm h,0} \not = \emptyset$.
\end{corol}

\begin{defi}
The locally closed subvarieties $\mathcal N_{I,\mathbf{\Lambda}}^{\mathbbm h,0}$, where $(I,\mathbf{\Lambda})$ runs over all the Bruhat-Tits indices, forms the \textit{Bruhat-Tits stratification} of $\mathcal N_{E/F}^{\mathbbm h} \times \kappa_{\breve E}$. The strata $\mathcal N_{I,\mathbf{\Lambda}}^{\mathbbm h,0}$ are called the \textit{Bruhat-Tits strata}, and their closures $\mathcal N_{I,\mathbf{\Lambda}}^{\mathbbm h}$ are called the \textit{closed Bruhat-Tits strata}.
\end{defi}

By construction, if $k$ is an algebraically closed field containing $\kappa_{\breve E}$, then we have 
\begin{equation*}
\mathcal N_{E/F}^{\mathbbm h}(k) = \bigsqcup_{I,\mathbf{\Lambda}} \mathcal N_{I,\mathbf{\Lambda}}^{\mathbbm h,0}(k),
\end{equation*}
where $(I,\mathbf{\Lambda})$ runs over all the Bruhat-Tits indices. This justifies the terminology for ``stratification''.

\begin{rk}\label{ComparisonWithCho}
Our definition of Bruhat-Tits stratum disagrees with the definition given in \cite{cho} Definition 3.19. We consider the maximal parahoric case, so that $m=1$, and we assume that $0 < h < n$. For instance, let $I = \{1\}$ and let $\Lambda_0 \in \mathcal L_0^{\geq h+1}$. Adapting to our notations, the definition of loc. cit. is 
\begin{equation*}
\mathcal N_{\{1\},\{\Lambda_0\}}^{h,0,\mathrm{bis}} := \mathcal N_{\{1\},\{\Lambda_0\}}^{h} \setminus \bigcup_{\substack{\Lambda \in \mathcal L_0^{\geq h+1}\\ \Lambda \subset \Lambda_0}} \mathcal N_{\{1\},\{\Lambda\}}^{h}.
\end{equation*}
In other words, compared to $\mathcal N_{\{1\},\{\Lambda_0\}}^{h,0}$, we do not remove the closed strata corresponding to Bruhat-Tits indices of the form $\{0,1\}$. In fact, we have 
\begin{equation*}
\mathcal N_{\{1\},\{\Lambda_0\}}^{h,0,\mathrm{bis}} = \mathcal N_{\{1\},\{\Lambda_0\}}^{h,0} \sqcup \bigsqcup_{\substack{\Lambda_1 \in \mathcal L_1^{\geq n-h+1} \\ \pi\Lambda_1^{\vee} \subset \Lambda_0}} \mathcal N_{\{0,1\},\{\Lambda_1,\Lambda_0\}}^{h,0}.
\end{equation*}
\end{rk}

\begin{rk}
From the construction of $X_{I,\mathbf{\Lambda}}^{0}$, it turns out that the Bruhat-Tits strata are isomorphic to a disjoint union of several fine Deligne-Lusztig varieties in general, and not just a single one. In fact, each stratum is isomorphic to a single fine Deligne-Lusztig varieties if and only if we are in one of the following three cases 
\begin{enumerate}[noitemsep,nolistsep]
\item $m=1$ and $h=0$,
\item $m=1$ and $h=n$,
\item $n$ is even, $m=2$ and $\mathbbm h = (0,n)$.
\end{enumerate}
Case 1 and case 2 are isomorphic to each other, and correspond to the hyperspecial case studied in \cite{vw2}. Incidentally, these three cases are exactly those for which the associated affine Deligne-Lusztig variety is ``of Coxeter type'', following the terminology of \cite{cox24}.
\end{rk}

\subsection{Action of $\mathrm{Aut}(\mathbb X^{[h_i]})$ and irreducible components}

For $1 \leq j \leq m$, let us consider the group $J^{[h_j]} := \mathrm{Aut}(\mathbb X^{[h_j]})$ of automorphisms of the strict formal $\mathcal O_F$-module $\mathbb X^{[h_j]}$ which are compatible with the additional structures. Then $J^{[h_j]}$ can be regarded as a connected reductive group over $F$, which is identified with $\mathrm{GU}^{0}(C,\{\cdot,\cdot\}_{[h_j]})$, that is the group of height $0$ unitary similitudes of $(C,\{\cdot,\cdot\}_{[h_j]})$. Explicitely, for any $F$-algebra $R$, we have 
\begin{equation*}\hspace{-1cm}
J^{[h_j]}(R) = \left\{g \in \mathrm{GL}_{F_{\mathbb F_{q^2}}\otimes_F R}(C\otimes_{F} R) \,\middle|\, \exists c(g) \in \mathcal O_F^{\times}, \forall v,w \in C\otimes_{F} R, \{gv,gw\}_{[h_j]} = c(g)\{v,w\}_{[h_j]}\right\}.
\end{equation*}
The group $J^{[h_j]}(F)$ acts on $\mathcal N_{E/F}^{\mathbbm h}$ as follows. For $g \in J^{[h_j]}(F)$ and $(X^{[i]},i_{X^{[i]}},\lambda_{X^{[i]}},\rho_{X^{[i]}})_{1\leq i \leq m} \in \mathcal N_{E/F}^{\mathbbm h}(S)$ where $S \in \mathbf{Nilp}$,
\begin{equation*}
g\cdot (X^{[i]},i_{X^{[i]}},\lambda_{X^{[i]}},\rho_{X^{[i]}})_{1\leq i \leq m} = (X^{[i]},i_{X^{[i]}},\lambda_{X^{[i]}},\rho'_{X^{[i]}})_{1\leq i \leq m},
\end{equation*}
where $\forall i < j$ we have $\rho'_{X^{[i]}} := (\alpha_{h_j,h_i})_{\overline S}\circ g \circ (\alpha_{h_j,h_i})^{-1}_{\overline S}\circ\rho_{X^{[i]}}$, $\forall i > j$ we have $\rho'_{X^{[i]}} := (\alpha_{h_i,h_j})_{\overline S}^{-1}\circ g \circ (\alpha_{h_i,h_j})_{\overline S}\circ\rho_{X^{[i]}}$, and of course $\rho'_{X^{[j]}} = g\circ \rho_{X^{[j]}}$. The actions of $J^{[h_j]}(F)$ for varying $1 \leq i \leq m$ agree with eachother, as the groups are mutually isomorphic via the isogenies $\alpha_{h_j,h_i}$ for $i < j$. From now on, in accordance with our convention, we will only consider the action of $J(F) := J^{[h_1]}(F)$.\\
Given an algebraically closed field $k$ containing $\kappa_{\breve E}$ and $g \in J(F)$, the action of $J(F)$ on $\mathcal N_{E/F}^{\mathbbm h}(k)$ is given by 
\begin{equation*}
g\cdot (A_m \subset \ldots \subset B_m) = (g(A_m) \subset \ldots \subset g(B_m)).
\end{equation*}
Here, by abuse of notations $g$ denotes the automorphism $g\otimes \mathrm{id}$ of $C \otimes_F F_k$. Since $g$ preserves duals, inclusions and indices of lattices, it is clear that $(g(A_m) \subset \ldots \subset g(B_m))$ defines a point in $\mathcal N_{E/F}^{\mathbbm h}(k)$. Moreover, $J(F)$ acts on the set of Bruhat-Tits indices via 
\begin{equation*}
g\cdot (I,\mathbf{\Lambda}) = (I,g(\mathbf{\Lambda})),
\end{equation*}
where $g(\mathbf{\Lambda})$ is the collection of lattices consisting of $g(\Lambda_0^i)$ and $g(\Lambda_1^j)$ for all $i \in I\setminus\{0\}$ and $j\in I \setminus\{m\}$. It is clear that $g$ preserves vertex lattices, and that $(I,g(\mathbf{\Lambda}))$ is again a Bruhat-Tits index. This action is compatible with the Bruhat-Tits strata.

\begin{prop}
Let $(I,\mathbf{\Lambda})$ be a Bruhat-Tits index, and let $g \in J(F)$. Then $g$ induces an isomorphism
\begin{equation*}
g:\mathcal N_{I,\mathbf{\Lambda}}^{\mathbbm h} \xrightarrow{\sim} \mathcal N_{I,g(\mathbf{\Lambda})}^{\mathbbm h}.
\end{equation*}
Moreover, it induces an isomorphism $\mathcal N_{I,\mathbf{\Lambda}}^{\mathbbm h,0} \xrightarrow{\sim} \mathcal N_{I,g(\mathbf{\Lambda})}^{\mathbbm h,0}$ between the corresponding Bruhat-Tits strata as well. 
\end{prop}

\begin{proof}
It is clear that the morphisms are well-defined, and that $g^{-1}$ is the desired inverse.
\end{proof}

\begin{rk}
Given a Bruhat-Tits index $(I,\mathbf{\Lambda})$, the stabilizer $J_{\mathbf{\Lambda}} := \mathrm{Stab}_{J(E)}(\mathbf{\Lambda})$ defines a parahoric subgroup of $J(E)$. The induced action of $J_{\mathbf{\Lambda}}$ on $\mathcal N_{I,\mathbf{\Lambda}}^{\mathbbm h}$ factors through the maximal reductive quotient $\mathcal J_{\mathbf{\Lambda}}$ of $J_{\mathbf{\Lambda}}$. Up to the similitude factor, the finite group $\mathcal J_{\mathbf{\Lambda}}$ can be decomposed as a product of finite unitary and general linear groups. It turns out that the isomorphism $\mathcal N_{I,\mathbf{\Lambda}}^{\mathbbm h} \xrightarrow{\sim} X_{I,\mathbf{\Lambda}}^{\mathbbm h} \times \kappa_{\breve E}$ is $\mathcal J_{\mathbf{\Lambda}}$-equivariant, where the action on the right-hand side is inherited from the construction of $X_{I,\mathbf{\Lambda}}^{\mathbbm h}$ as the closure of a fine Deligne-Lusztig variety.
\end{rk}

The orbits of $J(F)$ on the set of Bruhat-Tits indices are described as follows. 

\begin{prop}\label{OrbitJ(E)Action}
Let $(I,\mathbf{\Lambda})$ and $(I',\mathbf{\Lambda'})$ be two Bruhat-Tits indices. They are in the same $J(F)$-orbit if and only if $I=I'$, and for all $i\in I\setminus\{0\}$ and $j\in I\setminus\{m\}$, we have $t(\Lambda_0^{i}) = t(\Lambda_0^{\prime i})$ and $t(\Lambda_1^{j}) = t(\Lambda_1^{\prime j})$.
\end{prop}

\begin{proof}
Since the action of $J(F)$ preserves the type of vertex lattices, it is clear that any two Bruhat-Tits indices in the same $J(F)$ orbit satisfy the conditions of the Proposition. Moreover, it is also clear that $I=I'$ is a necessary condition. \\
Consider the subgroup $H := \mathrm{SU}(C,\{\cdot,\cdot\}_{[h_1]}) \subset J$. According to Theorem 3.5 of \cite{vw1}, the simplicial complex of the Bruhat-Tits building of $H(F)$ can be identified with $\mathcal L_0$, which is given a simplicial structure by decreeing that for $k\geq 0$, a $k$-simplex is a subset $S \subset \mathcal L_0$ such that $\#S = k+1$ and, for some ordering $\Lambda^0,\ldots,\Lambda^k$ of the elements of $S$, we have 
\begin{equation*}
\pi\Lambda^{k\vee} \subsetneq \Lambda^0 \subsetneq \ldots \subsetneq \Lambda^k.
\end{equation*}
We point out that the proof of loc. cit. was written is the case $F=\mathbb Q_p$, but it adapts to the general case without difficulty.\\
Now, let $(I,\mathbf{\Lambda})$ and $(I,\mathbf{\Lambda'})$ be two Bruhat-Tits indices satisfying the conditions of the Proposition. By construction, we can arrange the lattices in $\mathbf{\Lambda}$ and in $\mathbf{\Lambda'}$ so that they form two simplices in $\mathcal L_0$. For instance, if $0\leq i_1 < \ldots < i_s \leq m$ denote the elements of $I$, and if $i_1 \not = 0$ and $i_s \not = m$, then the chain 
\begin{equation*}
\Lambda_0^{i_1} \subset \pi\Lambda_1^{i_1\vee} \subset \Lambda_0^{i_2} \subset \ldots \subset \Lambda_0^{i_s} \subset \pi\Lambda_1^{i_s\vee}
\end{equation*}
forms a simplex in $\mathcal L_0$. Note that $t(\pi\Lambda_1^{i_j\vee}) \leq h_{i_j+1}-1$ and $t(\Lambda_0^{i_{j+1}}) \geq h_{i_{j+1}} + 1$, so that we have $\pi\Lambda_1^{i_j\vee} \subsetneq \Lambda_0^{i_{j+1}}$. However, one could have $\Lambda_0^{i_j} = \pi\Lambda_1^{i_j\vee}$ for some $1 \leq j \leq s$. For this reason, this simplex is only at most $(2s-1)$-dimensional. One can proceed similarly for $\mathbf{\Lambda'}$, and also in the case $i_1=0$ and/or $i_s=m$.\\
We complete the two simplices into maximal simplices. These correspond to two alcoves $c$ and $c'$ in the Bruhat-Tits building of $H(F)$. Since $H(F)$ acts transitively on the set of alcoves of its building, there exists some $g\in H(F)$ such that $g(c) = c'$. Since the lattices in $\mathbf{\Lambda}$ and in $\mathbf{\Lambda'}$ are mutually of the same types, we deduce that $g(\Lambda_0^i) = \Lambda_0^{\prime i}$ and $g(\Lambda_1^j) = \Lambda_1^{\prime j}$ for all $i\in I\setminus\{0\}$ and $j\in I\setminus \{m\}$. In particular, the Bruhat-Tits indices $(I,\mathbf{\Lambda})$ and $(I,\mathbf{\Lambda'})$ are in the same $H(F)$-orbit, hence a fortiori in the same $J(F)$-orbit.
\end{proof}

We are now going to investigate the irreducible components of $\mathcal N_{E/F,\mathrm{red}}^{\mathbbm h}$, and the number of orbits of such components under the action of $J(E)$. 

\begin{lem}\label{BruhatTitsStrataAsIntersection}
Let $(I,\mathbf{\Lambda})$ be a Bruhat-Tits type. We partition $\mathbf{\Lambda}$ as follows. For $i \in I$, define
\begin{equation*}
\mathbf{\Lambda}^i := \begin{cases}
\{\Lambda_0^i,\Lambda_1^i\} & \text{if } 0 < i < m,\\
\{\Lambda_1^0\} & \text{if } i=0,\\
\{\Lambda_0^m\} & \text{if } i=m.
\end{cases}
\end{equation*}
We have 
\begin{equation*}
\mathcal N_{I,\mathbf{\Lambda}}^{\mathbbm h} = \bigcap_{i \in I} \mathcal N_{\{i\},\mathbf{\Lambda}^{i}}^{\mathbbm h}.
\end{equation*}
\end{lem}

\begin{rk}
We point out that the partition of $\mathbf{\Lambda}$ into the subsets $\mathbf{\Lambda}^i$ differs from the partition into the subsets $\mathbf{\Lambda}_j$ which was introduced in \ref{Section3.3}. 
\end{rk}

\begin{proof}
By construction, it is clear that $(\{i\},\mathbf{\Lambda}^i)$ defines a Bruhat-Tits index for all $i \in I$, and that their successive intersections are well-defined. Moreover, we have 
\begin{equation*}
\bigcap_{i\in I} (\{i\},\mathbf{\Lambda}_i) = (I,\mathbf{\Lambda}).
\end{equation*} 
The result then follows from Proposition \ref{CombinatoricsBTStratification}.
\end{proof}

We define two integers as follows.

\begin{align*}
t_{\mathrm{min}} := \begin{cases}
0 & \text{if } h_1 \text{ is odd},\\
1 & \text{if } h_1 \text{ is even},
\end{cases}
 & & 
t_{\mathrm{max}} := \begin{cases}
n & \text{if } n-h_1 \text{ is odd},\\
n-1 & \text{if } n-h_1 \text{ is even}.
\end{cases}
\end{align*}
Since all the $h_i$'s have the same parity, the choice of $h_1$ in the definition is of no importance. 

\begin{corol}\label{IrreducibleComponents}
The irreducible components of $\mathcal N_{E/F,\mathrm{red}}^{\mathbbm h}$ consists of all the maximal closed Bruhat-Tits strata, which are exactly those listed below:
\begin{enumerate}[noitemsep,nolistsep]
\item \underline{if $h_1 \not = 0$:} $\mathcal N_{\{0\},\{\Lambda_1^0\}}^{\mathbbm h}$ where $\Lambda_1^0 \in \mathcal L_1$ has type $t(\Lambda_1^0) = n-t_{\mathrm{min}}$,
\item \underline{if $h_m \not = n$:} $\mathcal N_{\{m\},\{\Lambda_0^m\}}^{\mathbbm h}$ where $\Lambda_0^m \in \mathcal L_0$ has type $t(\Lambda_0^m) = t_{\mathrm{max}}$,
\item \underline{for $0 < i < m$:} $\mathcal N_{\{i\},\{\Lambda_0^i,\Lambda_1^i\}}^{\mathbbm h}$ where $\Lambda_0^i \in \mathcal L_0^{\geq h_i+1}$, $\Lambda_1^i \in \mathcal L_1^{\geq n-h_{i+1}+1}$ and $\Lambda_0^i = \pi\Lambda_1^{i\vee}$. 
\end{enumerate}
In case 1. we have 
\begin{equation*}
\dim\left(\mathcal N_{\{0\},\{\Lambda_1^0\}}^{\mathbbm h}\right) = n - \frac{h_1+t_{\mathrm{min}}+1}{2}.
\end{equation*}
In case 2. we have 
\begin{equation*}
\dim\left(\mathcal N_{\{m\},\{\Lambda_0^m\}}^{\mathbbm h}\right) = \frac{t_{\mathrm{max}}+h_m-1}{2}.
\end{equation*}
In case 3. we have 
\begin{equation*}
\dim\left(\mathcal N_{\{i\},\{\Lambda_0^i,\Lambda_1^i\}}^{\mathbbm h}\right) = n-1-\Delta h_{i}.
\end{equation*}
\end{corol}

\begin{rk}
In the maximal parahoric case, ie. when $m=1$, only the cases 1. and 2. can occur. Our description agrees with the one given in \cite{cho} Theorem 3.16.
\end{rk}

\begin{proof}
Since all the closed Bruhat-Tits strata are irreducible and they cover the whole of $\mathcal N_{E/F,\mathrm{red}}^{\mathbbm h}$, the irreducible components agree with the maximal closed Bruhat-Tits strata. By Lemma \ref{BruhatTitsStrataAsIntersection}, for a closed Bruhat-Tits stratum $\mathcal N_{I,\mathbf{\Lambda}}^{\mathbbm h}$ to be maximal, one necessary condition is that $\#I = 1$, which we know assume.\\
If $I = \{0\}$, the closed Bruhat-Tits stratum $\mathcal N_{\{0\},\{\Lambda_1^0\}}^{\mathbbm h}$ is maximal if and only $\Lambda_1^{0}$ is maximal in $\mathcal L_1^{\geq n-h_1+1}$. This amount to the condition that $t(\Lambda_1^{0}) = n - t_{\mathrm{min}}$. Eventually, $\mathcal N_{\{0\},\{\Lambda_1^0\}}^{\mathbbm h}$ is isomorphic to $X_{\{\Lambda_1^0\}}^{\mathbbm h}$, from which we deduce the dimension. \\
If $I = \{m\}$, the closed Bruhat-Tits stratum $\mathcal N_{\{m\},\{\Lambda_0^m\}}^{\mathbbm h}$ is maximal if and only $\Lambda_0^{m}$ is maximal in $\mathcal L_0^{\geq h_m+1}$. This amount to the condition that $t(\Lambda_0^{m}) = t_{\mathrm{max}}$. Eventually, $\mathcal N_{\{m\},\{\Lambda_0^m\}}^{\mathbbm h}$ is isomorphic to $X_{\{\Lambda_0^m\}}^{\mathbbm h}$, from which we deduce the dimension.\\
If $I = \{i\}$ for some $0<i<m$, we have an inclusion $\mathcal N_{\{i\},\{\Lambda_0^i,\Lambda_1^i\}}^{\mathbbm h} \subset \mathcal N_{\{i\},\{\Lambda_0^{\prime i},\Lambda_1^{\prime i}\}}^{\mathbbm h}$ if and only if $\Lambda_0^i \subset \Lambda_0^{\prime i}$ and $\Lambda_1^i \subset \Lambda_1^{\prime i}$. In this case, we have 
\begin{equation*}
\Lambda_0^i \subset \Lambda_0^{\prime i} \subset \pi\Lambda_1^{\prime i \vee} \subset \pi \Lambda_1^{i \vee}.
\end{equation*}
Therefore, $\mathcal N_{\{i\},\{\Lambda_0^i,\Lambda_1^i\}}^{\mathbbm h}$ is maximal if and only if $\Lambda_0^i = \pi \Lambda_1^{i \vee}$. Eventually, $\mathcal N_{\{i\},\{\Lambda_0^i,\Lambda_1^i\}}^{\mathbbm h}$ is isomorphic to $X_{\{\Lambda_0^i\}}^{\mathbbm h_0} \times X_{\{\Lambda_1^i\}}^{\mathbbm h_1}$, from which we deduce the dimension. 
\end{proof}

\begin{theo}\label{NumberOrbitsIrreducibleComponents}
The number of $J(E)$-orbits of irreducible components in $\mathcal N_{E/F,\mathrm{red}}^{\mathbbm h}$ is $\frac{h_m-h_1}{2}+\epsilon$, where $\epsilon$ is given by
\begin{equation*}
\epsilon = \begin{cases}
0 & \text{if } h_1 = 0 \text{ and } h_m = n,\\
1 & \text{if } (h_1 = 0 \text{ and } h_m<n) \text{ or } (h_1>0 \text{ and } h_m = n),\\
2 & \text{if } h_1 > 0 \text{ and } h_m < n.
\end{cases}
\end{equation*}
More precisely:
\begin{enumerate}[noitemsep,nolistsep]
\item all the irreducible components of the form $\mathcal N_{\{0\},\{\Lambda_1^0\}}^{\mathbbm h}$, when $h_1 \not = 0$, form a single orbit,
\item all the irreducible components of the form $\mathcal N_{\{m\},\{\Lambda_0^m\}}^{\mathbbm h}$, when $h_m \not = n$, form a single orbit,
\item the irreducible components of the form $\mathcal N_{\{i\},\{\Lambda_0^i,\Lambda_1^i\}}^{\mathbbm h}$, for some $0 < i < m$, form $\Delta h_i$ orbits. 
\end{enumerate}
\end{theo}

\begin{proof}
By Proposition \ref{OrbitJ(E)Action}, when $h_1 \not = 0$ (resp. $h_m \not = n$) it is clear that all the irreducible components of the form $\mathcal N_{\{0\},\{\Lambda_1^0\}}^{\mathbbm h}$ (resp. $\mathcal N_{\{m\},\{\Lambda_0^m\}}^{\mathbbm h}$) form a single orbit, since the vertex lattices  are required to have the same type. \\
If $0 < i < m$, two irreducible components $\mathcal N_{\{i\},\{\Lambda_0^i,\Lambda_1^i\}}^{\mathbbm h}$ and $\mathcal N_{\{i\},\{\Lambda_0^{\prime i},\Lambda_1^{\prime i}\}}^{\mathbbm h}$ are in the same orbit if and only if $t(\Lambda_0^i) = t(\Lambda_0^{\prime i})$ and $t(\Lambda_1^i) = t(\Lambda_1^{\prime i})$. Since we have $\Lambda_0^i = \pi\Lambda_1^{i\vee}$ and $\Lambda_0^{\prime i} = \pi\Lambda_1^{\prime i \vee}$, the $J(F)$-orbit is entirely determined by the single value of $t(\Lambda_0^i) = n-t(\Lambda_1^i)$. This can take any value of fixed parity between $h_i+1$ and $h_{i+1}-1$, for a total of $\Delta h_i$ orbits.\\
Since $\sum_{i=1}^{m-1} \Delta h_i = \frac{h_m-h_1}{2}$, the result follows.
\end{proof}

For instance, in the maximal parahoric case, there is a single orbit if $h=0$ or $h=n$, and exactly two orbits otherwise.

\subsection{Examples} \label{Examples}
\subsubsection{Case $m=2$}

We spell out the results of the previous section in the case $m=2$. We have $\mathbbm h = (h_1,h_2)$ and a geometric point in $\mathcal N_{E/F}^{h_1,h_2}\times \kappa_{\breve E}$ is given by a chain of lattices $A_2 \subset A_1 \subset B_1 \subset B_2$. There are at most 7 different kinds of Bruhat-Tits indices $(I,\mathbf{\Lambda})$ as follows. Below the notations $\Lambda_1^0$, $\Lambda_0^1$,$\Lambda_1^1$ and $\Lambda_0^2$ will always denote vertex lattices in  $\mathcal L_1^{\geq n-h_1+1}$, $\mathcal L_0^{\geq h_1+1}$, $\mathcal L_1^{\geq n-h_2+1}$ and $\mathcal L_0^{\geq h_2+1}$ respectively.
\begin{enumerate}
\item $I = \{1\}$ and $\mathbf{\Lambda} = \{\Lambda_0^1,\Lambda_1^1\}$ where $\Lambda_0^1 \subset \pi\Lambda_1^{1\vee}$,
\item \underline{if $h_1 \not = 0$:} $I = \{0\}$ and $\mathbf{\Lambda} = \{\Lambda_1^0\}$,
\item \underline{if $h_1 \not = 0$:} $I = \{0,1\}$ and $\mathbf{\Lambda} = \{\Lambda_1^0,\Lambda_0^1,\Lambda_1^1\}$ where $\pi\Lambda_1^{0\vee} \subset \Lambda_0^1 \subset \pi\Lambda_1^{1\vee}$,
\item \underline{if $h_2 \not = n$:} $I = \{2\}$ and $\mathbf{\Lambda} = \{\Lambda_0^2\}$,
\item \underline{if $h_2 \not = n$:} $I = \{1,2\}$ and $\mathbf{\Lambda} = \{\Lambda_0^1,\Lambda_1^1,\Lambda_0^2\}$ where $\Lambda_0^1 \subset \pi\Lambda_1^{1\vee} \subset \Lambda_0^2$,
\item \underline{if $h_1 \not = 0$ and $h_2 \not = n$:} $I = \{0,2\}$ and $\mathbf{\Lambda} = \{\Lambda_1^0,\Lambda_0^2\}$ where $\pi\Lambda_1^{0\vee} \subset \Lambda_0^2$,
\item \underline{if $h_1 \not = 0$ and $h_2 \not = n$:} $I = \{0,1,2\}$ and $\mathbf{\Lambda} = \{\Lambda_1^0,\Lambda_0^1,\Lambda_1^1,\Lambda_0^2\}$ where $\pi\Lambda_1^{0\vee} \subset \Lambda_0^1 \subset \pi\Lambda_1^{1\vee} \subset \Lambda_0^2$.
\end{enumerate}
We give the dimension of the associated closed Bruhat-Tits strata $\mathcal N_{I,\mathbf{\Lambda}}^{h_1,h_2}$.
\begin{enumerate}
\item $\mathcal N_{\{1\},\{\Lambda_0^1,\Lambda_1^1\}}^{h_1,h_2} \simeq X_{\{\Lambda_0^1\}}^{h_1} \times X_{\{\Lambda_1^1\}}^{h_2}$ has dimension 
\begin{equation*}
\dim(\mathcal N_{\{1\},\{\Lambda_0^1,\Lambda_1^1\}}^{h_1,h_2}) = \frac{t(\Lambda_0^1)+t(\Lambda_1^1)+n}{2}-\Delta h_1 -1,
\end{equation*} 
\item \underline{if $h_1 \not = 0$:} $\mathcal N_{\{0\},\{\Lambda_1^0\}}^{h_1,h_2} \simeq X_{\{\Lambda_1^0\}}^{h_1,h_2}$ has dimension 
\begin{equation*}
\dim(\mathcal N_{\{0\},\{\Lambda_1^0\}}^{h_1,h_2}) = \frac{t(\Lambda_1^0)+(n-h_1)-1}{2},
\end{equation*} 
\item \underline{if $h_1 \not = 0$:} $\mathcal N_{\{0,1\},\{\Lambda_1^0,\Lambda_0^1,\Lambda_1^1\}}^{h_1,h_2} \simeq X_{\{\Lambda_1^0,\Lambda_0^1\}}^{h_1} \times X_{\{\Lambda_1^1\}}^{h_2}$ has dimension 
\begin{equation*}
\dim(\mathcal N_{\{0,1\},\{\Lambda_1^0,\Lambda_0^1,\Lambda_1^1\}}^{h_1,h_2}) = \frac{t(\Lambda_1^0)+t(\Lambda_0^1)+t(\Lambda_1^1)-h_2-1}{2}-1,
\end{equation*} 
\item \underline{if $h_2 \not = n$:} $\mathcal N_{\{2\},\{\Lambda_0^2\}}^{h_1,h_2} \simeq X_{\{\Lambda_0^2\}}^{h_1,h_2}$ has dimension 
\begin{equation*}
\dim(\mathcal N_{\{2\},\{\Lambda_0^2\}}^{h_1,h_2}) = \frac{t(\Lambda_0^2)+h_2-1}{2},
\end{equation*} 
\item \underline{if $h_2 \not = n$:} $\mathcal N_{\{1,2\},\{\Lambda_0^1,\Lambda_1^1,\Lambda_0^2\}}^{h_1,h_2} \simeq X_{\{\Lambda_0^1\}}^{h_1} \times X_{\{\Lambda_1^1,\Lambda_0^2\}}^{h_2}$ has dimension 
\begin{equation*}
\dim(\mathcal N_{\{1,2\},\{\Lambda_0^1,\Lambda_1^1,\Lambda_0^2\}}^{h_1,h_2}) = \frac{t(\Lambda_0^1) + t(\Lambda_1^1) + t(\Lambda_0^2)+h_1-n-1}{2}-1,
\end{equation*} 
\item \underline{if $h_1 \not = 0$ and $h_2 \not = n$:} $\mathcal N_{\{0,2\},\{\Lambda_1^0,\Lambda_0^2\}}^{h_1,h_2} \simeq X_{\{\Lambda_1^0,\Lambda_0^2\}}^{h_1,h_2}$ has dimension 
\begin{equation*}
\dim(\mathcal N_{\{0,2\},\{\Lambda_1^0,\Lambda_0^2\}}^{h_1,h_2}) = \frac{t(\Lambda_1^0)+t(\Lambda_0^2)+(h_2-h_1)-n}{2}-1,
\end{equation*}
\item \underline{if $h_1 \not = 0$ and $h_2 \not = n$:} $\mathcal N_{\{0,1,2\},\{\Lambda_1^0,\Lambda_0^1,\Lambda_1^1,\Lambda_0^2\}}^{h_1,h_2} \simeq X_{\{\Lambda_1^0,\Lambda_0^1\}}^{h_1} \times X_{\{\Lambda_1^1,\Lambda_0^2\}}^{h_2}$ has dimension 
\begin{equation*}
\dim(\mathcal N_{\{0,1,2\},\{\Lambda_1^0,\Lambda_0^1,\Lambda_1^1,\Lambda_0^2\}}^{h_1,h_2}) = \frac{t(\Lambda_1^0)+ t(\Lambda_0^1) + t(\Lambda_1^1) +t(\Lambda_0^2)}{2}-n-2.
\end{equation*}
\end{enumerate}

\subsubsection{Iwahori case}

The Iwahori case corresponds to $k$ being maximal, the parity of the components of $\mathbbm h$ being fixed. If $n = 2n'+1$ is odd with $n'\geq 0$, we have $m=n'+1$ and we consider \begin{equation*}
\mathbbm h := (0,2,\ldots,2n').
\end{equation*} 
Another choice would be $\mathbbm h' = (1,3,\ldots,2n'+1)$, but the resulting moduli spaces $\mathcal N_{E/F}^{\mathbbm h}$ and $\mathcal N_{E/F}^{\mathbbm h'}$ are isomorphic so that we only consider $\mathbbm h$.\\
If $n = 2n'$ with $n'\geq 1$, we consider 
\begin{align*}
\mathbbm h^+ := (0,2,\ldots , 2n'), & & \mathbbm h^- := (1,3,\ldots,2n'-1),
\end{align*}
so that we have $m^+ = n'+1$ and $m^- = n'$. In the Iwahori case, we have $\Delta h_i = 1$ for all the possible values of $i$.\\
If $n$ is odd, there are $n'+1$ orbits of irreducible components under the action of $J(E)$. Moreover we have $t_{\mathrm{max}} = n$ and $t_{\mathrm{min}} = 1$. More precisely:

\begin{itemize}
\item irreducible components of the form $\mathcal N_{\{n'+1\},\{\Lambda_0^{n'+1}\}}^{\mathbbm h}$ for some $\Lambda_0^{n'+1} \in \mathcal L_0$ with $t(\Lambda_0^{n'+1}) = n$ make a single $J(E)$-orbit, and the dimension is 
\begin{equation*}
\dim(\mathcal N_{\{n'+1\},\{\Lambda_0^{n'+1}\}}^{\mathbbm h}) = n-1,
\end{equation*}
\item irreducible components of the form $\mathcal N_{\{i\},\{\Lambda_0^i,\Lambda_1^i\}}^{\mathbbm h}$ for some $1 \leq i \leq n'$, $\Lambda_0^i = \pi\Lambda_1^{i\vee} \in \mathcal L_0$ and $t(\Lambda_0^i) = 2i-1$ make a single $J(E)$-orbit, and the dimension is
\begin{equation*}
\dim(\mathcal N_{\{i\},\{\Lambda_0^i,\Lambda_1^i\}}^{\mathbbm h}) = n-2.
\end{equation*}
\end{itemize}

If $n$ is even, with $\mathbbm h^+$ there are $n'$ orbits of irreducible components under the action of $J(E)$. Moreover we have $t_{\mathrm{max}} = n-1$ and $t_{\mathrm{min}} = 1$. More precisely:

\begin{itemize}
\item irreducible components of the form $\mathcal N_{\{i\},\{\Lambda_0^i,\Lambda_1^i\}}^{\mathbbm h^+}$ for some $1 \leq i \leq n'$, $\Lambda_0^i = \pi\Lambda_1^{i\vee} \in \mathcal L_0$ and $t(\Lambda_0^i) = 2i-1$ make a single $J(E)$-orbit, and the dimension is
\begin{equation*}
\dim(\mathcal N_{\{i\},\{\Lambda_0^i,\Lambda_1^i\}}^{\mathbbm h^+}) = n-2.
\end{equation*}
\end{itemize}

In particular, $\mathcal N_{E/F,\mathrm{red}}^{\mathbbm h}$ has pure dimension $n-2$. \\
With $\mathbbm h^-$ there are $n'+1$ orbits of irreducible components under the action of $J(E)$. Moreover we have $t_{\mathrm{max}} = n$ and $t_{\mathrm{min}} = 0$. More precisely:

\begin{itemize}
\item irreducible components of the form $\mathcal N_{\{0\},\{\Lambda_1^{0}\}}^{\mathbbm h^-}$ for some $\Lambda_1^{0} \in \mathcal L_1$ with $t(\Lambda_1^{0}) = n$ make a single $J(E)$-orbit, and the dimension is 
\begin{equation*}
\dim(\mathcal N_{\{0\},\{\Lambda_1^{0}\}}^{\mathbbm h^-}) = n-1,
\end{equation*}
\item irreducible components of the form $\mathcal N_{\{n'\},\{\Lambda_0^{n'}\}}^{\mathbbm h^-}$ for some $\Lambda_0^{n'} \in \mathcal L_0$ with $t(\Lambda_0^{n'}) = n$ make a single $J(E)$-orbit, and the dimension is 
\begin{equation*}
\dim(\mathcal N_{\{n'\},\{\Lambda_0^{n'}\}}^{\mathbbm h^-}) = n-1,
\end{equation*}
\item irreducible components of the form $\mathcal N_{\{i\},\{\Lambda_0^i,\Lambda_1^i\}}^{\mathbbm h^-}$ for some $1 \leq i \leq n'-1$, $\Lambda_0^i = \pi\Lambda_1^{i\vee} \in \mathcal L_0$ and $t(\Lambda_0^i) = 2i$ make a single $J(E)$-orbit, and the dimension is
\begin{equation*}
\dim(\mathcal N_{\{i\},\{\Lambda_0^i,\Lambda_1^i\}}^{\mathbbm h^-}) = n-2.
\end{equation*}
\end{itemize}

\phantomsection
\printbibliography[heading=bibintoc, title={Bibliography}]
\markboth{Bibliography}{Bibliography}

\end{document}